\documentclass[11pt,a4paper]{article}
\usepackage[T1]{fontenc}
\usepackage{amsmath,amssymb}
\usepackage{lmodern}
\usepackage[a4paper,left=25mm,right=25mm,top=25mm,bottom=25mm,footskip=10mm]{geometry}
\usepackage{microtype}
\usepackage{titling}
\usepackage{authblk}
\usepackage[authoryear,longnamesfirst]{natbib}
\usepackage{amsthm}
\usepackage{mathtools}
\usepackage{needspace}
\usepackage{xcolor}
\usepackage{hyperref}
\hypersetup{colorlinks=true,citecolor=blue,urlcolor=blue,linkcolor=blue,
  pdftitle={Refined Cram\'er-type moderate deviations by Stein's method under dependence},
  pdfauthor={Hui-Lin Liu, Song-Hao Liu, Dawei Lu}}
\predate{}
\postdate{}
\AtBeginDocument{%
  \setlength{\abovedisplayskip}{7.5pt plus 2pt minus 2pt}%
  \setlength{\belowdisplayskip}{7.5pt plus 2pt minus 2pt}%
  \setlength{\abovedisplayshortskip}{4pt plus 2pt}%
  \setlength{\belowdisplayshortskip}{6pt plus 2pt minus 2pt}%
}
\allowdisplaybreaks[0]
\DeclareMathOperator{\Var}{\mathbb{V}\mspace{-1mu}ar}

\theoremstyle{plain}

\newtheorem{theorem}{Theorem}[section]
\newtheorem{proposition}{Proposition}[section]
\newtheorem{lemma}[theorem]{Lemma}

\theoremstyle{definition}

\newtheorem*{assumption}{Assumption}

\newtheorem{corollary}{Corollary}[section]
\newtheorem{remark}{Remark}[section]
\newtheorem*{notation}{Notation}
\newtheorem*{organization}{Organization}

\numberwithin{equation}{section}

\newcommand{\E}{\mathbb{E}}
\newcommand{\I}{\mathbb{I}}
\newcommand{\p}{\mathbb{P}}

\begin{document}
\title{\LARGE\bfseries Refined Cram\'er-type moderate deviations\\
by Stein's method under dependence}
\renewcommand{\Authfont}{\large}
\renewcommand{\Affilfont}{\small}
\setlength{\affilsep}{0.5em}
\author{Hui-Lin Liu\thanks{Email: \href{mailto:liuhuilin@mail.dlut.edu.cn}{\nolinkurl{liuhuilin@mail.dlut.edu.cn}}.}}
\author{Song-Hao Liu\thanks{Email: \href{mailto:liusonghao@dlut.edu.cn}{\nolinkurl{liusonghao@dlut.edu.cn}}.}}
\author{Dawei Lu\thanks{Corresponding author. Email: \href{mailto:ludawei_dlut@163.com}{\nolinkurl{ludawei_dlut@163.com}}.}}
\affil{School of Mathematical Sciences, Dalian University of Technology\\Dalian, Liaoning, China}
\date{}
\maketitle

\begin{abstract}

\noindent In this paper, we establish refined Cram\'er-type moderate deviations for dependent random variables at any fixed order $m\geq3$. Combining higher-order Stein identities with exponential tilting, we obtain cumulant-corrected tail approximations with explicit error bounds for boundedly supported kernels. Our main application gives a fourth-order expansion under bounded local dependence, reaching an $O(n^{3/10})$ deviation range under standard normalization with fixed dependence parameters. It improves on previous work in both deviation range and applicability, and also covers bounded $m$-dependent sequences without an independent-base representation. Further applications to the generalized combinatorial central limit theorem and the Curie--Weiss model also extend some earlier results under permutation and mean-field dependence.

\end{abstract}


	\section{Introduction}

	Let $W$ be a random variable of interest, typically arising from a dependent system. Our aim is to approximate $\p(W>x)$ with cumulant corrections and explicit relative-error bounds. For standardized sums of i.i.d.\ random variables, such expansions go back to \citet{C1938}. Let $X_1,\ldots,X_n$ be i.i.d.\ random variables such that $\E X_1=0$, $\E X_1^2=1$ and $\E e^{t_0|X_1|}<\infty$ for some $t_0>0$. Set $W_n=n^{-1/2}\sum_{i=1}^nX_i$, and let $\Phi(x)$ denote the standard normal distribution function. The classical Cram\'er expansion takes the form
	\begin{equation}\label{intro-cramer-expansion}
		\frac{\p(W_n>x)}{1-\Phi(x)}
		=\exp\left\{\frac{x^3}{\sqrt n}
		\lambda\left(\frac{x}{\sqrt n}\right)\right\}
		\left[1+O\left(\frac{1+x}{\sqrt n}\right)\right],
	\end{equation}
	uniformly for $0\leq x=o(\sqrt n)$ (see \citet[Chapter~8, Theorem~1]{PETROV1975}). Here $\lambda(t)$ is the Cram\'er series, which converges near zero and satisfies $\lambda(0)=\E X_1^3/6$. The successive truncations of this series describe how cumulant corrections refine the normal-tail approximation. Omitting the exponential correction gives, for any fixed $C>0$,
	\begin{equation}\label{intro-normal-tail-ratio}
		\frac{\p(W_n>x)}{1-\Phi(x)}
		=1+O\left(\frac{1+x^3}{\sqrt n}\right),
		\qquad 0\leq x\leq Cn^{1/6}.
	\end{equation}
	This relative-error estimate remains informative for growing $x$, when an absolute normal-approximation bound may be too large compared with the tail probability. For related Cram\'er-type moderate deviation results under dependence, including dependent sums and martingales, see, for example, \citet{RAIC2007,FGL2013,FANSHAO2024,FK2023}. Retaining the first term of the Cram\'er series gives the skewness correction. With $\Gamma_{3,n}=\E W_n^3=n^{-1/2}\E X_1^3$, the third-order refinement is
	\begin{equation}\label{intro-third-cumulant-tail-ratio}
		\frac{\p(W_n>x)}{1-\Phi(x)}
		=e^{\Gamma_{3,n}x^3/6}
		\left[1+O\left\{\frac{1+x}{\sqrt n}+\frac{x^4}{n}\right\}\right],
		\qquad 0\leq x\leq Cn^{1/4}.
	\end{equation}
	Thus the third-cumulant correction extends the classical $n^{1/6}$ scale of \eqref{intro-normal-tail-ratio} to $n^{1/4}$. This improvement reflects the effect of skewness on the tail. Although $\Gamma_{3,n}$ is small, its product with $x^3$ need not be small when $x$ grows with $n$. Keeping this term in the exponential accounts for the leading difference from the normal tail instead of including it in the error. Higher-order refinements retain further terms of the Cram\'er series and require control of the remainder over a larger deviation range.

	Cumulant methods give a general approach to these expansions. \citet{DJS2022} survey the cumulant approach to normal approximation and Cram\'er--Petrov expansions, including cumulant bounds under dependence. Refined results are also available in specific settings. \citet{GSS2022} studied a refined approximation for self-normalized sums. For bounded local statistics of independent variables, \citet{FLS2020} obtained a third-order approximation by combining a standardized Poisson approximation with spline functions. Their Remark~2.1 identifies the restricted range of the exponential concentration inequality as an obstacle to higher-order expansions. For dependent variables, going beyond skewness correction requires more than retaining another term in a formal expansion. The moment generating function generally does not factorize, and the remainder must be controlled throughout the larger deviation range. In particular, computing a correction does not by itself provide the concentration inequalities and normal-approximation bounds needed to justify it. Our aim is to obtain these estimates together from a Stein identity, so that the same general theorem can be applied to different dependence structures.

	We combine higher-order Stein identities with exponential tilting and the normal Stein equation. The overall proof follows the classical exponential-tilting approach to Cram\'er-type deviations (see \citealp{C1938,PETROV1975}; \citealp[Section~6]{DJS2022}). After controlling the exponential moments, we introduce a tilted measure and establish a normal approximation under that measure. A change of measure then recovers the original tail probability. Stein's approach to normal approximation began with \citet{STEIN1972} and was developed systematically in \citet{STEIN1986}. It has proved useful for distributional approximation under a wide range of dependence structures \citep[see, e.g.,][]{STEIN1986,CS2004,SZ2019,LZ2023}. For uncorrected Cram\'er-type estimates, \citet{CFS2013AOP} used Stein identities with boundedly supported random measures. \citet{Z2023} treated unbounded exchangeable pairs, and \citet{LZ2023} removed the bounded-support assumption for general Stein identities. Here we use a higher-order identity to retain the cumulant corrections while using the usual bounds for the solution of the normal Stein equation.

	The main difficulty is to use these bounds under the tilted measure without losing the accuracy required by the higher-order corrections. For a fixed order $m\geq3$, we use the cumulant polynomial $\Gamma(u)=u^2/2+\sum_{j=3}^{m}\Gamma_ju^j/j!$, where $\Gamma_j$ is the $j$th cumulant of a standardized random variable $W$. We choose the tilt parameter $\lambda$ through $\Gamma'(\lambda)=x$. Under a tilt with parameter $t$, both the location and the fluctuations of the random variable change. We therefore conjugate the higher-order Stein identity and recenter it at $\Gamma'(t)$ (see \eqref{mth-gamma-conjugate-stein}). Collecting the first-derivative terms also identifies $\Gamma''(t)$ as the variance parameter of the Gaussian reference in \eqref{mth-BE-stein-equation}. This puts the deterministic center and variance corrections into the Stein operator, rather than treating them as errors. The remaining terms contain powers of the tilt parameter, higher derivatives and exponentially weighted kernels. Centered exponential-moment estimates and small shifts of the tilt parameter control these terms. We estimate scalar and kernel remainders separately.

	A further difficulty arises when we smooth the indicator in the Stein equation. Estimating the higher derivatives of its solution requires bounds on short-interval probabilities under the tilted measure. We adapt the recursive concentration argument of \citet[Proposition~3.1]{CS2004} to obtain the tilted concentration inequality \eqref{mth-tilted-concentration} in Proposition \ref{prop-mth-tilted-concentration}. Choosing the smoothing width according to the kernel and cumulant scales allows the recursive term to be absorbed. The recursive argument itself does not depend on higher-order cumulant corrections. It relies on a suitable Stein identity and estimates for the associated moments and remainder terms. Finally, recovering a relative tail estimate increases the effect of the tilted normal-approximation error because the Gaussian weighted tail integral is of order $(1+\lambda)^{-1}$. We therefore keep the dependence on the tilt parameter explicit in the normal-approximation, exponential-moment and variance estimates.

	Our main result, Theorem \ref{theorem22}, gives a refined Cram\'er-type moderate deviation expansion at any prescribed fixed order $m\geq3$, under quantitative conditions on boundedly supported kernels and remainder coefficients. It gives a saddlepoint approximation with explicit error bounds. The assumptions are stated directly in terms of the Stein identity rather than a particular dependence structure.

	Our principal application, Theorem \ref{theorem31}, is a fourth-order expansion for bounded locally dependent sums. For standardized $W$, with $\Gamma_3=\E W^3$ and $\Gamma_4=\E W^4-3$, the correction factor is
	\begin{equation}\label{intro-fourth-cumulant-factor}
		\exp\left\{\frac{\Gamma_3x^3}{6}
		+\frac{\Gamma_4-3\Gamma_3^2}{24}x^4\right\}.
	\end{equation}
	The local dependence conditions include the bounded local-statistics setting of \citet{FLS2020}. With fixed dependence parameters, $n^{-1/2}$-scaled summands and a number of independent base variables comparable to $n$, the fourth-order result extends their third-order deviation scale from $n^{1/4}$ to $n^{3/10}$. Section \ref{sec3} gives the precise assumptions and parameter-dependent bounds. For bounded $m$-dependent sequences, Corollary \ref{cor-bounded-m-dependent} gives an explicit bound in which the dependence range may vary with $n$, without requiring an independent-base representation. The fourth-order term also matters when the skewness vanishes. As a simple classical illustration, the standardized sum of $n$ independent symmetric random signs has $\Gamma_3=0$ and $\Gamma_4=-2/n$. Its fourth-order factor in \eqref{intro-fourth-cumulant-factor} is $\exp\{-x^4/(12n)\}$, which is already nontrivial on the $n^{1/4}$ scale despite the absence of a skewness correction.

	The applications also include two models with global dependence: the generalized combinatorial central limit theorem and the Curie--Weiss model. In the former, a random permutation imposes a global matching constraint, while in the latter, the mean-field interactions are described by an exponential-family Gibbs distribution. Exchangeable-pair Stein identities give a third-order expansion for the generalized combinatorial central limit theorem and a fourth-order expansion for the Curie--Weiss model. Together with the local-dependence result, these applications show how the same method applies to different forms of dependence.

	The proof also gives a concentration inequality under the original measure with an exponentially decaying factor (see Remark \ref{rem-original-exponential-concentration}). This inequality may be of independent interest beyond the moderate-deviation expansion.

	\phantomsection\label{notation-conventions}
	\begin{notation}
		Throughout the paper, $C$ denotes a generic positive constant that may change from line to line.  We write $a=O(b)$ if $|a|\leq C|b|$ and $a\asymp b$ if both $a=O(b)$ and $b=O(a)$. For a set $A$, $\#(A)$ denotes its cardinality, and $\I(\cdot)$ denotes the indicator function.  The auxiliary random variable $U\sim\mathrm{Unif}(0,1)$ is independent of the underlying random elements, but derived random functions may depend on $U$.
	\end{notation}

	\begin{organization}
	The rest of the paper is organized as follows. Section \ref{sec2} states the fixed-order moderate deviation theorem, and Section \ref{sec3} presents its applications. Section \ref{sec4} states the auxiliary estimates and uses them to prove the main theorem. Section \ref{sec5} first illustrates the construction of a third-order Stein identity and then proves the application results. Section \ref{sec-auxiliary-proofs} collects the proofs of the auxiliary estimates from Section \ref{sec4}. Appendices \ref{appB}--\ref{appD} provide the model-specific Stein identities and remainder estimates used in Section \ref{sec5}.
	\end{organization}

	\section{Main results}\label{sec2}\label{subsec-mth-order}
	Our main result gives a saddlepoint approximation based on an $m$th-order Stein identity. We state the assumptions below, keeping separate the quantities that control the exponential correction and the relative error.

	For a random variable $W$ and $t\in\mathbb R$ such that $\E e^{tW}<\infty$, let $\p_t$ denote the exponentially tilted probability measure and $\E_t$ the corresponding expectation. Define

	\begin{equation}\label{intro-tilted-definitions}
		M(t)\coloneqq\E e^{tW},\qquad \psi(t)\coloneqq\log M(t),\qquad
		\frac{d\p_t}{d\p}\coloneqq e^{tW-\psi(t)},\qquad
		\E_tY\coloneqq\E\{Ye^{tW-\psi(t)}\}.
	\end{equation}

	Let $m\geq3$ be fixed and let $W$ satisfy $\E W=0$ and $\E W^2=1$. Throughout this section, sums over empty index sets are understood to be zero, and $0^{-a}=+\infty$ for $a>0$. Define $\Gamma_j\coloneqq\operatorname{cum}_j(W)$ for $3\leq j\leq m$ and, for $u\in\mathbb R$,
	\begin{equation}\label{intro-cumulant-polynomial}
		\Gamma(u)\coloneqq\frac{u^2}{2}
		+\sum_{j=3}^{m}\frac{\Gamma_ju^j}{j!}.
	\end{equation}
	In this section, $C_m$ denotes a positive constant depending only on $m$ and may change from line to line.  We write $a=O_m(b)$ if $|a|\leq C_m|b|$.

	We first assume the following $m$th-order Stein identity.

	\begin{assumption}[{\rm A1}]
		\makeatletter
		\phantomsection\def\@currentlabel{A1}\label{A1}
		\makeatother
		There exist a constant $\rho_m>0$, random variables $\mathcal R_{m,1},\ldots,\mathcal R_{m,m}$, and random functions $\mathcal K_{m,1}(t),\mathcal K_{m,2}(t)$ supported on $[-\rho_m,\rho_m]$ such that the following identity holds for every $(m-1)$ times differentiable function $f(w)$ whose derivative $f^{(m-1)}(w)$ is absolutely continuous, whenever the expectations below are finite:
		\begin{align}\label{mth-Stein}
			\E\{Wf(W)\}
			&=\E f'(W)+\sum_{j=3}^{m}\frac{\Gamma_j}{(j-1)!}\E f^{(j-1)}(W)
			+\sum_{j=1}^{m}\E\{\mathcal R_{m,j}f^{(j-1)}(W)\}\\*
			&\quad+\E\int_{|t|\leq\rho_m}f^{(m)}(W+t)\mathcal K_{m,1}(t)dt
			+\E\int_{|t|\leq\rho_m}f^{(m)}(W+Ut)\mathcal K_{m,2}(t)dt. \notag
		\end{align}
	\end{assumption}
	\noindent Here and below, we define $f^{(0)}(w)\coloneqq f(w)$.

	We next control the random remainder coefficients in \eqref{mth-Stein} under exponential tilting.

	\begin{assumption}[{\rm A2}]
		\makeatletter
		\phantomsection\def\@currentlabel{A2}\label{A2}
		\makeatother
		There exist a constant $T_m\geq1$ and nonnegative even functions $r_{m,j}(t)$, $1\leq j\leq m$, increasing on $[0,T_m]$ such that, for all $|t|\leq T_m$,
		\begin{equation}\label{A2-moment}
			\E_t|\mathcal R_{m,j}|\leq r_{m,j}(t),\qquad 1\leq j\leq m.
		\end{equation}
	\end{assumption}

	The two remainder kernels are controlled by the following conditional bound.

	\begin{assumption}[{\rm A3}]
		\makeatletter
		\phantomsection\def\@currentlabel{A3}\label{A3}
		\makeatother
		There exists a nonnegative constant $\vartheta_m$ such that, for all $|t|\leq T_m$,
		\begin{equation}\label{A3-kernel}
			\E\biggl\{\int_{|u|\leq\rho_m}e^{|tu|}\bigl[|\mathcal K_{m,1}(u)|+|\mathcal K_{m,2}(u)|\bigr]du\,\bigg|\,W\biggr\}\leq \vartheta_m.
		\end{equation}
	\end{assumption}

	Finally, we impose deterministic bounds on the cumulant coefficients.

	\begin{assumption}[{\rm A4}]
		\makeatletter
		\phantomsection\def\@currentlabel{A4}\label{A4}
		\makeatother
		Let $\Gamma_{k,0}>0$, $3\leq k\leq m$, be constants such that $|\Gamma_k|\leq \Gamma_{k,0}$ for $3\leq k\leq m$.
	\end{assumption}

	The proof requires bounds on both the logarithmic moment generating function and the normal approximation under tilting. To state these bounds, define, for $u\in\mathbb R$,
	\[
		\Gamma_0(u)\coloneqq\frac{u^2}{2}
		+\sum_{j=3}^{m}\frac{\Gamma_{j,0}u^j}{j!}.
	\]
	Define
	\begin{align}
		\kappa_m(x)
		&\coloneqq \sum_{j=1}^{m}x^jr_{m,j}(x)+x^{m+1}\vartheta_m,\label{mth-kappa}\displaybreak[1]\\
		\eta_m(x)
		&\coloneqq\Gamma_0''(x)-1
		=\sum_{j=3}^{m}\frac{\Gamma_{j,0}x^{j-2}}{(j-2)!}.\label{mth-eta}
	\end{align}
	The quantity $\kappa_m(x)$ controls $|\psi(x)-\Gamma(x)|$, while $\eta_m(x)$ bounds the deviation of $\Gamma''(x)$ from $1$. We also set
	\begin{align}\label{mth-Delta-tilde-definition}
		\tilde{\Delta}_m(t)
		\coloneqq{}&
		\biggl[
		\rho_m+\max_{3\leq j\leq m}\Gamma_{j,0}^{1/(j-2)}
		+\vartheta_m^{1/(m-1)}
		\biggr]\biggl[
		e^{2\{\eta_m(t+1)+\kappa_m(t+1)\}}
		+\sum_{j=1}^{m}t^{j-1}r_{m,j}(t)\biggr]\notag\\
		&+\sum_{j=1}^{m}r_{m,j}(t)\sum_{\ell=1}^{j-1}t^{j-1-\ell}
		\biggl[
		\rho_m+\max_{3\leq k\leq m}\Gamma_{k,0}^{1/(k-2)}
		+\vartheta_m^{1/(m-1)}
		\biggr]^{1-\ell},
	\end{align}
	and
	\begin{align}\label{mth-Delta-definition}
		\Delta_m(t)
		\coloneqq{}&
		\tilde{\Delta}_m(t)
		\biggl[
		e^{2\kappa_m(t+1)}
		+\sum_{j=3}^{m}\Gamma_{j,0}
		\sum_{\ell=2}^{j-1}\bigl(t^{j-1-\ell}+1\bigr)\rho_m^{1-\ell}
		+\vartheta_m\sum_{\ell=2}^{m}\bigl(t^{m-\ell}+1\bigr)\rho_m^{1-\ell}
		\biggr]\notag\\
		&+e^{2\kappa_m(t+1)}
		\biggl\{
		\sum_{j=1}^{m}r_{m,j}(t+1)
		\biggl[
		t^{j-1}+1
		+\sum_{r=1}^{j-2}
		\rho_m^{-r}\bigl(t^{j-2-r}+1\bigr)
		\biggr]
		+\vartheta_m\bigl(t^m+1\bigr)
		\biggr\}.
	\end{align}
	We use $\tilde\Delta_m(t)$ to bound interval probabilities under tilting and $\Delta_m(t)$ for the subsequent Berry--Esseen estimate. With this notation, the fixed-order expansion takes the following form.

	\begin{theorem}\label{theorem22}
		Assume {\rm (A1)}--{\rm (A4)}.  Use the notation in \eqref{intro-cumulant-polynomial} and \eqref{intro-tilted-definitions}. For every $x$ satisfying
		\begin{equation}\label{mth-range}
			\begin{aligned}
				0\leq x\leq \frac18\biggl[\min\Bigl\{
				T_m,
				\rho_m^{-(m-1)/(m+1)},\vartheta_m^{-1/(m+1)},
				\min_{3\leq k\leq m}
				\Gamma_{k,0}^{-(m-1)/[(k-2)(m+1)]}
				\Bigr\}-1\biggr],
			\end{aligned}
		\end{equation}
		the saddlepoint equation
		\begin{equation}\label{mth-lambda-definition}
			x=\Gamma'(\lambda)
		\end{equation}
		has a unique solution $\lambda \coloneqq\lambda(x)$ in $[0,2x]$. Then, uniformly for $x$ satisfying \eqref{mth-range},
		\begin{equation}\label{mth-explicit-result}
			\begin{aligned}
			\frac{\p(W>x)}{1-\Phi(x)}
			&=\exp\biggl\{\Gamma(\lambda)-\lambda\Gamma'(\lambda)+\frac{\lambda^2}{2}
			+O_m\{\kappa_m(\lambda)\}\biggr\}\frac{1-\Phi(\lambda)}{1-\Phi(x)}\\
			&\quad\times
			\left[1+O_m\{\eta_m(\lambda)+\kappa_m(\lambda)+(1+\lambda)\Delta_m(\lambda)\}\right].
			\end{aligned}
		\end{equation}
	\end{theorem}

	\begin{remark}\label{rem-iid-main-theorem}
		We apply our result to the i.i.d.\ case to illustrate that the error bound in Theorem~\ref{theorem22} has the optimal order in general. Let $Y_1,\ldots,Y_n$ be i.i.d.\ with $\E Y_1=0$, $\E Y_1^2=1$ and $|Y_1|\leq B$, and set $W_n=n^{-1/2}\sum_{i=1}^nY_i$. Then $\Gamma_j=n^{-(j-2)/2}\operatorname{cum}_j(Y_1)$. Taylor's expansion at $W_n-Y_i/\sqrt n$, using independence and the moment--cumulant relations, gives {\rm (A1)} with $\mathcal R_{m,j}=0$ and $\mathcal K_{m,2}(u)=0$. The remaining kernel is supported on $[-B/\sqrt n,B/\sqrt n]$ and its $L^1$ norm is bounded by $C_{m,B}n^{-(m-1)/2}$ pointwise. Thus {\rm (A2)}--{\rm (A4)} hold with $T_m=\sqrt n$, $\rho_m=B/\sqrt n$, $r_{m,j}(t)=0$, $\vartheta_m=C_{m,B}n^{-(m-1)/2}$ and $\Gamma_{j,0}=C_{j,B}n^{-(j-2)/2}$. Let $\lambda(t)=\sum_{j\geq0}a_jt^j$ be the Cram\'er series associated with $Y_1$ in \eqref{intro-cramer-expansion}. Substituting these bounds into Theorem~\ref{theorem22} and expanding its saddlepoint expression gives
		\[
			\frac{\p(W_n>x)}{1-\Phi(x)}
			=\exp\left\{\sum_{j=0}^{m-3}
			\frac{a_jx^{j+3}}{n^{(j+1)/2}}\right\}
			\left[1+O_{m,B}\left(\frac{1+x}{\sqrt n}
			+\frac{x^{m+1}}{n^{(m-1)/2}}\right)\right],
		\]
		uniformly for $0\leq x\leq c_{m,B}n^{(m-1)/(2m+2)}$ and sufficiently large $n$. This matches the classical error order obtained by truncating \eqref{intro-cramer-expansion}. When $a_{m-2}\ne0$, the first omitted term has order $x^{m+1}n^{-(m-1)/2}$. Retaining terms only up to order three recovers \eqref{intro-third-cumulant-tail-ratio}.
	\end{remark}

	\begin{remark}
		The proof strategy for Theorem \ref{theorem22} also applies when $m=2$, with minor adjustments. In particular, the recursive concentration argument in \eqref{mth-kernel-first-derivative-bound}, which uses $\vartheta_m^{2/(m+1)}\leq\vartheta_m^{1/(m-1)}$, needs to be modified. We also need to adjust the formula for $\tilde{\Delta}_2(t)$ and the conventions for sums over empty index sets. Since this case is not the main focus of the paper, we omit the detailed proof.
	\end{remark}

	\begin{remark}\label{rem-unbounded-support-future}
		Our result can be extended to kernels with unbounded support, but this requires additional assumptions on the kernels. We leave a detailed treatment of this extension to future work.
	\end{remark}

	An identity of the form {\rm (A1)} can be obtained by expanding a classical Stein identity. We give a general third-order construction in Section \ref{subsubsec-third-order-Stein-derivation}. The resulting auxiliary identities will be used in the exchangeable-pair applications below.

\section{Applications}\label{sec3}
	We now apply Theorem \ref{theorem22} to three settings: local dependence, the generalized combinatorial central limit theorem, and the Curie--Weiss model. In each case, we construct a suitable Stein identity and estimate its remainders. For bounded locally dependent sums, nested neighborhoods give a fourth-order expansion that retains the third and fourth cumulants. Exchangeable pairs give a third-order expansion for the generalized combinatorial central limit theorem and a fourth-order expansion for the Curie--Weiss model.

\subsection{Sums of locally dependent random variables}\label{subsec-local-dependence}

Let $\{X_i:i\in\mathcal J\}$ be centered random variables, where $\mathcal J$ is finite, and write $X_A=\{X_i:i\in A\}$ for $A\subset\mathcal J$. Set $W=\sum_{i\in\mathcal J}X_i$, assume $\Var(W)=1$, and let $n=\#(\mathcal J)$. The dependence conditions below allow the neighborhood to grow as more indices enter the expansion.

\medskip \noindent {\rm (LD1)} For each $i\in\mathcal J$, there exists $A_i\subset\mathcal J$ such that $X_i$ and $X_{A_i^c}$ are independent.

\smallskip \noindent {\rm (LD$^*$2)} For $i\in\mathcal J$ and $j\in A_i$, there exists $A_{ij}\supset A_i$ such that $\{X_i,X_j\}$ and $X_{A_{ij}^c}$ are independent.

\smallskip \noindent {\rm (LD$^*$3)} For $i\in\mathcal J$, $j\in A_i$ and $k\in A_{ij}$, there exists $A_{ijk}\supset A_{ij}$ such that $\{X_i,X_j,X_k\}$ and $X_{A_{ijk}^c}$ are independent.

\smallskip \noindent {\rm (LD$^*$4)} For $i\in\mathcal J$, $j\in A_i$, $k\in A_{ij}$ and $l\in A_{ijk}$, there exists $A_{ijkl}\supset A_{ijk}$ such that $\{X_i,X_j,X_k,X_l\}$ and $X_{A_{ijkl}^c}$ are independent.
\medskip

These indexed neighborhoods are also used in refined normal approximation (see \citet{CSZ2026}). Their sizes determine the dependence factors in our error bound. Define a common size bound by
\begin{equation}\label{local-neighborhood-size}
 \kappa\coloneqq
 \max_{\substack{i\in\mathcal J,\ j\in A_i\\k\in A_{ij},\ l\in A_{ijk}}}
 \#(A_{ijkl}).
\end{equation}
By nesting, $\kappa$ also bounds $\#(A_i)$, $\#(A_{ij})$ and $\#(A_{ijk})$. The fourth-order bound below expresses the effect of dependence through this common neighborhood size.

\begin{theorem}\label{theorem31}
Assume {\rm (LD1)} and {\rm (LD$^*$2)}--{\rm (LD$^*$4)}, and let $|X_i|\leq\delta$ for every $i\in\mathcal J$. Write $\Gamma_3=\E W^3$ and $\Gamma_4=\E W^4-3$. There exists an absolute constant $c>0$ such that, uniformly for
\[
 0\leq x\leq c\min\bigl\{(n\kappa^4\delta^5)^{-1/5},
 (n\kappa^3\delta^4)^{-3/10},
 (n\kappa^2\delta^3)^{-3/5}\bigr\},
\]
we have
\begin{equation}\label{bounded-local-fourth-order-result}
 \begin{aligned}
 \frac{\p(W>x)}{1-\Phi(x)}
 &=\exp\left\{\frac{\Gamma_3x^3}{6}
 +\frac{\Gamma_4-3\Gamma_3^2}{24}x^4\right\}\\
 &\quad\times
 \biggl[1+O\biggl\{
 n\kappa^2\delta^3(1+x)+n\kappa^3\delta^4x^3
 +\bigl(n\kappa^4\delta^5+n^3\kappa^6\delta^9\bigr)x^5
 \biggr\}\biggr].
 \end{aligned}
\end{equation}
\end{theorem}

\begin{remark}\label{rem-local-higher-order}
Theorem~\ref{theorem31} gives the fourth-order case, but the same construction extends to any fixed order under the corresponding higher-order neighborhood conditions. We obtain the higher-order Stein identities by repeating the centering steps and applications of Taylor's expansion in the proof of Lemma~\ref{lemma-local-fourth-stein}. Since the summands are bounded, the kernel supports, their $L^1$ norms and the cumulants can be controlled directly. The conditions of Theorem~\ref{theorem22} are then straightforward to verify.
\end{remark}

\begin{remark}\label{rem-local-statistics}
Local statistics form a special case of local dependence (see also \citet[Remark~2.3]{LZ2023}). Let $(\xi_\alpha)_{\alpha\in\mathcal I}$ be independent and write $X_i=f_i(\xi_{I_i})$, where $I_i\subset\mathcal I$ and $\xi_{I_i}=(\xi_\alpha)_{\alpha\in I_i}$. We may take $A_i=\{j:I_i\cap I_j\neq\varnothing\}$ and $A_{i_1\cdots i_k}=\bigcup_{\ell=1}^{k}A_{i_\ell}$. The selected vector and the variables outside its neighborhood depend on disjoint sets of independent base variables and are therefore independent. Hence the neighborhood conditions hold at every order without any further independence assumptions. If $\#I_i\leq s$ and each base variable occurs in at most $d$ summands, then $\#A_{i_1\cdots i_k}\leq ksd$. In particular, $\kappa\leq4sd$, while an expansion of any fixed order $r$ has neighborhood sizes bounded by $rsd$. Both $s$ and $d$ may depend on $n$. Theorem~\ref{theorem31} therefore covers the bounded local-statistics setting of \citet[Theorem~2.1]{FLS2020} and extends their expansion to fourth order. When the number of base variables is of order $n$, $s,d=O(1)$ and $\delta\asymp n^{-1/2}$, their third-order estimate has scale $n^{1/4}$, whereas \eqref{bounded-local-fourth-order-result} gives the scale $n^{3/10}$.
\end{remark}

We next apply Theorem \ref{theorem31} to $m$-dependent sequences. No representation in terms of independent base variables is needed here (see also the discussion in \citet[Section~2.2]{FLS2020}). With a slight abuse of notation, we use $m$ for the dependence range in the following corollary and its discussion. Recall that a sequence is $m$-dependent if $(X_1,\ldots,X_j)$ and $(X_{j+m+1},\ldots,X_n)$ are independent whenever both blocks are nonempty.

\begin{corollary}[$m$-dependent sequences]\label{cor-bounded-m-dependent}
Let $m\geq0$ be an integer, and let $X_1,\ldots,X_n$ be centered, $m$-dependent random variables such that $|X_i|\leq\delta$ and $\Var(W)=1$, where $W=\sum_{i=1}^nX_i$. Here $m$ may depend on $n$. Then the conclusion of Theorem \ref{theorem31} holds with $\kappa=5m+1$. In particular, if $\delta\leq Bn^{-1/2}$ for a constant $B>0$, then
\begin{equation}\label{m-dependent-fourth-order-result}
 \begin{aligned}
 \frac{\p(W>x)}{1-\Phi(x)}
 &=\exp\left\{\frac{\Gamma_3x^3}{6}
 +\frac{\Gamma_4-3\Gamma_3^2}{24}x^4\right\}\\
 &\quad\times
 \left[1+O_B\left\{
 \frac{(m+1)^2(1+x)}{\sqrt n}
 +\frac{(m+1)^3x^3}{n}
 +\frac{(m+1)^6x^5}{n^{3/2}}\right\}\right],
 \end{aligned}
\end{equation}
uniformly for $0\leq x\leq c_B n^{3/10}(m+1)^{-6/5}$, where $\Gamma_3=\E W^3$, $\Gamma_4=\E W^4-3$, and $c_B>0$ depends only on $B$.
\end{corollary}

\begin{proof}[Proof of Corollary \ref{cor-bounded-m-dependent}]
Write $[n]=\{1,\ldots,n\}$. For each admissible tuple of indices, take
\[
 A_{i_1\cdots i_r}
 =\bigl[\min_{1\leq s\leq r}i_s-m,\,
         \max_{1\leq s\leq r}i_s+m\bigr]\cap[n],
 \qquad 1\leq r\leq4.
\]
The block from $\min_s i_s$ to $\max_s i_s$ is separated by more than $m$ from each exterior block in $A_{i_1\cdots i_r}^c$. By $m$-dependence, the left block is independent of the joint vector formed by the central and right blocks, and the latter two blocks are independent. Thus the selected vector is independent of $X_{A_{i_1\cdots i_r}^c}$. The neighborhoods are nested, and each additional selected index increases their span by at most $m$. Thus $\#A_{i_1\cdots i_r}\leq(r+1)m+1$, which verifies {\rm (LD1)} and {\rm (LD$^*$2)}--{\rm (LD$^*$4)} with $\kappa\leq5m+1$. Theorem \ref{theorem31} now gives the result.
\end{proof}

For fixed $m$, Corollary~\ref{cor-bounded-m-dependent} retains the classical fourth-order scale $n^{3/10}$. The same interval neighborhoods satisfy the higher-order local dependence conditions for any fixed tuple length, so the construction also extends to higher orders when $m$ varies with $n$.

	For the next two applications, we construct the Stein identities from exchangeable pairs with bounded increments. A permutation transposition provides the pair for the generalized combinatorial central limit theorem, and a single-spin update does so for the Curie--Weiss model.

	\subsection{Generalized combinatorial central limit theorem}
	Let $n\geq3$, and let $X=\{X_{i,j}:1\leq i,j\leq n\}$ be an array of independent random variables.  Define $a_{i,j}\coloneqq\E X_{i,j}$ and $c_{i,j}^{2}\coloneqq\Var(X_{i,j})$. Assume that the deterministic mean array is centered in every row and column:
	\begin{equation}\label{cclt-row-column}
		\sum_{j=1}^{n}a_{i,j}=0,\quad 1\leq i\leq n,
		\qquad
		\sum_{i=1}^{n}a_{i,j}=0,\quad 1\leq j\leq n,
	\end{equation}
	and that the variance is normalized by
	\begin{equation}\label{cclt-normalization}
		\frac{1}{n-1}\sum_{i=1}^{n}\sum_{j=1}^{n}a_{i,j}^{2}
		+\frac{1}{n}\sum_{i=1}^{n}\sum_{j=1}^{n}c_{i,j}^{2}=1.
	\end{equation}
	Let $\pi$ be a uniform random permutation of $\{1,\ldots,n\}$, independent of $X$, and let
	\begin{equation}\label{cclt-W}
		W=\sum_{i=1}^{n}X_{i,\pi(i)}.
	\end{equation}
	Then \eqref{cclt-row-column} and \eqref{cclt-normalization} imply $\E W=0$ and $\E W^2=1$.  We shall work under the boundedness condition
	\begin{equation}\label{cclt-bounded}
		|X_{i,j}|\leq \delta/12,\qquad 1\leq i,j\leq n.
	\end{equation}

	The classical combinatorial central limit theorem originates with \citet{H1951} for deterministic permutation sums.  In that setting, \citet{G2005} obtained a Berry--Esseen bound by Stein's method and zero-bias coupling, and \citet{CFS2013AOP} proved a Cram\'er-type moderate deviation estimate.  When the array elements are random, Berry--Esseen bounds were established by \citet{HC1978} and \citet{CF2015}, while \citet{LZ2023} obtained Cram\'er-type moderate deviations for unbounded entries. The theorem below retains the third-cumulant correction for bounded entries. Here the exchangeable-pair regression has an additional centered full-array remainder, which is absent for deterministic arrays and requires a separate estimate.

	\begin{theorem}\label{theorem33}
		Assume \eqref{cclt-row-column}--\eqref{cclt-bounded}, and let $\Gamma_3=\E W^3$.  Then there exists an absolute constant $c>0$ such that, uniformly for
		\begin{equation}\label{cclt-proof-range}
			0\leq x\leq c\min\{(n\delta^3)^{-1/2},\,\delta^{-1/2}\},
		\end{equation}
		we have
		\begin{equation}\label{cclt-third-order-result}
			\frac{\p(W>x)}{1-\Phi(x)}
			=e^{\Gamma_3 x^3/6}\left[1+O\{n\delta^3(1+x^2)\}\right].
		\end{equation}
	\end{theorem}

	\begin{remark}
		Compared with the second-order result for unbounded array elements in Theorem 3.1 of \citet{LZ2023}, \eqref{cclt-third-order-result} retains the third-cumulant correction under \eqref{cclt-bounded}.  Under the standard scale $\delta=O(n^{-1/2})$, this extends the range of $x$ from $O(n^{1/6})$ to $O(n^{1/4})$.
	\end{remark}

\subsection{Curie--Weiss model}\label{subsec-cw-model}
The Curie--Weiss model couples all spins through a mean-field interaction. Its exponential-family Gibbs distribution on the spin configurations $\Sigma=(\sigma_1,\ldots,\sigma_n)\in\{-1,1\}^n$ is given by

\begin{equation}\label{cw-density}
 Z_{\beta,h}^{-1}\exp\biggl\{\frac{\beta}{n}
 \sum_{1\leq i<j\leq n}\sigma_i\sigma_j
 +\beta h\sum_{i=1}^n\sigma_i\biggr\}.
\end{equation}

Here $Z_{\beta,h}$ is the normalizing constant, $\beta>0$ is the inverse temperature, and $h\in\mathbb R$ is the external field. Write $S=\sum_{i=1}^n\sigma_i$ for the total magnetization.

Classical concentration and limit results were obtained by \citet{E1985} and \citet{EN1978A,EN1978B}. Cram\'er-type moderate deviations were proved by \citet{CFS2013AOP} and extended to a cubic Curie--Weiss model by \citet{EICHELSBACHER2024}. For the classical model, \citet{Z2023} extended the zero-field range to $x=O(n^{1/2})$. Berry--Esseen bounds were obtained by \citet{SZ2019} for the Curie--Weiss model and by \citet{DM2023} for the Ising model.

Throughout this subsection, $\beta$ and $h$ are fixed, with either $0<\beta<1$ and $h\in\mathbb R$, or $\beta\geq1$ and $h\ne0$. With a slight abuse of notation, let $m$ denote the unique solution of the mean-field equation in \eqref{cw-normalization} satisfying $mh\geq0$. Under the Gibbs law \eqref{cw-density}, we normalize the magnetization by

\begin{equation}\label{cw-normalization}
 m=\tanh(\beta m+\beta h),\qquad
 W=\frac{S-nm}{\sigma},\qquad
 \sigma^2=\frac{n(1-m^2)}{1-\beta(1-m^2)}.
\end{equation}

We apply Theorem \ref{theorem22} at order four, retaining the cumulants $\Gamma_3=\E(W-\E W)^3$ and $\Gamma_4=\E(W-\E W)^4-3\{\Var(W)\}^2$. The normalization in \eqref{cw-normalization} gives $\E W=O_{\beta,h}(n^{-1/2})$ and $\Var(W)=1+O_{\beta,h}(n^{-1})$ by Lemma \ref{lemma-cw-moment-expansion}, rather than exact centering and unit variance. These discrepancies contribute to the error term in the following theorem.

\begin{theorem}\label{theorem32}
 Assume that either $0<\beta<1$ and $h\in\mathbb R$, or $\beta\geq1$ and $h\ne0$, with both parameters fixed. Use the Gibbs law \eqref{cw-density} and the normalization \eqref{cw-normalization}. There exists $c=c(\beta,h)>0$ such that, uniformly for $0\leq x\leq cn^{3/10}$,

 \begin{equation}\label{cw-fourth-order-result}
  \frac{\p(W>x)}{1-\Phi(x)}
  =\exp\left\{\frac{\Gamma_3x^3}{6}
  +\frac{\Gamma_4-3\Gamma_3^2}{24}x^4\right\}
  \left[1+O_{\beta,h}\left\{
   \frac{1+x}{\sqrt n}+\frac{(1+x)^3}{n}
   +\frac{(1+x)^5}{n^{3/2}}
  \right\}\right].
 \end{equation}

\end{theorem}

\begin{remark}
 The moment calculation in Lemma \ref{lemma-cw-moment-expansion} yields

 \begin{equation}\label{cw-explicit-cumulants}
  \begin{aligned}
  \Gamma_3&=-\frac{2m}{\sigma\{1-\beta(1-m^2)\}^2}
  +O_{\beta,h}(n^{-3/2}),\\
  \Gamma_4&=\frac{2\{3m^2-1+\beta(1-m^2)(3m^2+1)\}}
  {\sigma^2\{1-\beta(1-m^2)\}^4}
  +O_{\beta,h}(n^{-2}).
  \end{aligned}
 \end{equation}
 \noindent These explicit leading terms may replace $\Gamma_3$ and $\Gamma_4$ in \eqref{cw-fourth-order-result} without changing its error.
\end{remark}

\begin{remark}
 When $0<\beta<1$ and $h=0$, \citet{Z2023} gives a normal-tail approximation without cumulant corrections, with an error bound valid up to $x=O(\sqrt n)$. For $h\neq0$, Proposition 4.3 of \citet{CFS2013AOP} establishes an uncorrected Cram\'er-type moderate-deviation estimate on the $O(n^{1/6})$ scale. For $0<\beta<1$ and arbitrary $h\in\mathbb R$, \citet[Theorem~1.4 and Corollary~1.7(1)]{ADS2026} used the cumulant method to obtain a Cram\'er-type expansion on the $O(\sqrt n)$ scale.

 Theorem \ref{theorem32} gives a fourth-order expansion via a Stein identity and also covers $\beta\geq1$ with $h\ne0$. For nonzero fields, it refines the uncorrected estimate of \citet{CFS2013AOP} and extends its uniform range from $O(n^{1/6})$ to $O(n^{3/10})$.
\end{remark}

	\section{Proof of Theorem \ref{theorem22}}\label{sec4}
	We first state the auxiliary estimates used in the proof, deferring their proofs to Section \ref{sec-auxiliary-proofs}. The proof uses exponential tilting, following the classical argument in \citet[Chapter~8]{PETROV1975}. The higher-order Stein identity provides the estimates needed to carry out this argument under our assumptions. Recall the cumulant polynomial $\Gamma(u)$ from \eqref{intro-cumulant-polynomial} and the tilted measure and logarithmic moment generating function from \eqref{intro-tilted-definitions}. For a tilt parameter $t$ satisfying $\Gamma'(t)=x$, the exact change of measure gives

	\begin{equation}\label{intro-change-measure}
		\p(W>x)
		=e^{\psi(t)-tx}
		\E_t\{e^{-t(W-x)}\I(W>x)\}.
	\end{equation}
	The two factors in \eqref{intro-change-measure} require different estimates. For the exponential factor, we need to control $\psi(t)-\Gamma(t)$. For the expectation, we need a normal approximation under $\p_t$, which we obtain using a concentration inequality and the normal Stein equation. When $x$ is bounded, a Berry--Esseen bound under the original measure will suffice.

	We begin with the exponential factor in \eqref{intro-change-measure}. The following proposition controls $\psi(t)-\Gamma(t)$, with $\kappa_m(x)$ defined in \eqref{mth-kappa}.

\begin{proposition}[Exponential moment inequality under the original measure]\label{prop-mth-mgf}
		Under the assumptions of Theorem \ref{theorem22}, for all $|t|\leq T_m$,
		\begin{equation}\label{mth-mgf-estimate}
			\left|\psi(t)-\Gamma(t)\right|\leq \kappa_m(|t|).
		\end{equation}
\end{proposition}
	To obtain a normal approximation under $\p_t$, we first rewrite the Stein identity under that measure. The substitution $g(w)=e^{tw}f(w)$ introduces binomial sums of derivatives, which we denote as follows. For $0\leq k\leq m$, define $f_{k,t}(w)\coloneqq\sum_{\ell=0}^{k}\binom{k}{\ell}t^{k-\ell}f^{(\ell)}(w)$. Substituting $g(w)$ into \eqref{mth-Stein}, dividing by $M(t)$, and rearranging, we obtain
	\begin{equation}\label{mth-conjugate-stein}
		\E_t f'(W)=\E_t\{(W-t)f(W)\}-\sum_{j=3}^{m}\frac{\Gamma_j}{(j-1)!}\E_t f_{j-1,t}(W)-\mathfrak R_{m,t}(f),
	\end{equation}
	where we define
	\begin{align}\label{mth-conjugate-remainder}
		\mathfrak R_{m,t}(f)
		&\coloneqq
		\sum_{j=1}^{m}\E_t\{\mathcal R_{m,j}f_{j-1,t}(W)\}
		+\E_t\int_{|u|\leq\rho_m}e^{tu}f_{m,t}(W+u)\mathcal K_{m,1}(u)du\\
		&\quad+\E_t\int_{|u|\leq\rho_m}e^{tUu}f_{m,t}(W+Uu)\mathcal K_{m,2}(u)du. \notag
	\end{align}
	Using $\Gamma'(t)-t=\sum_{j=3}^{m}\Gamma_jt^{j-1}/(j-1)!$, we rewrite \eqref{mth-conjugate-stein} as
	\begin{align}\label{mth-gamma-conjugate-stein}
		\E_t f'(W)
		={}&\E_t\{(W-\Gamma'(t))f(W)\}
		-\sum_{j=3}^{m}\frac{\Gamma_j}{(j-1)!}
		\E_t\{f_{j-1,t}(W)-t^{j-1}f(W)\}
		-\mathfrak R_{m,t}(f).
	\end{align}

	The first term on the right of \eqref{mth-gamma-conjugate-stein} leads us to estimate moments of $W-\Gamma'(t)$ under the tilted measure.

	\begin{lemma}[Centered exponential moment under the tilted measure]\label{lemma-mth-centered-exp}
		Assume {\rm (A1)}--{\rm (A4)}.  Uniformly for $0\leq t\leq T_m-1$,
		\begin{equation}\label{mth-centered-exp-bound}
			\E_te^{\pm(W-\Gamma'(t))}
			\leq e^{1/2}\exp\{2[\eta_m(t+1)+\kappa_m(t+1)]\}.
		\end{equation}
		Consequently, for $1\leq k\leq m$,
		\begin{equation}\label{mth-centered-moment-bound}
			\E_t|W-\Gamma'(t)|^k
			\leq C_m\exp\{2[\eta_m(t+1)+\kappa_m(t+1)]\}.
		\end{equation}
	\end{lemma}

	The centered moment bound controls the first term in \eqref{mth-gamma-conjugate-stein}, but the higher derivatives require control of short-interval probabilities. For the statement of the concentration inequality, let $B(m,m)=\int_0^1v^{m-1}(1-v)^{m-1}dv$ be the beta function evaluated at $(m,m)$. For $0\leq s \leq 1$, define
	\begin{equation}\label{mth-smoothing-functions}
		q_m(s)\coloneqq B(m,m)^{-1}\int_0^s v^{m-1}(1-v)^{m-1}dv,\qquad
		Q_m(s)\coloneqq\int_0^s q_m(v)dv.
	\end{equation}
	Then $q_m(1)=1$, $q_m^{(r)}(0)=q_m^{(r)}(1)=0$ for $1\leq r\leq m-1$, and $Q_m(1)=1/2$.  Define
	\begin{equation*}
		A_m\coloneqq\max\biggl\{1,
		\max_{2\leq r\leq m}\sup_{0\leq s\leq1}|Q_m^{(r)}(s)|\biggr\},
	\end{equation*}
	and
	\begin{align}\label{mth-concentration-length}
		\epsilon_m
		\coloneqq
		2^{m+4}A_m\biggl[
		\rho_m+\max_{3\leq j\leq m}\Gamma_{j,0}^{1/(j-2)}
		+\vartheta_m^{1/(m-1)}
		\biggr].
	\end{align}
	
\begin{proposition}[Concentration inequality under the tilted measure]\label{prop-mth-tilted-concentration}
		For $\epsilon_m$ defined in \eqref{mth-concentration-length} and every
		\[
		\begin{aligned}
			(16A_m)^{(m-1)/2}\leq t\leq \min\biggl\{
			T_m-1,\rho_m^{-(m-1)/(m+1)},\vartheta_m^{-1/(m+1)},
			\min_{3\leq k\leq m}\Gamma_{k,0}^{-(m-1)/[(k-2)(m+1)]}\biggr\},
		\end{aligned}
		\]
		under the assumptions of Theorem \ref{theorem22}, uniformly for such $t$ and for all $a<b$,
		\begin{equation}\label{mth-tilted-concentration}
			\begin{aligned}
			\p_t(a\leq W\leq b)
			&\leq
			C_m(b-a+\epsilon_m)\biggl[
			e^{2\{\eta_m(t+1)+\kappa_m(t+1)\}}
			+\sum_{j=1}^{m}t^{j-1}r_{m,j}(t)
			\biggr]\\
			&\quad+C_m\sum_{j=1}^{m}r_{m,j}(t)
			\sum_{\ell=1}^{j-1}t^{j-1-\ell}\epsilon_m^{1-\ell}.
			\end{aligned}
		\end{equation}
		In particular, with $\tilde{\Delta}_m(t)$ defined in \eqref{mth-Delta-tilde-definition}, we have
		\begin{equation}\label{mth-local-concentration}
			\sup_{z\in\mathbb R}\p_t(z\leq W\leq z+\epsilon_m)
			\leq C_m\tilde{\Delta}_m(t).
		\end{equation}
	\end{proposition}
	\begin{remark}\label{rem-original-exponential-concentration}
		Let $x$ lie in the range of Proposition \ref{prop-mth-tilted-concentration}.  For every $\varepsilon>0$,
		\[
		\p_x(x\leq W\leq x+\varepsilon)
		=\frac{\E\{e^{xW}\I(x\leq W\leq x+\varepsilon)\}}{M(x)}
		\geq\frac{e^{x^2}}{M(x)}
		\p(x\leq W\leq x+\varepsilon).
		\]
		Combining this change of measure with Propositions \ref{prop-mth-mgf} and \ref{prop-mth-tilted-concentration} gives a concentration inequality for $W$ under the original measure, with the exponentially decaying factor $e^{\Gamma(x)-x^2}$. The Stein identity thus gives another way to obtain such an inequality (compare Proposition 3.2 of \citet{FLS2020}).  For related concentration inequalities, see \citet{HC1978,SS2006,S2010}.
	\end{remark}
	By combining the tilted concentration inequality with the normal Stein equation, we obtain the following Berry--Esseen bound.

\begin{proposition}[Berry--Esseen bound under the tilted measure]\label{prop-mth-tilted-BE}
		Define $Y_{m,t}\coloneqq W-\Gamma'(t)$.  Under the assumptions of Theorem \ref{theorem22}, uniformly for $t$ in the range of Proposition \ref{prop-mth-tilted-concentration} such that $t+1\leq T_m$,
		\begin{equation}\label{mth-tilted-BE}
			\sup_{z\in\mathbb R}\left|\p_t(Y_{m,t}\leq z)-\Phi\bigl(z/\sqrt{\Gamma''(t)}\bigr)\right|\leq C_m\Delta_m(t),
		\end{equation}
		where $\Delta_m(t)$ is defined in \eqref{mth-Delta-definition}.
	\end{proposition}
	To handle bounded values of $x$, we also need a Berry--Esseen bound under the original measure. Proposition \ref{prop-mth-tilted-BE} does not cover $t=0$, so we establish this bound separately using the original Stein identity \eqref{mth-Stein}.

\begin{lemma}[Berry--Esseen bound under the original measure]
\label{prop-mth-original-BE}
		Under the assumptions of Theorem \ref{theorem22},
		\begin{equation}\label{mth-original-BE}
			\sup_{z\in\mathbb R}
			\left|\p(W\leq z)-\Phi(z)\right|
			\leq C_m\Delta_m(0).
		\end{equation}
	\end{lemma}
	We now combine these estimates in \eqref{intro-change-measure}, with $\lambda$ defined by \eqref{mth-lambda-definition}. For bounded $\lambda$, the normal approximation under the original measure is enough. For larger $\lambda$, we use the tilted normal approximation and the exponential-moment estimate.

	\begin{proof}[Proof of Theorem \ref{theorem22}]
		To verify that the saddlepoint is well defined, note that
		\begin{equation}\label{mth-saddlepoint-series-bound}
			\sum_{k=3}^{m}\frac{4^{-(k-2)}}{(k-2)!}
			\leq e^{1/4}-1<\frac12.
		\end{equation}
		Fix $x$ satisfying \eqref{mth-range}.  Since $x\geq0$, the upper bound in \eqref{mth-range} must be nonnegative.  Hence the minimum on the right-hand side of \eqref{mth-range} is at least one.  In particular, $\Gamma_{k,0}\leq1$, $3\leq k\leq m$.  Moreover, for each $3\leq k\leq m$, \eqref{mth-range} implies
		\begin{equation}\label{mth-saddlepoint-cumulant-bound}
			\Gamma_{k,0}(2x)^{k-2}
			\leq4^{-(k-2)}\Gamma_{k,0}^{2/(m+1)}
			\leq4^{-(k-2)}.
		\end{equation}
		Combining \eqref{mth-eta}, \eqref{mth-saddlepoint-series-bound}, and \eqref{mth-saddlepoint-cumulant-bound}, we obtain $|\Gamma''(t)-1|\leq\eta_m(2x)\leq\frac12$ for $0\leq t\leq2x$. Hence $\Gamma'(t)$ is strictly increasing in $t$ on $[0,2x]$.  Since $\Gamma'(0)=0$ and $\Gamma'(2x)=\int_0^{2x}\Gamma''(t)\,dt\geq x$, there exists a unique $\lambda\in[0,2x]$ satisfying \eqref{mth-lambda-definition}.  Finally, \eqref{mth-range} and $\lambda\leq2x$ yield
		\[
			\lambda+1\leq
			\min\Bigl\{
			T_m,
			\rho_m^{-(m-1)/(m+1)},
			\vartheta_m^{-1/(m+1)},
			\min_{3\leq k\leq m}
			\Gamma_{k,0}^{-(m-1)/[(k-2)(m+1)]}
			\Bigr\}.
		\]
		
		Suppose first that $\lambda$ is bounded by a sufficiently large constant depending only on $m$, chosen to be at least $(16A_m)^{(m-1)/2}$.  Since $\Gamma_{k,0}\leq1$, $3\leq k\leq m$, the identity $x=\Gamma'(\lambda)$ implies $|\Gamma'(\lambda)|\leq C_m$, so $x$ is bounded as well.  Hence $1-\Phi(\lambda)$ and $1-\Phi(x)$ are bounded below by a positive constant depending only on $m$, and Lemma \ref{prop-mth-original-BE} implies the required relative normal approximation.  It remains to compare the deterministic factor in \eqref{mth-explicit-result} with one.  Using $x=\Gamma'(\lambda)$ and Proposition \ref{prop-mth-mgf}, we obtain
	\begin{align*}
		&\biggl|
		\exp\biggl\{
		\Gamma(\lambda)-\lambda\Gamma'(\lambda)
		+\frac{\lambda^2}{2}
		+O_m\{\kappa_m(\lambda)\}
		\biggr\}
		\frac{1-\Phi(\lambda)}{1-\Phi(x)}
		-1
		\biggr|(1-\Phi(x))\\
		\leq{}&
		C_m\bigl\{
		\eta_m(\lambda)
		+\kappa_m(\lambda)
		\bigr\}e^{C_m\kappa_m(\lambda)}.
	\end{align*}
	Together with Lemma \ref{prop-mth-original-BE}, this proves the theorem when $\lambda$ is bounded, after increasing the constant $C_m$ in the remainder term.
	
	If $\rho_m$, $\vartheta_m$, or some $\Gamma_{k,0}$, $3\leq k\leq m$, is bounded below by a fixed positive constant depending only on $m$, the corresponding upper bound in \eqref{mth-range} implies $\lambda+1\leq C_m$.  Such cases are therefore covered by the argument for bounded $\lambda$ above.  It remains to consider $\lambda$ sufficiently large and $\rho_m$, $\vartheta_m$, and $\Gamma_{k,0}$, $3\leq k\leq m$, sufficiently small.
	
		The estimates at the beginning of the proof show that $t=\lambda$ satisfies all the upper bounds in Proposition \ref{prop-mth-tilted-concentration}.  The choice of the cutoff between the two cases gives $\lambda\geq(16A_m)^{(m-1)/2}$, so the lower bound in that proposition is also satisfied.  Since $\lambda+1\leq T_m$, Proposition \ref{prop-mth-tilted-BE} applies as well. By the saddlepoint equation $x=\Gamma'(\lambda)$, the two events coincide: $W>x\Longleftrightarrow Y_{m,\lambda}>0$. The change-of-measure identity at $\lambda$ then takes the form
		\begin{align}\label{mth-recovery-start}
			\p(W>x)
			=e^{\psi(\lambda)-\lambda x}
			\E_\lambda\{e^{-\lambda(W-x)}\I(W>x)\}
			=e^{\psi(\lambda)-\lambda x}
			\E_\lambda\{e^{-\lambda Y_{m,\lambda}}\I(Y_{m,\lambda}>0)\}.
		\end{align}
		Moreover, Proposition \ref{prop-mth-tilted-BE}, applied at $t=\lambda$, yields
		\begin{equation*}
			\sup_{y\in\mathbb R}\left|\p_\lambda(Y_{m,\lambda}\leq y)-\Phi\bigl(y/\sqrt{\Gamma''(\lambda)}\bigr)\right|\leq C_m\Delta_m(\lambda).
		\end{equation*}
		To compare the expectation in \eqref{mth-recovery-start} with its Gaussian counterpart, we apply integration by parts on $[0,\infty)$ to the difference of the two Stieltjes integrals.  For $Z\sim N(0,1)$, we obtain
		\begin{align}\label{mth-normal-integral-comparison}
			&\bigl|\E_\lambda\{e^{-\lambda Y_{m,\lambda}}\I(Y_{m,\lambda}>0)\}
			-\E\{e^{-\lambda\sqrt{\Gamma''(\lambda)}Z}\I(Z>0)\}\bigr|\\
			\leq{}&
			\sup_y\left|\p_\lambda(Y_{m,\lambda}\leq y)-\Phi\bigl(y/\sqrt{\Gamma''(\lambda)}\bigr)\right|\notag\\
			&+\lambda\int_{0}^{\infty}e^{-\lambda y}
			\sup_u\left|\p_\lambda(Y_{m,\lambda}\leq u)-\Phi\bigl(u/\sqrt{\Gamma''(\lambda)}\bigr)\right|dy
			\leq C_m\Delta_m(\lambda). \notag
		\end{align}
		Moreover, the Gaussian expectation can be evaluated explicitly:
		\begin{equation}\label{mth-normal-integral-value}
			\E\{e^{-\lambda\sqrt{\Gamma''(\lambda)}Z}\I(Z>0)\}
			=e^{\lambda^2\Gamma''(\lambda)/2}\{1-\Phi(\lambda\sqrt{\Gamma''(\lambda)})\}.
		\end{equation}
		By Mills' ratio and the range condition,
		\begin{equation}\label{mth-normal-integral-asymptotic}
		e^{\lambda^2\Gamma''(\lambda)/2}\{1-\Phi(\lambda\sqrt{\Gamma''(\lambda)})\}
		=e^{\lambda^2/2}\{1-\Phi(\lambda)\}\left[1+O_m\{\eta_m(\lambda)\}\right].
		\end{equation}
		The Gaussian reference term is of order $(1+\lambda)^{-1}$.  Hence, after division by this term, the absolute error in \eqref{mth-normal-integral-comparison} is of order $(1+\lambda)\Delta_m(\lambda)$.  Combining \eqref{mth-normal-integral-comparison}, \eqref{mth-normal-integral-value}, and \eqref{mth-normal-integral-asymptotic}, we obtain
		\begin{align}\label{mth-weighted-normal-tail}
			\E_\lambda\{e^{-\lambda Y_{m,\lambda}}\I(Y_{m,\lambda}>0)\}
			={}&e^{\lambda^2/2}\{1-\Phi(\lambda)\}
			\left[1+O_m\{\eta_m(\lambda)+(1+\lambda)\Delta_m(\lambda)\}\right].
		\end{align}
		Finally, applying Proposition \ref{prop-mth-mgf} at $t=\lambda$ and using $x=\Gamma'(\lambda)$, we obtain
		\begin{align}\label{mth-exponent-replacement}
			\psi(\lambda)-\lambda x+\frac{\lambda^2}{2}
			={}&
			\Gamma(\lambda)-\lambda\Gamma'(\lambda)
			+\frac{\lambda^2}{2}
			+O_m\{\kappa_m(\lambda)\}.
		\end{align}
		Combining \eqref{mth-exponent-replacement} with \eqref{mth-recovery-start} and \eqref{mth-weighted-normal-tail}, and then dividing by $1-\Phi(x)$, proves \eqref{mth-explicit-result}.
		\end{proof}

	\section{Proofs of applications}\label{sec5}

	The dependence structure enters each proof through the Stein identity and its remainder terms. Once these are controlled, Theorem \ref{theorem22} gives the tail approximation in terms of the saddlepoint, which we then express through $x$. The third-order construction below explains how to obtain such identities before we turn to the three applications.

	Throughout this section, for $y\in\mathbb R$, define $H_y(t)\coloneqq\I(-y\leq t\leq0)-\I(0<t\leq-y)$. We shall repeatedly use the elementary difference formula
	\begin{equation}\label{difference-formula}
		f(w)-f(w-y)=\int_{-\infty}^{\infty}f'(w+t)H_y(t)dt.
	\end{equation}
	In each application, the difference formula \eqref{difference-formula} converts the local or exchangeable-pair decomposition into the kernel form used in Section \ref{sec2}.

\subsection{Construction of a third-order Stein identity}
\label{subsubsec-third-order-Stein-derivation}
Let $\E W=0$, $\E W^2=1$ and $\Gamma_3=\E W^3$. For $m=3$, the identity in {\rm (A1)} becomes
\begin{align}\label{third-order-Stein-identity}
 \E\{Wf(W)\}
 &=\E f'(W)+\frac{\Gamma_3}{2}\E f''(W)
   +\sum_{j=1}^{3}\E\{\mathcal R_{3,j}f^{(j-1)}(W)\}\notag\\
 &\quad+\E\int_{|t|\leq\rho_3}f^{(3)}(W+t)\mathcal K_{3,1}(t)dt
   +\E\int_{|t|\leq\rho_3}f^{(3)}(W+Ut)\mathcal K_{3,2}(t)dt.
\end{align}
Starting from a classical Stein identity, we expand three contributions: the fluctuation of the kernel mass, the regression remainder, and the Taylor remainder. Separating their mean coefficients from the remaining terms produces \eqref{third-order-Stein-identity}, with integral remainders involving $f^{(3)}(w)$. All hatted kernels below are independent of $U$ and supported on a common deterministic interval $[-\rho_3,\rho_3]$. This interval must contain the supports introduced by the successive expansions, not just the support of the initial kernel. Every identity is understood for the stated class of test functions whenever the expectations and integrals are finite.

Suppose that a random function $\hat K(u)$ and a random variable $R$ satisfy, for every absolutely continuous $f(w)$,
\begin{equation}\label{intro-general-Stein-identity}
 \E\{Wf(W)\}
 =\E\int_{|u|\leq\rho_3}f'(W+u)\hat K(u)du+\E\{Rf(W)\}.
\end{equation}
Set $\hat L=\int_{|u|\leq\rho_3}\hat K(u)du-1$. Adding and subtracting $\E f'(W)$ in \eqref{intro-general-Stein-identity} gives
\begin{equation}\label{intro-Stein-kernel-mass-decomposition}
 \E\{Wf(W)\}
 =\E f'(W)+\E\{\hat Lf'(W)\}+\E\{Rf(W)\}
 +\E\int_{|u|\leq\rho_3}
       [f'(W+u)-f'(W)]\hat K(u)du.
\end{equation}
Take $f(w)$ to be twice differentiable with an absolutely continuous second derivative.  To expand $\E\{\hat Lf'(W)\}$, suppose that, for every absolutely continuous $g(w)$,
\begin{equation}\label{third-L-K1-decomposition}
 \E\{\hat Lg(W)\}
 =\E\int_{|t|\leq\rho_3}g'(W+t)\hat K_1(t)dt+\E\{R_1g(W)\}.
\end{equation}
Define
\begin{equation}\label{third-L1-definition}
 \gamma_{3,1}=2\E\int_{|t|\leq\rho_3}\hat K_1(t)dt,
 \qquad
 \hat L_1=\int_{|t|\leq\rho_3}\hat K_1(t)dt-\frac{\gamma_{3,1}}2,
\end{equation}
and suppose that, for every absolutely continuous $h(w)$,
\begin{equation}\label{third-L1-K4-decomposition}
 \E\{\hat L_1h(W)\}
 =\E\int_{|t|\leq\rho_3}h'(W+t)\hat K_4(t)dt+\E\{R_4h(W)\}.
\end{equation}
Applying \eqref{third-L-K1-decomposition} with $g(w)=f'(w)$ and \eqref{third-L1-K4-decomposition} with $h(w)=f''(w)$ yields
\begin{align}\label{third-L-fprime-expansion}
 \E\{\hat Lf'(W)\}
 &=\frac{\gamma_{3,1}}2\E f''(W)
   +\E\{R_1f'(W)\}+\E\{R_4f''(W)\}\notag\\
 &\quad+\E\int_{|t|\leq\rho_3}f^{(3)}(W+t)\hat K_4(t)dt
   +\E\int_{|t|\leq\rho_3}f^{(3)}(W+Ut)t\hat K_1(t)dt.
\end{align}
For $\E\{Rf(W)\}$, suppose that
\begin{equation}\label{third-R-K2-decomposition}
 \E\{Rf(W)\}
 =\E\int_{|t|\leq\rho_3}f'(W+t)\hat K_2(t)dt+\E\{R_2f(W)\}.
\end{equation}
Define $\hat L_2=\int_{|t|\leq\rho_3}\hat K_2(t)dt$ and $\hat L_{21}=\int_{|t|\leq\rho_3}t\hat K_2(t)dt$, and suppose that, for every absolutely continuous $g(w)$,
\begin{equation}\label{third-L2-K3-decomposition}
 \E\{\hat L_2g(W)\}
 =\E\int_{|t|\leq\rho_3}g'(W+t)\hat K_3(t)dt+\E\{R_3g(W)\}.
\end{equation}
Set $\gamma_{3,2}=2\E\int_{|t|\leq\rho_3}\hat K_3(t)dt$ and $\hat L_3=\int_{|t|\leq\rho_3}\hat K_3(t)dt-\gamma_{3,2}/2$. Applying \eqref{third-L2-K3-decomposition} with $g(w)=f'(w)$ and Taylor's expansion with integral remainder in \eqref{third-R-K2-decomposition}, we obtain
\begin{align}\label{third-R-refined-expansion}
 \E\{Rf(W)\}
 &=\frac{\gamma_{3,2}}2\E f''(W)
   +\E\{(\hat L_{21}+\hat L_3)f''(W)\}
   +\E\{R_2f(W)\}+\E\{R_3f'(W)\}\notag\\
 &\quad+\E\int_{|t|\leq\rho_3}f^{(3)}(W+Ut)
       [t\hat K_3(t)+(1-U)t^2\hat K_2(t)]dt.
\end{align}
For the last term on the right-hand side of \eqref{intro-Stein-kernel-mass-decomposition}, define
\begin{equation}\label{third-L4-definition}
 \gamma_{3,3}=2\E\int_{|t|\leq\rho_3}t\hat K(t)dt,
 \qquad
 \hat L_4=\int_{|t|\leq\rho_3}t\hat K(t)dt-\frac{\gamma_{3,3}}2,
\end{equation}
and suppose that, for every absolutely continuous $h(w)$,
\begin{equation}\label{third-L4-K5-decomposition}
 \E\{\hat L_4h(W)\}
 =\E\int_{|t|\leq\rho_3}h'(W+t)\hat K_5(t)dt+\E\{R_5h(W)\}.
\end{equation}
Taylor's expansion and \eqref{third-L4-K5-decomposition} with $h(w)=f''(w)$ give
\begin{equation}\label{third-K-Taylor-expansion}
 \begin{aligned}
 &\E\int_{|t|\leq\rho_3}[f'(W+t)-f'(W)]\hat K(t)dt\\
 ={}&\frac{\gamma_{3,3}}2\E f''(W)+\E\{R_5f''(W)\}
   +\E\int_{|t|\leq\rho_3}f^{(3)}(W+t)\hat K_5(t)dt\\
 &{}+\E\int_{|t|\leq\rho_3}f^{(3)}(W+Ut)(1-U)t^2\hat K(t)dt.
 \end{aligned}
\end{equation}
Combining \eqref{intro-Stein-kernel-mass-decomposition}, \eqref{third-L-fprime-expansion}, \eqref{third-R-refined-expansion}, and \eqref{third-K-Taylor-expansion} proves \eqref{third-order-Stein-identity} with
\begin{equation}\label{third-canonical-direct-identification}
 \begin{gathered}
 \mathcal R_{3,1}=R_2,\quad \mathcal R_{3,2}=R_1+R_3,\quad
 \mathcal R_{3,3}=\frac{\gamma_{3,1}+\gamma_{3,2}+\gamma_{3,3}-\Gamma_3}{2}
     +R_4+R_5+\hat L_{21}+\hat L_3,\\
 \mathcal K_{3,1}(t)=\hat K_4(t)+\hat K_5(t),\quad
 \mathcal K_{3,2}(t)=t\{\hat K_1(t)+\hat K_3(t)\}
      +(1-U)t^2\{\hat K(t)+\hat K_2(t)\}.
 \end{gathered}
\end{equation}
The remainder kernels in \eqref{third-canonical-direct-identification} are supported on $[-\rho_3,\rho_3]$, and $|Ut|\leq\rho_3$ throughout the second integral in \eqref{third-order-Stein-identity}.

\begin{remark}\label{remark-exchangeable-third-expansion}
 Suppose that $(W,W')$ is exchangeable and $D=D(W,W')$ is antisymmetric. Let $\mathcal F$ be a $\sigma$-field with respect to which $W$ and $R$ are measurable, and assume $\E\{D\mid\mathcal F\}=\lambda(W-R)$ for some $\lambda>0$. Write $\Delta=W-W'$ and assume $|\Delta|\leq\rho_3$. In \eqref{intro-general-Stein-identity}, take
 \begin{equation}\label{third-exchangeable-initial-kernel}
  \hat K(t)=\frac{D}{2\lambda}
  [\I(-\Delta\leq t\leq0)-\I(0<t\leq-\Delta)].
 \end{equation}
 Since $D\Delta^2$ is antisymmetric, $\gamma_{3,3}=-(2\lambda)^{-1}\E D\Delta^2=0$ and $\hat L_4=-D\Delta^2/(4\lambda)$. Exchangeability gives \eqref{third-L4-K5-decomposition} with
 \begin{equation}\label{third-exchangeable-terminal-kernel}
  \hat K_5(t)=-\frac{D\Delta^2}{8\lambda}
  [\I(-\Delta\leq t\leq0)-\I(0<t\leq-\Delta)],\qquad R_5=0.
 \end{equation}
 Both kernels in \eqref{third-exchangeable-initial-kernel} and \eqref{third-exchangeable-terminal-kernel} vanish outside $[-\rho_3,\rho_3]$.
\end{remark}

\subsection{Proof of Theorem \ref{theorem31}}\label{subsec-proof-theorem31}

For locally dependent sums, the variables in each Taylor coefficient are independent of the sum outside a suitable neighborhood. This allows us to center the coefficient and expand its contribution over the next neighborhood. We use Taylor's expansion directly through the integral remainder involving $f^{(4)}(w)$, collecting the mean terms into the third and fourth cumulants. The resulting identity, stated below, has no scalar remainder terms.

\begin{lemma}\label{lemma-local-fourth-stein}
Under the assumptions of Theorem \ref{theorem31}, define
\begin{equation}\label{local-neighborhood-sums}
 Y_i\coloneqq\sum_{j\in A_i}X_j,\quad
 Y_{ij}\coloneqq\sum_{k\in A_{ij}}X_k,\quad
 Y_{ijk}\coloneqq\sum_{l\in A_{ijk}}X_l,\quad
 Y_{ijkl}\coloneqq\sum_{r\in A_{ijkl}}X_r.
\end{equation}
There is a random kernel $\mathcal K_{4,1}(u)$ supported on $[-\kappa\delta,\kappa\delta]$ such that
\begin{equation}\label{local-fourth-stein}
 \E\{Wf(W)\}
 =\E f'(W)+\frac{\Gamma_3}{2}\E f''(W)
 +\frac{\Gamma_4}{6}\E f^{(3)}(W)
 +\E\int_{|u|\leq\kappa\delta}
 f^{(4)}(W+u)\mathcal K_{4,1}(u)du
\end{equation}
for every three times differentiable $f(w)$ with absolutely continuous $f^{(3)}(w)$, whenever the expectations are finite. Moreover,
\begin{equation}\label{local-fourth-size-bounds}
 |\Gamma_3|\leq Cn\kappa^2\delta^3,\qquad
 |\Gamma_4|\leq Cn\kappa^3\delta^4,\qquad
 \int_{|u|\leq\kappa\delta}|\mathcal K_{4,1}(u)|du
 \leq Cn\kappa^4\delta^5\quad\text{a.s.}
\end{equation}
An explicit kernel and the proof are given in Appendix \ref{appB}.
\end{lemma}

\begin{proof}[Proof of Theorem \ref{theorem31}]
Since $1=\E W^2\leq n\kappa\delta^2$ and $\kappa\leq n$, we have
\begin{equation}\label{local-parameter-comparison}
 \kappa\delta\leq n\kappa^2\delta^3,\qquad
 \kappa^2\delta^2\leq n\kappa^3\delta^4\leq n^2\kappa^4\delta^6,
 \qquad (n\kappa^4\delta^5)^{-1/5}\kappa\delta\leq1.
\end{equation}
If $n\kappa^2\delta^3$ is bounded below by a fixed positive constant, the stated range bounds $x$ and the cumulant correction. The conclusion then follows by increasing the absolute constant in the error term, so we assume that $n\kappa^2\delta^3$ is sufficiently small.

Apply Theorem \ref{theorem22} with $m=4$ and
\begin{equation}\label{local-theorem-parameters}
 \begin{aligned}
 &\mathcal R_{4,j}=0\ (1\leq j\leq4),\qquad
 \mathcal K_{4,2}(u)=0,\qquad T_4=(n\kappa^4\delta^5)^{-1/5},\\
 &\rho_4=Cn\kappa^2\delta^3,\qquad \vartheta_4=Cn\kappa^4\delta^5,\qquad
 \Gamma_{3,0}=Cn\kappa^2\delta^3,\qquad\Gamma_{4,0}=Cn\kappa^3\delta^4.
 \end{aligned}
\end{equation}
Lemma \ref{lemma-local-fourth-stein} verifies {\rm (A1)}, {\rm (A2)} and {\rm (A4)}. For $|t|\leq T_4$, \eqref{local-parameter-comparison} bounds $e^{|tu|}$ by $e$ on the actual kernel support. Hence the pathwise bound in \eqref{local-fourth-size-bounds} also verifies {\rm (A3)}. The range in Theorem \ref{theorem31}, with its constant chosen sufficiently small, ensures \eqref{mth-range}.

Let $\lambda$ solve
\begin{equation}\label{local-saddlepoint-equation}
 x=\Gamma'(\lambda)=\lambda+\frac{\Gamma_3}{2}\lambda^2
 +\frac{\Gamma_4}{6}\lambda^3.
\end{equation}
By \eqref{local-parameter-comparison} and \eqref{local-theorem-parameters}, the definitions \eqref{mth-Delta-tilde-definition} and \eqref{mth-Delta-definition} give, for $0\leq t\leq2x$,
\[
 \tilde\Delta_4(t)\leq Cn\kappa^2\delta^3,\qquad
 \Delta_4(t)\leq C\{n\kappa^2\delta^3+n\kappa^3\delta^4(1+t)+n\kappa^4\delta^5(1+t^4)\}.
\]
Consequently, the error quantities in Theorem \ref{theorem22} satisfy
\begin{equation}\label{local-theorem-error-bound}
 \eta_4(\lambda)+\kappa_4(\lambda)+(1+\lambda)\Delta_4(\lambda)
 \leq C\{n\kappa^2\delta^3(1+x)+n\kappa^3\delta^4(1+x)^3+(n\kappa^4\delta^5+n^3\kappa^6\delta^9)(1+x)^5\}.
\end{equation}
The right-hand side is bounded throughout the stated range. Theorem \ref{theorem22} therefore gives
\begin{equation}\label{local-lambda-form}
 \begin{aligned}
 \frac{\p(W>x)}{1-\Phi(x)}
 &=\exp\{\Gamma(\lambda)-\lambda x+\lambda^2/2\}
 \frac{1-\Phi(\lambda)}{1-\Phi(x)}\\
 &\quad\times\bigl[1+O\{n\kappa^2\delta^3(1+x)+n\kappa^3\delta^4(1+x)^3
 +(n\kappa^4\delta^5+n^3\kappa^6\delta^9)(1+x)^5\}\bigr].
 \end{aligned}
\end{equation}
To express \eqref{local-lambda-form} in terms of $x$, expand \eqref{local-saddlepoint-equation} to obtain
\begin{equation}\label{local-saddlepoint-expansion}
 \lambda=x-\frac{\Gamma_3}{2}x^2
 +\left(\frac{\Gamma_3^2}{2}-\frac{\Gamma_4}{6}\right)x^3
 +O(n^3\kappa^6\delta^9x^4).
\end{equation}
Here $\lambda\asymp x$ for $x>0$, and $n\kappa^2\delta^3x$ is uniformly small. Mills' ratio gives
\begin{equation}\label{local-mills-ratio-expansion}
 \log\frac{1-\Phi(\lambda)}{1-\Phi(x)}
 =\frac{x^2-\lambda^2}{2}
 +O\left(\frac{|x-\lambda|}{1+x}\right).
\end{equation}
Combining \eqref{local-saddlepoint-expansion} and \eqref{local-mills-ratio-expansion} yields
\begin{equation}\label{local-exponent-expansion}
 \Gamma(\lambda)-\lambda x+\lambda^2/2
 +\log\frac{1-\Phi(\lambda)}{1-\Phi(x)}
 =\frac{\Gamma_3x^3}{6}
 +\frac{\Gamma_4-3\Gamma_3^2}{24}x^4
 +O\{n\kappa^2\delta^3(1+x)+n^3\kappa^6\delta^9(1+x)^5\}.
\end{equation}
Substitute \eqref{local-exponent-expansion} into \eqref{local-lambda-form}. For $0\leq x\leq1$, the terms $n\kappa^3\delta^4(1+x)^3+(n\kappa^4\delta^5+n^3\kappa^6\delta^9)(1+x)^5$ are bounded by $Cn\kappa^2\delta^3$. For $x\geq1$, replace $(1+x)^3$ and $(1+x)^5$ by $Cx^3$ and $Cx^5$, respectively. These bounds prove \eqref{bounded-local-fourth-order-result}.
\end{proof}

	\subsection{Proof of Theorem \ref{theorem33}}\label{subsec-proof-theorem33}

	Recall the permutation sum $W$ in \eqref{cclt-W}, with $a_{i,j}=\E X_{i,j}$. Write $\bar X_{i,j}\coloneqq X_{i,j}-a_{i,j}$. By \eqref{cclt-bounded}, $|a_{i,j}|\leq\delta/12$ and $|\bar X_{i,j}|\leq\delta/6$.

	Swapping two images of the permutation gives an exchangeable pair. Let $[n]_2\coloneqq\{(i_1,i_2)\in[n]^2:i_1\neq i_2\}$. Let $\mathbf I=(I_1,I_2)$ be chosen uniformly from $[n]_2$ and independently of $(X,\pi)$, and define
	\begin{equation*}
		W'\coloneqq W-X_{I_1,\pi(I_1)}-X_{I_2,\pi(I_2)}
		+X_{I_1,\pi(I_2)}+X_{I_2,\pi(I_1)} .
	\end{equation*}
	Then $(W,W')$ is exchangeable.  Define $\Delta\coloneqq W-W'$.  Set $\lambda=2/(n-1)$, and define $R\coloneqq n^{-1}\sum_{i=1}^{n}\sum_{j=1}^{n}\bar X_{i,j}$. A direct summation using \eqref{cclt-row-column} yields $\E\{\Delta\mid X,\pi\}=\lambda(W-R)$. Hence, by exchangeability,
	\begin{equation}\label{cclt-exchangeable-stein}
		\E\{Wf(W)\}=\frac{1}{2\lambda}\E\{\Delta[f(W)-f(W')]\}+\E\{Rf(W)\}.
	\end{equation}
	For $\mathbf i=(i_1,i_2)\in[n]_2$, define
	\[
	D_{\mathbf i}\coloneqq X_{i_1,\pi(i_1)}+X_{i_2,\pi(i_2)}
	-X_{i_1,\pi(i_2)}-X_{i_2,\pi(i_1)}.
	\]
	Thus $\Delta=D_{\mathbf I}$.  By \eqref{cclt-exchangeable-stein}, we obtain
	\begin{equation}\label{cclt-basic-stein}
		\E\{Wf(W)\}
		=\E\int_{|t|\leq\delta/3}f'(W+t)\hat K(t)dt+\E\{Rf(W)\},
	\end{equation}
	where we define
	\begin{equation}\label{cclt-K}
		\hat K(t)\coloneqq\frac1{4n}\sum_{\mathbf i\in[n]_2}D_{\mathbf i}H_{D_{\mathbf i}}(t).
	\end{equation}

	The kernel-mass fluctuations in \eqref{cclt-basic-stein} still depend on the permutation sum. To expand them, we prescribe selected matches and remove the corresponding array entries, leaving a sum independent of those entries. The changes in the sums determine the kernel supports. The full-array term $R$ is kept as a scalar remainder. The following lemma records the resulting identity and bounds, whose proof is given in Appendix \ref{appC}.

	\begin{lemma}\label{lemma-cclt-third-stein}
		Assume \eqref{cclt-row-column}--\eqref{cclt-bounded}, with $0<\delta\leq1$. There exist a constant $\gamma_{3,1}$, random functions $\hat K_1(t),\hat K_4(t),\hat K_5(t)$, and random variables $R_1,R_4$ such that, on defining
		\begin{alignat*}{2}
			\mathcal R_{3,1}&=R,\qquad
			\mathcal R_{3,2}=R_1,\qquad&
			\mathcal R_{3,3}&=R_4+\frac{\gamma_{3,1}-\Gamma_3}{2},\\
			\mathcal K_{3,1}(t)&=\hat K_4(t)+\hat K_5(t),\qquad&
			\mathcal K_{3,2}(t)&=t\hat K_1(t)+(1-U)t^2\hat K(t),
		\end{alignat*}
		the identity \eqref{third-order-Stein-identity} holds.  Moreover, for $|t|\leq \delta^{-1/2}$,
		\begin{align}
			\E_t|\mathcal R_{3,1}|&\leq C\delta,\label{cclt-R-bound}\displaybreak[1]\\
			\E_t|\mathcal R_{3,2}|&\leq C(n^{-1}+\delta^2),\label{cclt-R1-bound}\displaybreak[1]\\
			\E_t|\mathcal R_{3,3}|&\leq C\delta^3.\label{cclt-R4-bound}
		\end{align}
		Furthermore, $\mathcal K_{3,1}(u)$ and $\mathcal K_{3,2}(u)$ are supported on $[-\delta,\delta]$ and satisfy, for $|t|\leq\delta^{-1/2}$,
		\begin{equation}\label{cclt-kernel-bound}
			\E\biggl\{\int_{|u|\leq\delta}e^{2|tu|}
			\bigl[|\mathcal K_{3,1}(u)|+|\mathcal K_{3,2}(u)|\bigr]du
			\,\bigg|\,W\biggr\}
			\leq C(\delta^2+n\delta^4).
		\end{equation}
	\end{lemma}

	\begin{proof}[Proof of Theorem \ref{theorem33}]
		The moment computation \eqref{cclt-gamma-proof}, together with \eqref{cclt-bounded}, shows that $|\Gamma_3|\leq Cn\delta^3$. Moreover, \eqref{cclt-normalization} and \eqref{cclt-bounded} imply $1\leq Cn\delta^2$.

		If either $n\delta^3$ or $\delta$ is not sufficiently small, the range \eqref{cclt-proof-range} is bounded and $n\delta^3(1+x^2)$ is bounded away from zero.  In this case, \eqref{cclt-third-order-result} follows by increasing the absolute constant in the error term.  We may therefore assume that $n\delta^3$ and $\delta$ are sufficiently small.  In the notation of Theorem \ref{theorem22} with $m=3$, take
		\begin{equation}\label{cclt-third-parameter-choices}
			\begin{alignedat}{4}
			T_3&=\delta^{-1/2},\qquad &
			\Gamma_{3,0}&=Cn\delta^3,\qquad &
			\rho_3&=C\Gamma_{3,0},\\
			r_{3,1}(t)&=C\delta,\qquad &
			r_{3,2}(t)&=C(n^{-1}+\delta^2),\qquad &
			r_{3,3}(t)&=C\delta^3,\qquad &
			\vartheta_3&=C(\delta^2+n\delta^4).
			\end{alignedat}
		\end{equation}
		The constants may be chosen so that $\delta\leq\rho_3$.  Hence Lemma \ref{lemma-cclt-third-stein}, together with \eqref{cclt-R-bound}--\eqref{cclt-kernel-bound}, verifies {\rm (A1)}--{\rm (A3)}, while the bound for $\Gamma_3$ verifies {\rm (A4)}. Since $1\leq Cn\delta^2$, the choices in \eqref{cclt-third-parameter-choices} satisfy
		\begin{equation}\label{cclt-third-parameter-bounds}
		r_{3,j}(t)\leq C\Gamma_{3,0}^j,\quad 1\leq j\leq3,
		\qquad \vartheta_3\leq C\Gamma_{3,0}^2,
		\qquad \rho_3^{-1}\leq C\Gamma_{3,0}^{-1},
		\qquad \rho_3^{-2}\leq C\Gamma_{3,0}^{-2}.
		\end{equation}
		After decreasing $c$ in \eqref{cclt-proof-range}, \eqref{cclt-third-parameter-bounds} shows that \eqref{mth-range} holds.  For $0\leq t\leq2x$, substitution into \eqref{mth-kappa}--\eqref{mth-Delta-definition} yields
		\begin{equation}\label{cclt-third-error-components}
		\tilde{\Delta}_3(t)\leq C\Gamma_{3,0},\qquad
		\Delta_3(t)\leq C\{\Gamma_{3,0}
		+\Gamma_{3,0}^2(1+t^3)\}.
		\end{equation}
		Consequently, \eqref{cclt-third-error-components} gives, uniformly in this range, $\eta_3(t)+\kappa_3(t)+(1+t)\Delta_3(t)\leq C\Gamma_{3,0}(1+t^2)$.

		Let $\lambda$ be the solution of the saddlepoint equation in \eqref{mth-lambda-definition}.  Since $m=3$, this equation is $x=\lambda+\Gamma_3\lambda^2/2$. Since $x-\lambda=\Gamma_3\lambda^2/2$ and $0\leq\lambda\leq2x$, the Mills-ratio bounds imply
		\[
		\left|\log\left\{
		\frac{1-\Phi(\lambda)}{1-\Phi(x)}
		\exp\left(-\frac{x^2-\lambda^2}{2}\right)\right\}\right|
		\leq\frac{C|x-\lambda|}{1+x}
		\leq C\Gamma_{3,0}(1+x).
		\]
		Therefore,
		\[
		\frac{1-\Phi(\lambda)}{1-\Phi(x)}
		=\exp\left\{\frac{x^2-\lambda^2}{2}\right\}
		\left[1+O\{\Gamma_{3,0}(1+x)\}\right].
		\]
		Furthermore, the saddlepoint equation implies
		\begin{equation*}
		\Gamma(\lambda)-\lambda\Gamma'(\lambda)+\frac{\lambda^2}{2}
		+\frac{x^2-\lambda^2}{2}
		=\frac{\Gamma_3\lambda^3}{6}
		+\frac{\Gamma_3^2\lambda^4}{8}
		=\frac{\Gamma_3x^3}{6}
		+O(\Gamma_{3,0}^2x^4).
		\end{equation*}
		In the range \eqref{cclt-proof-range}, both $\Gamma_{3,0}^2x^4$ and the error terms in \eqref{mth-explicit-result} are bounded by $C\Gamma_{3,0}(1+x^2)$.  Substitution into \eqref{mth-explicit-result}, followed by $\Gamma_{3,0}=Cn\delta^3$, proves \eqref{cclt-third-order-result}.
	\end{proof}

\subsection{Proof of Theorem \ref{theorem32}}\label{subsec-proof-theorem32}
For $\Sigma$ with the Gibbs law \eqref{cw-density}, we construct $(W,W')$ by updating one spin. Recall $m$, $W$ and $\sigma$ from \eqref{cw-normalization}. By stability, $\lambda=\{1-\beta(1-m^2)\}/n\asymp_{\beta,h}n^{-1}$ and $\sigma\asymp_{\beta,h}\sqrt n$. Choose $I$ uniformly from $[n]$, independently of $\Sigma$. Given $I=i$, draw $s\in\{-1,1\}$ from the conditional law of $\sigma_i$ given the other spins in \eqref{cw-density}, with probability
\[
 \frac{1+s\tanh\{\beta(S-\sigma_i)/n+\beta h\}}2.
\]
Replace the selected spin $\sigma_i$ by $s$, write $\sigma_i'=s$, and define
\begin{equation}\label{cw-exchangeable-pair}
 W'=W-\frac{\sigma_I-\sigma_I'}{\sigma},\qquad
 \Delta=W-W',\qquad |\Delta|\leq\frac2\sigma.
\end{equation}
Conditionally on $I$ and the other spins, $\sigma_I$ and $\sigma_I'$ are independent draws from the same conditional law. Thus $(W,W')$ has a symmetric joint distribution and is exchangeable. Define $R$ by $\E(\Delta\mid\Sigma)=\lambda(W-R)$. The expansion of $R$ is given in \eqref{cw-fourth-R-local}. Exchangeability gives
\begin{equation}\label{cw-exchangeable-stein}
 \E\{Wf(W)\}
 =\frac1{2\lambda}\E\{\Delta[f(W)-f(W')]\}
 +\E\{Rf(W)\}.
\end{equation}
Using the notation of Section \ref{subsubsec-third-order-Stein-derivation}, we take
\begin{equation}\label{cw-basic-kernel}
 \hat K(u)=\frac1{2\lambda}\E\{\Delta H_\Delta(u)\mid\Sigma\},
 \qquad
 \hat L=\int_{|u|\leq2/\sigma}\hat K(u)du-1.
\end{equation}
Thus \eqref{intro-general-Stein-identity} holds with kernel support $[-2/\sigma,2/\sigma]$. Repeatedly applying \eqref{cw-exchangeable-stein} to the polynomial terms in $R$ and $\hat L$ gives the fourth-order identity below. Its remainder bounds follow from tilted moment estimates, as shown in Appendix \ref{appD}.

\begin{lemma}\label{lemma-cw-fourth-stein}

 Assume the conditions of Theorem \ref{theorem32}. Let $\gamma_{3,n}$ and $\gamma_{4,n}$ denote the leading terms in \eqref{cw-explicit-cumulants}. There exist random variables $\mathcal R_{4,1},\ldots,\mathcal R_{4,4}$ and a random kernel $\mathcal K_{4,1}(u)$ such that, for every three times differentiable $f(w)$ with absolutely continuous $f^{(3)}(w)$, whenever the expectations below are finite,
 \begin{align}\label{cw-fourth-stein}
  \E\{Wf(W)\}
  ={}&\E f'(W)+\frac{\gamma_{3,n}}2\E f''(W)
  +\frac{\gamma_{4,n}}6\E f^{(3)}(W)\notag\\
  &+\sum_{j=1}^4\E\{\mathcal R_{4,j}f^{(j-1)}(W)\}
  +\E\int_{|u|\leq2/\sigma}f^{(4)}(W+u)
  \mathcal K_{4,1}(u)du.
 \end{align}
 For some $c_0=c_0(\beta,h)>0$, uniformly over $|t|\leq c_0n^{3/10}$,
 \begin{align}
  \E_t|\mathcal R_{4,1}|
   &\leq C_{\beta,h}\{n^{-1/2}+n^{-1}|t|+n^{-3/2}|t|^4\},
   \label{cw-fourth-R41-bound}\displaybreak[1]\\
  \E_t|\mathcal R_{4,2}|
   &\leq C_{\beta,h}\{n^{-1}+n^{-3/2}|t|^3\},
   \label{cw-fourth-R42-bound}\displaybreak[1]\\
  \E_t|\mathcal R_{4,3}|
   &\leq C_{\beta,h}n^{-3/2}(1+t^2),
   \label{cw-fourth-R43-bound}\displaybreak[1]\\
  \E_t|\mathcal R_{4,4}|
   &\leq C_{\beta,h}n^{-2}(1+t^2).
   \label{cw-fourth-R44-bound}
 \end{align}
 The kernel is supported on $[-2/\sigma,2/\sigma]$ and satisfies
 \begin{equation}\label{cw-kernel-bound}
  \E\left\{\int_{|u|\leq2/\sigma}e^{|tu|}
  |\mathcal K_{4,1}(u)|du\,\middle|\,W\right\}
  \leq C_{\beta,h}n^{-3/2}.
 \end{equation}
 Finally,
 \begin{equation}\label{cw-cumulant-matching}
  |\Gamma_3-\gamma_{3,n}|\leq C_{\beta,h}n^{-3/2},\qquad
  |\Gamma_4-\gamma_{4,n}|\leq C_{\beta,h}n^{-2}.
 \end{equation}
\end{lemma}

\begin{proof}[Proof of Theorem \ref{theorem32}]

 Theorem \ref{theorem22} requires exact centering and unit variance, so we first standardize $W$ and later transfer the expansion back. Set $\mu=\E W$, $s^2=\Var(W)$, and $\overline W=(W-\mu)/s$. By Lemma \ref{lemma-cw-moment-expansion},
 \begin{equation}\label{cw-affine-orders}
  |\mu|\leq C_{\beta,h}n^{-1/2},\qquad
  |s^2-1|\leq C_{\beta,h}n^{-1}.
 \end{equation}
 In particular, $s$ is bounded away from zero for sufficiently large $n$. Substituting $f\{(w-\mu)/s\}$ into \eqref{cw-fourth-stein} shows that $\overline W$ satisfies \eqref{mth-Stein} at order four, with cumulants $\overline\Gamma_3=\Gamma_3/s^3$, $\overline\Gamma_4=\Gamma_4/s^4$, and
 \begin{align}\label{cw-standardized-remainders}
  \overline{\mathcal R}_{4,1}
   &=s^{-1}(\mathcal R_{4,1}-\mu),
  &\overline{\mathcal R}_{4,2}
   &=s^{-2}(1+\mathcal R_{4,2})-1,\notag\\*
  \overline{\mathcal R}_{4,3}
   &=s^{-3}\left(\mathcal R_{4,3}
      +\frac{\gamma_{3,n}-\Gamma_3}{2}\right),
  &\overline{\mathcal R}_{4,4}
   &=s^{-4}\left(\mathcal R_{4,4}
      +\frac{\gamma_{4,n}-\Gamma_4}{6}\right).
 \end{align}
 The transformed kernels are $\overline{\mathcal K}_{4,1}(u)=s^{-4}\mathcal K_{4,1}(su)$ and $\overline{\mathcal K}_{4,2}(u)=0$. Tilting $\overline W$ by $t$ is equivalent to tilting $W$ by $t/s$. Equations \eqref{cw-fourth-R41-bound}--\eqref{cw-cumulant-matching} and \eqref{cw-affine-orders} therefore give the same bounds for the transformed remainders, after enlarging $C_{\beta,h}$ and decreasing $c_0$. Thus {\rm (A2)} holds with the right-hand sides of \eqref{cw-fourth-R41-bound}--\eqref{cw-fourth-R44-bound} as $r_{4,1}(t),\ldots,r_{4,4}(t)$. Conditions {\rm (A1)}, {\rm (A3)} and {\rm (A4)} hold with $T_4=c_0n^{3/10}$, $\rho_4=C_{\beta,h}n^{-1/2}$, $\vartheta_4=C_{\beta,h}n^{-3/2}$, $\Gamma_{3,0}=C_{\beta,h}n^{-1/2}$ and $\Gamma_{4,0}=C_{\beta,h}n^{-1}$.

 Put $\varepsilon_n(t)=(1+t)n^{-1/2}+(1+t)^3n^{-1}+(1+t)^5n^{-3/2}$ for $t\geq0$. Substituting these bounds into \eqref{mth-kappa}--\eqref{mth-Delta-definition} gives
 \begin{equation}\label{cw-standardized-error-bound}
  \eta_4(t)+\kappa_4(t)+(1+t)\Delta_4(t)
  \leq C_{\beta,h}\varepsilon_n(t),\qquad 0\leq t\leq cn^{3/10},
 \end{equation}
 for sufficiently small $c=c(\beta,h)>0$. All restrictions in \eqref{mth-range} have order $n^{3/10}$. Theorem \ref{theorem22} and the saddlepoint calculation in the proof of Theorem \ref{theorem31} now yield
 \begin{equation}\label{cw-standardized-result}
  \frac{\p(\overline W>y)}{1-\Phi(y)}
  =\exp\biggl\{\frac{\overline\Gamma_3y^3}{6}
   +\frac{\overline\Gamma_4-3\overline\Gamma_3^2}{24}y^4\biggr\}
   [1+O_{\beta,h}\{\varepsilon_n(y)\}],
 \end{equation}
 uniformly for $0\leq y\leq cn^{3/10}$.

 To return to $W$, put $y=(x-\mu)/s$. Equation \eqref{cw-affine-orders} gives $|y-x|\leq C_{\beta,h}(n^{-1/2}+x/n)$. For $x\geq2|\mu|$, both $x$ and $y$ are nonnegative, and Mills' ratio gives
 \begin{equation}\label{cw-normal-tail-transfer}
  \left|\log\frac{1-\Phi(y)}{1-\Phi(x)}\right|
  \leq C_{\beta,h}\left\{\frac{1+x}{\sqrt n}
   +\frac{(1+x)^2}{n}\right\}.
 \end{equation}
 Moreover, the cumulant bounds and \eqref{cw-affine-orders} imply
 \begin{equation}\label{cw-exponent-transfer}
  \left|\frac{\overline\Gamma_3y^3-\Gamma_3x^3}{6}
   +\frac{(\overline\Gamma_4-3\overline\Gamma_3^2)y^4
    -(\Gamma_4-3\Gamma_3^2)x^4}{24}\right|
  \leq C_{\beta,h}\left\{\frac{(1+x)^3}{n}
   +\frac{(1+x)^5}{n^{3/2}}\right\}.
 \end{equation}
 Combining \eqref{cw-standardized-result}, \eqref{cw-normal-tail-transfer} and \eqref{cw-exponent-transfer} proves \eqref{cw-fourth-order-result} on this range, after decreasing $c$ if necessary. For $0\leq x<2|\mu|$, apply \eqref{mth-original-BE} to $\overline W$ and use \eqref{cw-affine-orders}, which gives an error of $O_{\beta,h}(n^{-1/2})$. Increasing the constant covers the finitely many remaining values of $n$.
\end{proof}

\section{Proofs of the auxiliary results in Section \ref{sec4}}\label{sec-auxiliary-proofs}
We now prove the propositions and lemmas stated in Section \ref{sec4}, using the notation introduced there.

\begin{proof}[Proof of Proposition \ref{prop-mth-mgf}]
		Following the exponential-moment approach of \citet{CHATTERJEE2007,CFS2013AOP}, we control $\psi'(t)-\Gamma'(t)$ through the remainder assumptions and then integrate. For fixed $|t|<T_m$, the choice $f(w)=e^{tw}$ gives $f^{(k)}(w)=t^ke^{tw}$. Using $\Gamma'(t)=t+\sum_{j=3}^{m}\Gamma_jt^{j-1}/(j-1)!$ and $\psi'(t)=\E_tW$, we substitute this function into \eqref{mth-Stein} and divide by $M(t)$ to obtain
		\begin{align}\label{mth-mgf-derivative-identity}
			\psi'(t)-\Gamma'(t)
			&=\sum_{j=1}^{m}t^{j-1}\E_t\mathcal R_{m,j}+t^m\E_t\int_{|u|\leq\rho_m}
			e^{tu}\mathcal K_{m,1}(u)\,du \notag\\
			&\quad+t^m\E_t\int_{|u|\leq\rho_m}
			e^{tUu}\mathcal K_{m,2}(u)\,du.
		\end{align}
		By {\rm (A2)} and the identity $r_{m,j}(-t)=r_{m,j}(t)$,
		\begin{equation}\label{mth-mgf-remainder-bound}
			|\E_t\mathcal R_{m,j}|
			\leq\E_t|\mathcal R_{m,j}|
			\leq r_{m,j}(|t|).
		\end{equation}
		For the two kernel terms, $0\leq U\leq1$ implies $e^{tu}\leq e^{|tu|}$ and $e^{tUu}\leq e^{|tu|}$. Applying the conditional bound in {\rm (A3)} and taking expectations gives
		\begin{align}
			&\left|\E_t\int_{|u|\leq\rho_m}
			e^{tu}\mathcal K_{m,1}(u)\,du\right|
			+\left|\E_t\int_{|u|\leq\rho_m}
			e^{tUu}\mathcal K_{m,2}(u)\,du\right|\notag\\*
			\leq&\frac{1}{M(t)}\E\biggl[e^{tW}
			\E\biggl\{\int_{|u|\leq\rho_m}e^{|tu|}
			\bigl\{|\mathcal K_{m,1}(u)|+|\mathcal K_{m,2}(u)|\bigr\}\,du
			\,\bigg|\,W\biggr\}\biggr]
			\leq
			\vartheta_m.\label{mth-mgf-kernel-bound}
		\end{align}
		Combining \eqref{mth-mgf-derivative-identity}, \eqref{mth-mgf-remainder-bound}, and \eqref{mth-mgf-kernel-bound}, we obtain
		\begin{equation}\label{mth-mgf-derivative-bound}
			|\psi'(t)-\Gamma'(t)|
			\leq\sum_{j=1}^{m}|t|^{j-1}r_{m,j}(|t|)
			+|t|^m\vartheta_m,
			\qquad |t|<T_m.
		\end{equation}

		For $t\geq0$, integrate \eqref{mth-mgf-derivative-bound} over $[0,t]$. For $t<0$, use $\psi(t)-\Gamma(t)=-\int_t^0\{\psi'(s)-\Gamma'(s)\}\,ds$ and then make the substitution $v=-s$.  Since $\psi(0)=\Gamma(0)=0$, integrating \eqref{mth-mgf-derivative-bound} over the corresponding interval and using the monotonicity of $r_{m,j}(v)$ on $[0,T_m]$, we obtain, in both cases,
		\begin{align*}
			|\psi(t)-\Gamma(t)|
			\leq\sum_{j=1}^{m}\int_0^{|t|}
			v^{j-1}r_{m,j}(v)\,dv
			+\vartheta_m\int_0^{|t|} v^m\,dv
			\leq\sum_{j=1}^{m}\frac{|t|^j}{j}r_{m,j}(|t|)
			+\frac{|t|^{m+1}}{m+1}\vartheta_m
			\leq\kappa_m(|t|).
		\end{align*}
		By continuity of $M(t)$, taking limits from the interior and using $\kappa_m(|t|)\leq\kappa_m(T_m)$ gives the same bound at $t=\pm T_m$. This proves \eqref{mth-mgf-estimate} for $|t|\leq T_m$.
	\end{proof}

	\begin{proof}[Proof of Lemma \ref{lemma-mth-centered-exp}]
		For $s\in\{-1,1\}$, we compare the moment generating function at $t$ and $t+s$ to estimate the centered exponential moments. This comparison leaves a second-order Taylor remainder for $\Gamma(u)$. The resulting exponential bound will also control the polynomial moments. Proposition \ref{prop-mth-mgf} and the monotonicity of $\kappa_m(x)$ for $x\geq0$ give
		\[
			\log\E_te^{s(W-\Gamma'(t))}
			=\psi(t+s)-\psi(t)-s\Gamma'(t)
			\leq|\Gamma(t+s)-\Gamma(t)-s\Gamma'(t)|
			+2\kappa_m(t+1).
		\]
		Taylor's expansion gives
		\[
		\Gamma(t+s)-\Gamma(t)-s\Gamma'(t)=s^2\int_0^1(1-v)\Gamma''(t+vs)\,dv.
		\]
		Since $|t+vs|\leq t+1$, \eqref{mth-eta} and {\rm (A4)} imply $|\Gamma(t+s)-\Gamma(t)-s\Gamma'(t)|\leq\frac12\{1+\eta_m(t+1)\}$. Consequently,
		\[
		\log\E_te^{s(W-\Gamma'(t))}\leq\frac12+2\{\eta_m(t+1)+\kappa_m(t+1)\},
		\]
		which proves \eqref{mth-centered-exp-bound}.  The moment bound \eqref{mth-centered-moment-bound} follows from $|y|^k\leq C_k(e^y+e^{-y})$.
	\end{proof}

	\begin{proof}[Proof of Proposition \ref{prop-mth-tilted-concentration}]
		Following Proposition 3.1 of \citet{CS2004}, we use a smoothed test function in \eqref{mth-gamma-conjugate-stein} to derive recursive bounds for interval probabilities. To express this recursion, for $r\geq0$ set $L_t(r)\coloneqq\sup_{z\in\mathbb R}\p_t(z\leq W\leq z+r)$. We choose the test function so that its derivative majorizes the interval indicator and its higher derivatives are supported near the endpoints. Their expectations are then controlled by $L_t(\epsilon_m)$, whose coefficient can be absorbed by the choice of smoothing width. Fix $a<b$ and, using $Q_m(s)$ from \eqref{mth-smoothing-functions}, define
		\begin{equation*}
			\mathfrak f_m(w)=
			\begin{cases}
				-\dfrac{b-a+\epsilon_m}{2}, & w\leq a-\epsilon_m,\\[1.5ex]
				-\dfrac{b-a+\epsilon_m}{2}
				+\epsilon_m Q_m\Bigl(\dfrac{w-a+\epsilon_m}{\epsilon_m}\Bigr),
				& a-\epsilon_m<w\leq a,\\[2ex]
				w-\dfrac{a+b}{2}, & a<w\leq b,\\[1.5ex]
				\dfrac{b-a}{2}+(w-b)
				-\epsilon_m Q_m\Bigl(\dfrac{w-b}{\epsilon_m}\Bigr),
				& b<w\leq b+\epsilon_m,\\[2ex]
				\dfrac{b-a+\epsilon_m}{2}, & w>b+\epsilon_m.
			\end{cases}
		\end{equation*}
		Then $\mathfrak f_m(w)$ is $m$ times continuously differentiable.  Moreover,
		\begin{equation}\label{mth-smooth-basic-bounds}
			|\mathfrak f_m(w)|\leq\frac{b-a+\epsilon_m}{2},\qquad
			\I(a\leq w\leq b)\leq\mathfrak f_m'(w)\leq1,
		\end{equation}
		and, for $2\leq r\leq m$,
		\begin{equation}\label{mth-smooth-derivative-bounds}
			|\mathfrak f_m^{(r)}(w)|
			\leq A_m\epsilon_m^{1-r}\bigl[\I(a-\epsilon_m\leq w\leq a)+\I(b\leq w\leq b+\epsilon_m)\bigr].
		\end{equation}
		By \eqref{mth-smooth-basic-bounds}, $\p_t(a\leq W\leq b)\leq \E_t\mathfrak f_m'(W)$, so it remains to bound $\E_t\mathfrak f_m'(W)$ using \eqref{mth-gamma-conjugate-stein}. For the first term on its right-hand side, \eqref{mth-smooth-basic-bounds} and Lemma \ref{lemma-mth-centered-exp} yield
		\begin{align}\notag
			\left|\E_t\{(W-\Gamma'(t))\mathfrak f_m(W)\}\right|
			\leq
			C_m(b-a+\epsilon_m)
			e^{2\{\eta_m(t+1)+\kappa_m(t+1)\}}.
		\end{align}
		For the cumulant terms, we have
		\begin{equation}\label{mth-cumulant-term-expansion}
		\sum_{j=3}^{m}\frac{\Gamma_j}{(j-1)!}
		\E_t\{\mathfrak f_{m,j-1,t}(W)-t^{j-1}\mathfrak f_m(W)\}
		=\sum_{j=3}^{m}\frac{\Gamma_j}{(j-1)!}
		\sum_{\ell=1}^{j-1}\binom{j-1}{\ell}t^{j-1-\ell}
		\E_t\mathfrak f_m^{(\ell)}(W).
		\end{equation}
		We first consider the term with $\ell=1$ on the right-hand side of \eqref{mth-cumulant-term-expansion}, which is $\sum_{j=3}^{m}\Gamma_jt^{j-2}\allowbreak\E_t\mathfrak f_m'(W)/(j-2)!$. The upper bound on $t$ ensures $\Gamma_{j,0}t^{j-2}\leq\Gamma_{j,0}^{2/(m+1)}$.  Moreover, since the same $t$ satisfies the lower bound in Proposition \ref{prop-mth-tilted-concentration}, $(16A_m)^{(m-1)/2}\leq\Gamma_{j,0}^{-(m-1)/[(j-2)(m+1)]}$, which implies
		\begin{equation*}
			\frac{|\Gamma_j|t^{j-2}}{(j-2)!}
			\leq\frac{\Gamma_{j,0}t^{j-2}}{(j-2)!}
			\leq\frac{(16A_m)^{-(j-2)}}{(j-2)!},
			\qquad 3\leq j\leq m.
		\end{equation*}
		Consequently,
		\begin{align}\label{mth-cumulant-main-sum}
			\sum_{j=3}^{m}\frac{|\Gamma_j|t^{j-2}}{(j-2)!}
			&\leq\sum_{j=3}^{m}\frac{\Gamma_{j,0}t^{j-2}}{(j-2)!}
			\leq\sum_{r=1}^{m-2}\frac{(16A_m)^{-r}}{r!} \notag\\
			&\leq\sum_{r=1}^{\infty}(16A_m)^{-r}
			=\frac{1}{16A_m-1}\leq\frac{1}{15A_m}\leq\frac{1}{15}.
		\end{align}
		Here we used $A_m\geq1$. Note that $\mathfrak f_m'(w)=1$ on $[a,b]$, while $0\leq\mathfrak f_m'(w)\leq1$ on $[a-\epsilon_m,a]$ and $[b,b+\epsilon_m]$.  Hence, by the definition of $L_t(r)$,
		\begin{align}\label{mth-first-derivative-splitting}
			0\leq\E_t\mathfrak f_m'(W)-\p_t(a\leq W\leq b)
			\leq2L_t(\epsilon_m).
		\end{align}
		By \eqref{mth-cumulant-main-sum}, the coefficient of $\p_t(a\leq W\leq b)$ arising from $\E_t\mathfrak f_m'(W)$ is at most $1/15$ in absolute value, so this term is moved to the left-hand side of the inequality obtained from \eqref{mth-gamma-conjugate-stein}. The remaining part of the $\ell=1$ term is bounded by \eqref{mth-first-derivative-splitting}.  For $\ell\geq2$, \eqref{mth-smooth-derivative-bounds} shows that $\mathfrak f_m^{(\ell)}(w)$ is supported on the same two transition intervals. Hence all the remaining cumulant terms satisfy
		\begin{align}\label{mth-cumulant-iteration}
			&\sum_{j=3}^{m}\frac{\Gamma_{j,0}}{(j-1)!}\biggl[
			(j-1)t^{j-2}\E_t\left|\mathfrak f_m'(W)
			-\I(a\leq W\leq b)\right|+\sum_{\ell=2}^{j-1}\binom{j-1}{\ell}t^{j-1-\ell}
			\E_t|\mathfrak f_m^{(\ell)}(W)|\biggr]\\
			\leq{}&2A_m\sum_{j=3}^{m}\frac{\Gamma_{j,0}}{(j-1)!}
			\sum_{\ell=1}^{j-1}\binom{j-1}{\ell}
			t^{j-1-\ell}\epsilon_m^{1-\ell}L_t(\epsilon_m).\notag
		\end{align}
		For $\ell=1$, \eqref{mth-cumulant-main-sum} bounds the coefficient of $L_t(\epsilon_m)$ on the right-hand side of \eqref{mth-cumulant-iteration} by
		\begin{equation}\label{mth-cumulant-iteration-first-coefficient}
			2A_m\sum_{j=3}^{m}\frac{\Gamma_{j,0}t^{j-2}}{(j-2)!}
			\leq\frac{2A_m}{16A_m-1}
			\leq\frac{2}{15}.
		\end{equation}
		For $2\leq\ell\leq j-1$, \eqref{mth-concentration-length} and the upper bound on $t$ imply $\Gamma_{j,0}^{1/(j-2)}\epsilon_m^{-1}\leq(2^{m+4}A_m)^{-1}$ and $\Gamma_{j,0}^{1/(j-2)}t\leq1$. Consequently,
		\[
		\Gamma_{j,0}t^{j-1-\ell}\epsilon_m^{1-\ell}
		=\{\Gamma_{j,0}^{1/(j-2)}t\}^{j-1-\ell}
		\{\Gamma_{j,0}^{1/(j-2)}\epsilon_m^{-1}\}^{\ell-1}
		\leq(2^{m+4}A_m)^{-1},
		\]
		and the coefficient of $L_t(\epsilon_m)$ corresponding to $2\leq\ell\leq j-1$ is bounded by
		\begin{align}\label{mth-cumulant-iteration-higher-coefficient}
				2A_m\sum_{j=3}^{m}\frac{\Gamma_{j,0}}{(j-1)!}
				\sum_{\ell=2}^{j-1}\binom{j-1}{\ell}
				t^{j-1-\ell}\epsilon_m^{1-\ell}
				&\leq 2^{-(m+3)}
				\sum_{j=3}^{m}\frac{1}{(j-1)!}
				\sum_{\ell=2}^{j-1}\binom{j-1}{\ell}\\ \notag
				&= 2^{-(m+3)}\sum_{j=3}^{m}\frac{2^{j-1}-j}{(j-1)!}
				<2^{-(m+3)}(e^2-2e)
				<2^{-m-2}.
		\end{align}
		Combining \eqref{mth-cumulant-iteration-first-coefficient} and \eqref{mth-cumulant-iteration-higher-coefficient} shows that the left-hand side of \eqref{mth-cumulant-iteration} is bounded by
		\begin{equation}\label{mth-cumulant-iteration-total-bound}
		\left(\frac{2}{15}+2^{-m-2}\right)L_t(\epsilon_m).
		\end{equation}
		For the scalar remainder terms in \eqref{mth-conjugate-remainder}, \eqref{A2-moment} and \eqref{mth-smooth-basic-bounds}--\eqref{mth-smooth-derivative-bounds} yield
		\begin{align}\label{mth-concentration-R}
			\sum_{j=1}^{m}\left|\E_t\{\mathcal R_{m,j}\mathfrak f_{m,j-1,t}(W)\}\right|
			\leq
			C_m\sum_{j=1}^{m}r_{m,j}(t)\biggl[t^{j-1}(b-a+\epsilon_m)
			+\sum_{\ell=1}^{j-1}t^{j-1-\ell}\epsilon_m^{1-\ell}\biggr].
		\end{align}
		It remains to consider the two kernel integrals in \eqref{mth-conjugate-remainder}.  It follows from \eqref{A3-kernel} and $|\mathfrak f_m'(w)|\leq1$ that the terms with $\ell=1$ satisfy
		\begin{align}
			&\left|\E_t\int_{|u|\leq\rho_m}e^{tu}mt^{m-1}
			\mathfrak f_m'(W+u)\mathcal K_{m,1}(u)du\right|+\left|\E_t\int_{|u|\leq\rho_m}e^{tUu}mt^{m-1}
			\mathfrak f_m'(W+Uu)\mathcal K_{m,2}(u)du\right|\notag\\*
			\leq{}&m\vartheta_mt^{m-1}
			\leq m\vartheta_m^{2/(m+1)}
			\leq m\vartheta_m^{1/(m-1)}
			\leq C_m\epsilon_m.\label{mth-kernel-first-derivative-bound}
		\end{align}
		The second inequality in \eqref{mth-kernel-first-derivative-bound} follows from $t\leq\vartheta_m^{-1/(m+1)}$, while the third uses $m\geq3$ and $\vartheta_m\leq1$, the latter being implied by the nonempty range of $t$.  For $2\leq\ell\leq m$, the support of $\mathfrak f_m^{(\ell)}(W+v)$ is contained in two intervals of length $\epsilon_m+2\rho_m\leq2\epsilon_m$ whenever $|v|\leq\rho_m$. Each can be covered by two intervals of length $\epsilon_m$, which implies $L_t(2\epsilon_m)\leq2L_t(\epsilon_m)$.  Hence
		\begin{align}\label{mth-kernel-iteration}
			&\biggl|\E_t\int_{|u|\leq\rho_m}e^{tu}
			\sum_{\ell=2}^{m}\binom{m}{\ell}t^{m-\ell}
			\mathfrak f_m^{(\ell)}(W+u)\mathcal K_{m,1}(u)du\biggr| \notag\\
			&\quad+\biggl|\E_t\int_{|u|\leq\rho_m}e^{tUu}
			\sum_{\ell=2}^{m}\binom{m}{\ell}t^{m-\ell}
			\mathfrak f_m^{(\ell)}(W+Uu)\mathcal K_{m,2}(u)du\biggr| \notag\\
			\leq{}&4A_m\vartheta_m\sum_{\ell=2}^{m}\binom{m}{\ell}
			t^{m-\ell}\epsilon_m^{1-\ell}L_t(\epsilon_m).
		\end{align}
		Furthermore, the upper bound $t\leq\vartheta_m^{-1/(m+1)}$ and $\vartheta_m\leq1$ yield
		\begin{align*}
			\vartheta_mt^{m-\ell}\epsilon_m^{1-\ell}
			=\{\vartheta_m^{1/(m+1)}t\}^{m-\ell}
			\{\vartheta_m^{1/(m-1)}\epsilon_m^{-1}\}^{\ell-1}
			\vartheta_m^{2(m-\ell)/(m^2-1)}
			\leq(2^{m+4}A_m)^{-1}.
		\end{align*}
		Thus the coefficient of $L_t(\epsilon_m)$ on the right-hand side of \eqref{mth-kernel-iteration} satisfies
		\begin{equation}\label{mth-kernel-iteration-coefficient}
			4A_m\vartheta_m\sum_{\ell=2}^{m}\binom{m}{\ell}
			t^{m-\ell}\epsilon_m^{1-\ell}
			\leq2^{-m-2}\sum_{\ell=2}^{m}\binom{m}{\ell}
			=\frac14-\frac{m+1}{2^{m+2}}.
		\end{equation}
		Hence the left-hand side of \eqref{mth-kernel-iteration} is bounded by $\{1/4-(m+1)/2^{m+2}\}L_t(\epsilon_m)$. The terms corresponding to $\ell=0$ are bounded directly by \eqref{A3-kernel}:
		\begin{align}\notag
			&\biggl|\E_t\int_{|u|\leq\rho_m}e^{tu}t^m
			\mathfrak f_m(W+u)\mathcal K_{m,1}(u)du\biggr|+\biggl|\E_t\int_{|u|\leq\rho_m}e^{tUu}t^m
			\mathfrak f_m(W+Uu)\mathcal K_{m,2}(u)du\biggr|\\ \notag
			\leq & C_m\vartheta_mt^m(b-a+\epsilon_m)
			\leq C_m(b-a+\epsilon_m).
		\end{align}
		Here we used $t\leq\vartheta_m^{-1/(m+1)}$. We now combine the estimates for the cumulant, scalar remainder, and kernel terms.  By \eqref{mth-cumulant-main-sum}, the cumulant term containing $\p_t(a\leq W\leq b)$ has coefficient at most $1/15$ and has already been moved to the left-hand side.  Dividing by $1-1/15$ multiplies the remaining right-hand side by at most $15/14$.  Before this division, \eqref{mth-cumulant-iteration-total-bound} bounds the cumulant-iteration term, while \eqref{mth-kernel-iteration-coefficient} bounds the term from the kernel iteration by $\{1/4-(m+1)/2^{m+2}\}L_t(\epsilon_m)$.  Their final coefficient is bounded by
		\[
		\frac{15}{14}\left(\frac{2}{15}+2^{-m-2}
		+\frac14-\frac{m+1}{2^{m+2}}\right)
		\leq\frac{15}{14}\left(\frac{2}{15}+\frac14\right)
		=\frac{23}{56}.
		\]
		The terms containing $t^{j-1}r_{m,j}(t)(b-a+\epsilon_m)$ in \eqref{mth-concentration-R} form the second term in the brackets below. Consequently,
		\begin{align}\label{mth-tilted-concentration-recursive}
			\p_t(a\leq W\leq b)
			\leq{}& C_m(b-a+\epsilon_m)\biggl[
			e^{2\{\eta_m(t+1)+\kappa_m(t+1)\}}
			+\sum_{j=1}^{m}t^{j-1}r_{m,j}(t)
			\biggr]\notag\\
			&+C_m\sum_{j=1}^{m}r_{m,j}(t)
			\sum_{\ell=1}^{j-1}t^{j-1-\ell}\epsilon_m^{1-\ell}
			+\frac{23}{56}L_t(\epsilon_m).
		\end{align}
		Choosing $a=z$ and $b=z+\epsilon_m$ makes the interval length equal to the smoothing width, allowing the recursive inequality to close. Taking the supremum over $z$ and moving $\frac{23}{56}L_t(\epsilon_m)$ to the left-hand side yields \eqref{mth-local-concentration}.  Substituting \eqref{mth-local-concentration} into \eqref{mth-tilted-concentration-recursive} removes the remaining $L_t(\epsilon_m)$ term and yields \eqref{mth-tilted-concentration}.
	\end{proof}

\begin{proof}[Proof of Proposition \ref{prop-mth-tilted-BE}]
		Collecting the coefficients of $f(W)$ and $f'(W)$ in \eqref{mth-conjugate-stein} identifies a Gaussian reference with center $\Gamma'(t)$ and variance $\Gamma''(t)$. For $m\geq4$, the variance correction must be retained to obtain the stated error bound. The comparison for the centered variable $Y_{m,t}$ therefore uses a normal Stein equation with variance $\Gamma''(t)$. For $z\in\mathbb R$, define its smoothed indicator using $q_m(s)$ from \eqref{mth-smoothing-functions} by
		\begin{equation*}
			h_{m,z}(w)=
			\begin{cases}
				1, & w<z,\\[0.8ex]
				1-q_m\Bigl(\dfrac{w-z}{\rho_m}\Bigr), & z\leq w<z+\rho_m,\\[2ex]
				0, & w\geq z+\rho_m.
			\end{cases}
		\end{equation*}
		By construction, $h_{m,z}(w)$ is $(m-1)$ times continuously differentiable and $h_{m,z}^{(m-1)}(w)$ is absolutely continuous.  Moreover, for $1\leq\ell\leq m$ and almost every $w$,
		\begin{equation}\label{mth-h-derivative-bound}
			|h_{m,z}^{(\ell)}(w)|\leq C_m\rho_m^{-\ell}\I(z\leq w\leq z+\rho_m).
		\end{equation}
		Let $f_{m,z}(w)$ solve the Stein equation
		\begin{equation}\label{mth-BE-stein-equation}
			\Gamma''(t)f_{m,z}'(w)-wf_{m,z}(w)=h_{m,z}(w)-\E h_{m,z}(\sqrt{\Gamma''(t)}Z),\qquad Z\sim N(0,1).
		\end{equation}
		Since $\Gamma''(t)-1=\sum_{j=3}^{m}\Gamma_jt^{j-2}/(j-2)!$, \eqref{mth-cumulant-main-sum} implies that $1/2\leq\Gamma''(t)\leq2$ in the present range.  After rescaling the Stein equation to variance $\Gamma''(t)$, Lemma 2.4 of \citet{CGS2011} yields
		\begin{equation}\label{mth-BE-basic-Stein-bounds}
			|f_{m,z}(w)|\leq\sqrt{\pi/2},\qquad |f_{m,z}'(w)|\leq2.
		\end{equation}
		Repeated differentiation of \eqref{mth-BE-stein-equation} yields, for $\ell\geq1$,
		\begin{equation}\label{mth-BE-derivative-recursion}
			\Gamma''(t)f_{m,z}^{(\ell+1)}(w)=wf_{m,z}^{(\ell)}(w)+\ell f_{m,z}^{(\ell-1)}(w)+h_{m,z}^{(\ell)}(w)
		\end{equation}
		at every point where $h_{m,z}^{(\ell)}(w)$ exists.  Combining \eqref{mth-h-derivative-bound}--\eqref{mth-BE-derivative-recursion}, we obtain, for $2\leq\ell\leq m$,
		\begin{equation}\label{mth-BE-derivative-bound}
			|f_{m,z}^{(\ell)}(w)|
			\leq C_m(1+|w|^{\ell-1})
			+C_m\sum_{r=1}^{\ell-1}\rho_m^{-r}(1+|w|^{\ell-1-r})\I(z\leq w\leq z+\rho_m).
		\end{equation}

		Applying the Stein solution in \eqref{mth-conjugate-stein} requires bounds for the polynomial terms and interval indicators in \eqref{mth-BE-derivative-bound}. We begin with the indicators, controlled by the concentration inequality in Proposition \ref{prop-mth-tilted-concentration}. Smoothing bounds the approximation error by a Stein expectation, an interval probability, and a deterministic smoothing term:
		\begin{align}\label{mth-BE-smoothing-reduction}
			\left|\p_t(Y_{m,t}\leq z)-\Phi\bigl(z/\sqrt{\Gamma''(t)}\bigr)\right|
			\leq{}&
			\left|\E_t h_{m,z}(Y_{m,t})-\E h_{m,z}(\sqrt{\Gamma''(t)}Z)\right| \notag\\
			&\quad+\p_t(z\leq Y_{m,t}\leq z+\rho_m)+C\rho_m.
		\end{align}
		Since $\rho_m\leq\epsilon_m$, Proposition \ref{prop-mth-tilted-concentration} implies
		\begin{equation}\label{mth-BE-small-interval}
			\sup_{z\in\mathbb R}\p_t(z\leq Y_{m,t}\leq z+\rho_m)\leq C_m\tilde{\Delta}_m(t).
		\end{equation}
		To handle moments weighted by the random remainder coefficients, we shift the tilt parameter and use Proposition \ref{prop-mth-mgf}. The resulting exponential-shift argument based on \eqref{mth-mgf-estimate} yields, for $1\leq k\leq m$ and every random variable $G$ for which the right-hand side is finite,
		\begin{equation}\label{mth-BE-polynomial-moment}
			\E_t\{|G||W|^k\}
			\leq C_mt^k e^{2\kappa_m(t+(1+t)^{-1})}
			\bigl\{\E_{t+(1+t)^{-1}}|G|+\E_{t-(1+t)^{-1}}|G|\bigr\}.
		\end{equation}
		By \eqref{mth-cumulant-main-sum}, $|\Gamma'(t)|\leq C_mt$.  Hence $|Y_{m,t}|^k\leq C_m(|W|^k+t^k)$, and the same exponential-shift calculation shows that \eqref{mth-BE-polynomial-moment} also holds with $|W|^k$ replaced by $|Y_{m,t}|^k$. For the centered moments without a random weight, Lemma \ref{lemma-mth-centered-exp} applies directly. Since $Y_{m,t}=W-\Gamma'(t)$, it gives
		\begin{equation}\label{mth-BE-centered-moment}
			\E_t|Y_{m,t}|^k
			\leq C_m\exp\{2[\eta_m(t+1)+\kappa_m(t+1)]\},\qquad 1\leq k\leq m-1.
		\end{equation}
		Since the interval indicator is bounded by one, \eqref{mth-BE-centered-moment} implies, for $1\leq k\leq m-1$,
		\begin{align}\label{mth-BE-polynomial-indicator}
			\E_t\{|Y_{m,t}|^k\I(z-u\leq Y_{m,t}\leq z+\rho_m-u)\}
			\leq C_m\exp\{2[\eta_m(t+1)+\kappa_m(t+1)]\}.
		\end{align}
		For $k=0$, the interval probability is bounded by \eqref{mth-BE-small-interval}, uniformly over its location and hence also after the shift by $u$.
		
		Substituting $w=Y_{m,t}$ into \eqref{mth-BE-stein-equation} and taking expectations under $\p_t$, we obtain
		\begin{equation}\label{mth-BE-identity}
			\E_t h_{m,z}(Y_{m,t})-\E h_{m,z}(\sqrt{\Gamma''(t)}Z)
			=\E_t\{\Gamma''(t)f_{m,z}'(Y_{m,t})-Y_{m,t}f_{m,z}(Y_{m,t})\}.
		\end{equation}
		Now substitute $g(w)=f_{m,z}(w-\Gamma'(t))$ into \eqref{mth-conjugate-stein}.  Using the definition of $Y_{m,t}$ expands the right-hand side of \eqref{mth-BE-identity} as
		\begin{align}\label{mth-BE-expanded-identity}
			\E_t h_{m,z}(Y_{m,t})-\E h_{m,z}(\sqrt{\Gamma''(t)}Z)
			={}&
			-\sum_{j=3}^{m}\frac{\Gamma_j}{(j-1)!}
			\E_t\sum_{\ell=2}^{j-1}\binom{j-1}{\ell}t^{j-1-\ell}f_{m,z}^{(\ell)}(Y_{m,t}) \\*
			&-\sum_{j=1}^{m}\E_t\biggl\{
			\mathcal R_{m,j}\sum_{\ell=0}^{j-1}\binom{j-1}{\ell}t^{j-1-\ell}f_{m,z}^{(\ell)}(Y_{m,t})\biggr\} \notag\\*
			&-\E_t\int_{|u|\leq\rho_m}e^{tu}
			\sum_{\ell=0}^{m}\binom m\ell t^{m-\ell}f_{m,z}^{(\ell)}(Y_{m,t}+u)\mathcal K_{m,1}(u)du \notag\\*
			&-\E_t\int_{|u|\leq\rho_m}e^{tUu}
			\sum_{\ell=0}^{m}\binom m\ell t^{m-\ell}f_{m,z}^{(\ell)}(Y_{m,t}+Uu)\mathcal K_{m,2}(u)du. \notag
		\end{align}
		Here the terms with $\ell=0$ in the cumulant sum cancel with $(\Gamma'(t)-t)f_{m,z}(Y_{m,t})$ coming from $(W-t)g(W)=\{Y_{m,t}+\Gamma'(t)-t\}f_{m,z}(Y_{m,t})$, while the terms with $\ell=1$ cancel exactly with $(\Gamma''(t)-1)\E_tf_{m,z}'(Y_{m,t})$ on the right-hand side of \eqref{mth-BE-identity}. The first line on the right-hand side of \eqref{mth-BE-expanded-identity} contains only derivatives $f_{m,z}^{(\ell)}(Y_{m,t})$, $2\leq\ell\leq m-1$, with coefficients bounded by $C_m\Gamma_{j,0}t^{j-1-\ell}$.  By \eqref{mth-BE-derivative-bound}, \eqref{mth-BE-centered-moment}, \eqref{mth-BE-polynomial-indicator}, and \eqref{mth-BE-small-interval}, it follows that
		\begin{align}\label{mth-BE-cumulant-bound}
			&\biggl|\sum_{j=3}^{m}\frac{\Gamma_j}{(j-1)!}
			\E_t\sum_{\ell=2}^{j-1}\binom{j-1}{\ell}
			t^{j-1-\ell}f_{m,z}^{(\ell)}(Y_{m,t})\biggr|\\
			\leq{}&C_m
			\sum_{j=3}^{m}\Gamma_{j,0}
			\sum_{\ell=2}^{j-1}t^{j-1-\ell}
			\biggl[
			e^{2\{\eta_m(t+1)+\kappa_m(t+1)\}}
			\biggl(1+\sum_{1\leq r\leq\ell-2}\rho_m^{-r}\biggr)
			+\rho_m^{1-\ell}\tilde{\Delta}_m(t)
			\biggr]. \notag
		\end{align}
		For the scalar remainder terms in \eqref{mth-BE-expanded-identity}, we use \eqref{mth-BE-basic-Stein-bounds} and \eqref{A2-moment} for derivatives of orders zero and one.  For the higher derivatives, we use \eqref{mth-BE-derivative-bound}, bound the interval indicators by one, apply \eqref{mth-BE-polynomial-moment} with $G=\mathcal R_{m,j}$, and then use \eqref{A2-moment} at $t\pm(1+t)^{-1}$.  Since $t+(1+t)^{-1}\leq t+1$ and $\kappa_m(s)$ and $r_{m,j}(s)$ are nondecreasing in $s\in[0,t+1]$, we obtain
		\begin{align}\label{mth-BE-remainder-bound}
			&\sum_{j=1}^{m}\left|\E_t\biggl\{
			\mathcal R_{m,j}\sum_{\ell=0}^{j-1}\binom{j-1}{\ell}
			t^{j-1-\ell}f_{m,z}^{(\ell)}(Y_{m,t})\biggr\}\right|\\
			\leq{}&C_me^{2\kappa_m(t+1)}
			\sum_{j=1}^{m}r_{m,j}(t+1)
			\biggl[
			jt^{j-1}
			+\sum_{r=1}^{j-2}(j-1-r)
			\rho_m^{-r}t^{j-2-r}
			\biggr]. \notag
		\end{align}
		It remains to estimate the two kernel integrals.  We present the estimate for the integral involving $\mathcal K_{m,1}(u)$. The integral involving $\mathcal K_{m,2}(u)$ is bounded in the same way.  Applying \eqref{mth-BE-derivative-bound}, \eqref{mth-BE-polynomial-indicator}, and \eqref{mth-BE-small-interval}, and then using \eqref{A3-kernel}, we obtain
		\begin{align}\label{mth-BE-kernel-bound}
			&\left|\E_t\int_{|u|\leq\rho_m}e^{tu}
			\sum_{\ell=0}^{m}\binom m\ell t^{m-\ell}
			f_{m,z}^{(\ell)}(Y_{m,t}+u)\mathcal K_{m,1}(u)du\right|\\ \notag
			\leq{}&
			C_me^{2\kappa_m(t+1)}\vartheta_m t^m
			+C_m\vartheta_m
			\sum_{\ell=2}^{m}t^{m-\ell}
			\biggl[
			e^{2\{\eta_m(t+1)+\kappa_m(t+1)\}}
			\sum_{1\leq r\leq\ell-2}\rho_m^{-r}
			+\rho_m^{1-\ell}\tilde{\Delta}_m(t)
			\biggr].
		\end{align}
		The nonempty range in Proposition \ref{prop-mth-tilted-concentration} implies $0<\rho_m\leq1$, while \eqref{mth-Delta-tilde-definition} yields $\tilde{\Delta}_m(t)\geq\rho_m e^{2\{\eta_m(t+1)+\kappa_m(t+1)\}}$. Consequently, for $2\leq\ell\leq m$,
		\[
		e^{2\{\eta_m(t+1)+\kappa_m(t+1)\}}
		\biggl(1+\sum_{1\leq r\leq\ell-2}\rho_m^{-r}\biggr)
		\leq C_m\tilde{\Delta}_m(t)
		\rho_m^{1-\ell}.
		\]
		Thus, in \eqref{mth-BE-cumulant-bound}, the term not containing $\tilde{\Delta}_m(t)$ is bounded by the term containing $\rho_m^{1-\ell}\tilde{\Delta}_m(t)$.  The same comparison applies to the two terms in brackets in \eqref{mth-BE-kernel-bound}.  Substituting \eqref{mth-BE-cumulant-bound}, \eqref{mth-BE-remainder-bound}, and \eqref{mth-BE-kernel-bound} into \eqref{mth-BE-expanded-identity} bounds the first term on the right-hand side of \eqref{mth-BE-smoothing-reduction}.  The interval-probability and smoothing terms in \eqref{mth-BE-smoothing-reduction} are bounded by \eqref{mth-BE-small-interval} and $\rho_m\leq\tilde{\Delta}_m(t)$, respectively.  Since $e^{2\kappa_m(t+1)}\geq1$, both are included in the first term of \eqref{mth-Delta-definition}.  This proves \eqref{mth-tilted-BE}.
	\end{proof}

\begin{proof}[Proof of Lemma \ref{prop-mth-original-BE}]
		Since the left-hand side of \eqref{mth-original-BE} is at most one, it suffices to consider sufficiently small $\Delta_m(0)$. The threshold depends only on $m$, and increasing $C_m$ covers the remaining case. Under this restriction, the definitions \eqref{mth-Delta-tilde-definition}, \eqref{mth-Delta-definition}, and \eqref{mth-concentration-length} then imply
		\begin{equation}\label{mth-original-smallness-bounds}
			0<\rho_m\leq\epsilon_m\leq1,\qquad
			\vartheta_m\leq1,\qquad
			\eta_m(1)+\kappa_m(1)
			\leq C_m\Delta_m(0)\leq1.
		\end{equation}

		With these smallness bounds, we follow the smoothing argument of Proposition \ref{prop-mth-tilted-BE} under the original measure. The first step is the following concentration inequality, obtained with the smoothing width $\epsilon_m$ from \eqref{mth-concentration-length}:
		\begin{equation}\label{mth-original-concentration}
			\sup_{z\in\mathbb R}\p(z\leq W\leq z+\rho_m)
			\leq C_m\tilde\Delta_m(0).
		\end{equation}
		We follow the argument of Proposition \ref{prop-mth-tilted-concentration}, except that we apply the original-measure Stein identity \eqref{mth-Stein} directly. For $r\geq0$, let $L_0(r)\coloneqq\sup_{z\in\mathbb R}\p(z\leq W\leq z+r)$. Fix $a<b$, and let $\mathfrak f_m(w)$ be the smoothing function defined in the proof of Proposition \ref{prop-mth-tilted-concentration}.  By \eqref{mth-smooth-basic-bounds},
		\begin{equation}\label{mth-original-concentration-basic}
			\p(a\leq W\leq b)\leq\E\mathfrak f_m'(W),
			\qquad
			\left|\E\{W\mathfrak f_m(W)\}\right|
			\leq\frac{b-a+\epsilon_m}{2}\E|W|
			\leq\frac{b-a+\epsilon_m}{2}.
		\end{equation}
		Moreover, by {\rm (A2)}, {\rm (A4)}, \eqref{mth-smooth-basic-bounds}, and \eqref{mth-smooth-derivative-bounds}, we have
		\begin{align}
			\sum_{j=3}^{m}\frac{|\Gamma_j|}{(j-1)!}
			\E|\mathfrak f_m^{(j-1)}(W)|
			&\leq
			2A_m\sum_{j=3}^{m}
			\frac{\Gamma_{j,0}\epsilon_m^{2-j}}{(j-1)!}
			L_0(\epsilon_m),\notag\\
			\sum_{j=1}^{m}
			\E\{|\mathcal R_{m,j}|
			|\mathfrak f_m^{(j-1)}(W)|\}&\leq
			\frac{b-a+\epsilon_m}{2}r_{m,1}(0)+r_{m,2}(0)
			+A_m\sum_{j=3}^{m}r_{m,j}(0)\epsilon_m^{2-j}.
			\label{mth-original-concentration-components}
		\end{align}
		Furthermore, using the support argument leading to \eqref{mth-kernel-iteration}, with $t=0$ and $L_t(r)$ replaced by $L_0(r)$, it follows that
		\begin{align}\label{mth-original-concentration-kernel}
			&\left|\E\int_{|u|\leq\rho_m}
			\mathfrak f_m^{(m)}(W+u)\mathcal K_{m,1}(u)du\right|
			+\left|\E\int_{|u|\leq\rho_m}
			\mathfrak f_m^{(m)}(W+Uu)\mathcal K_{m,2}(u)du\right| \notag\\
			\leq{}&4A_m\vartheta_m\epsilon_m^{1-m}L_0(\epsilon_m).
		\end{align}
		Therefore, applying \eqref{mth-Stein} to $\mathfrak f_m(w)$ and using \eqref{mth-original-concentration-basic}, \eqref{mth-original-concentration-components}, and \eqref{mth-original-concentration-kernel}, we obtain
		\begin{align}\label{mth-original-concentration-recursion}
			\p(a\leq W\leq b)
			\leq{}&C_m\biggl[
			(b-a+\epsilon_m)\{1+r_{m,1}(0)\}
			+r_{m,2}(0)
			+\sum_{j=3}^{m}r_{m,j}(0)\epsilon_m^{2-j}
			\biggr]\\
			&+\biggl[
			2A_m\sum_{j=3}^{m}
			\frac{\Gamma_{j,0}\epsilon_m^{2-j}}{(j-1)!}
			+4A_m\vartheta_m\epsilon_m^{1-m}
			\biggr]L_0(\epsilon_m).\notag
		\end{align}
		By \eqref{mth-concentration-length}, for $3\leq j\leq m$, $\Gamma_{j,0}\epsilon_m^{2-j}\leq(2^{m+4}A_m)^{2-j}$ and $\vartheta_m\epsilon_m^{1-m}\leq(2^{m+4}A_m)^{1-m}$. Consequently, the coefficient of $L_0(\epsilon_m)$ in \eqref{mth-original-concentration-recursion} is bounded by
		\[
			2A_m\sum_{j=3}^{m}
			\frac{(2^{m+4}A_m)^{2-j}}{(j-1)!}
			+4A_m(2^{m+4}A_m)^{1-m}<\frac12,
		\]
		where we used $m\geq3$ and $A_m\geq1$.  With $b=a+\epsilon_m$, taking the supremum over $a$ in \eqref{mth-original-concentration-recursion} and absorbing the $L_0(\epsilon_m)$-term yields
		\[
			L_0(\epsilon_m)
			\leq C_m\biggl[
			\epsilon_m\{1+r_{m,1}(0)\}+r_{m,2}(0)
			+\sum_{j=3}^{m}r_{m,j}(0)\epsilon_m^{2-j}
			\biggr]
			\leq C_m\tilde\Delta_m(0).
		\]
		By \eqref{mth-concentration-length}, \eqref{mth-Delta-tilde-definition}, and \eqref{mth-original-smallness-bounds}, \eqref{mth-original-concentration} follows.

		The concentration inequality \eqref{mth-original-concentration} controls the error from smoothing the indicator, leaving a normal Stein expectation to estimate. For fixed $z\in\mathbb R$, take the smoothed indicator $h_{m,z}(w)$ from the proof of Proposition \ref{prop-mth-tilted-BE}. Let $f_{m,z}(w)$ solve \eqref{mth-BE-stein-equation} with $t=0$. Since $\Gamma''(0)=1$, applying \eqref{mth-BE-smoothing-reduction} at $t=0$ and using \eqref{mth-original-concentration} and $\rho_m\leq\tilde\Delta_m(0)$ yields
		\begin{equation}\label{mth-original-BE-smoothing}
			|\p(W\leq z)-\Phi(z)|
			\leq
			|\E h_{m,z}(W)-\E h_{m,z}(Z)|
			+C_m\tilde\Delta_m(0),
			\qquad Z\sim N(0,1).
		\end{equation}
		It therefore suffices to bound $|\E h_{m,z}(W)-\E h_{m,z}(Z)|$. At $t=0$, \eqref{mth-Stein} with $f(w)=f_{m,z}(w)$ and \eqref{mth-BE-stein-equation} imply
		\begin{align}\label{mth-original-expanded-identity}
			&\E h_{m,z}(W)-\E h_{m,z}(Z) \notag\\
			={}&-\sum_{j=3}^{m}\frac{\Gamma_j}{(j-1)!}
			\E f_{m,z}^{(j-1)}(W)
			-\sum_{j=1}^{m}\E\{\mathcal R_{m,j}f_{m,z}^{(j-1)}(W)\}\\
			&-\E\int_{|u|\leq\rho_m}
			f_{m,z}^{(m)}(W+u)\mathcal K_{m,1}(u)du-\E\int_{|u|\leq\rho_m}
			f_{m,z}^{(m)}(W+Uu)\mathcal K_{m,2}(u)du.\notag
		\end{align}
		For the first term on the right-hand side of \eqref{mth-original-expanded-identity}, Proposition \ref{prop-mth-mgf} and \eqref{mth-original-smallness-bounds} imply, for $0\leq k\leq m-1$,
		\begin{equation}\label{mth-original-polynomial-moments}
			1+|w|^k\leq C_m(e^w+e^{-w}),\qquad
			\E(1+|W|^k)\leq C_m.
		\end{equation}
		Since $\Gamma''(0)=1$, \eqref{mth-BE-basic-Stein-bounds} and \eqref{mth-BE-derivative-bound} hold for $f_{m,z}(w)$ at $t=0$. Taking expectations in \eqref{mth-BE-derivative-bound} and using \eqref{mth-original-polynomial-moments}, \eqref{mth-original-concentration}, and \eqref{mth-Delta-tilde-definition}, we obtain, for $2\leq\ell\leq m$,
		\begin{align}
			\E|f_{m,z}^{(\ell)}(W)|
			&\leq C_m\rho_m^{2-\ell}
			+C_m\rho_m^{1-\ell}\p(z\leq W\leq z+\rho_m) \notag\\
			&\leq C_m\bigl\{\rho_m^{2-\ell}
			+\rho_m^{1-\ell}\tilde\Delta_m(0)\bigr\}
			\leq C_m\rho_m^{1-\ell}\tilde\Delta_m(0).
			\label{mth-original-Stein-derivatives}
		\end{align}
		Then, by \eqref{mth-original-Stein-derivatives}, it follows that
		\begin{equation}\label{mth-original-cumulant-bound}
			\sum_{j=3}^{m}\frac{|\Gamma_j|}{(j-1)!}
			\E |f_{m,z}^{(j-1)}(W)|\leq
			C_m\tilde\Delta_m(0)
			\sum_{j=3}^{m}\Gamma_{j,0}\rho_m^{2-j}.
		\end{equation}
		For the second term on the right-hand side of \eqref{mth-original-expanded-identity}, \eqref{mth-BE-basic-Stein-bounds} and \eqref{A2-moment} bound the summands with $j=1,2$ by $C_m\{r_{m,1}(1)+r_{m,2}(1)\}$.  For $3\leq j\leq m$, we discard the interval indicators in \eqref{mth-BE-derivative-bound}.  The first inequality in \eqref{mth-original-polynomial-moments}, Proposition \ref{prop-mth-mgf}, \eqref{A2-moment} and \eqref{mth-original-smallness-bounds} then yield
		\begin{equation}\label{mth-original-remainder-bound}
			\sum_{j=1}^{m}
			\left|\E\{\mathcal R_{m,j}f_{m,z}^{(j-1)}(W)\}\right|
			\leq C_m\biggl[
			r_{m,1}(1)+r_{m,2}(1)
			+\sum_{j=3}^{m}r_{m,j}(1)\rho_m^{2-j}
			\biggr].
		\end{equation}
		For the last two terms on the right-hand side of \eqref{mth-original-expanded-identity}, the indicators corresponding to $r=m-1$ in \eqref{mth-BE-derivative-bound} restrict $W$ to an interval of length at most $3\rho_m$.  Thus \eqref{mth-BE-derivative-bound}, \eqref{mth-original-polynomial-moments}, \eqref{mth-original-concentration}, and \eqref{A3-kernel} at $t=0$ yield
		\begin{align}\label{mth-original-kernel-bound}
			&\left|\E\int_{|u|\leq\rho_m}
			f_{m,z}^{(m)}(W+u)\mathcal K_{m,1}(u)du\right|
			+\left|\E\int_{|u|\leq\rho_m}
			f_{m,z}^{(m)}(W+Uu)\mathcal K_{m,2}(u)du\right|\notag\\
			\leq& C_m\vartheta_m\rho_m^{2-m}
			+C_m\vartheta_m\rho_m^{1-m}
			\sup_{y\in\mathbb R}\p(y\leq W\leq y+3\rho_m)
			\leq
			C_m\vartheta_m\rho_m^{1-m}\tilde\Delta_m(0).
		\end{align}
		Therefore, combining \eqref{mth-original-BE-smoothing}, \eqref{mth-original-expanded-identity}, and \eqref{mth-original-cumulant-bound}-- \eqref{mth-original-kernel-bound}, we obtain
		\begin{align}\label{mth-original-BE-collected}
			|\p(W\leq z)-\Phi(z)|
			\leq{}&C_m\tilde\Delta_m(0)\biggl[
			1+\sum_{j=3}^{m}\Gamma_{j,0}\rho_m^{2-j}
			+\vartheta_m\rho_m^{1-m}
			\biggr] \notag\\
			&\quad+C_m\biggl[
			r_{m,1}(1)+r_{m,2}(1)
			+\sum_{j=3}^{m}r_{m,j}(1)\rho_m^{2-j}
			\biggr]\notag\\
			\leq{}&C_m\Delta_m(0),
		\end{align}
		where the second inequality in \eqref{mth-original-BE-collected} follows from \eqref{mth-Delta-definition}, since, at $t=0$,
		\begin{align*}
			\Delta_m(0)\geq
			\tilde\Delta_m(0)\biggl[
			1+\sum_{j=3}^{m}\Gamma_{j,0}\rho_m^{2-j}
			+\vartheta_m\rho_m^{1-m}
			\biggr]+r_{m,1}(1)+r_{m,2}(1)
			+\sum_{j=3}^{m}r_{m,j}(1)\rho_m^{2-j}.
		\end{align*}
		Taking the supremum over $z\in\mathbb R$ proves \eqref{mth-original-BE}.
	\end{proof}

\appendix
\section*{\fontsize{14}{16}\selectfont Appendix}

\section{Proof of Lemma \ref{lemma-local-fourth-stein}}\label{appB}

The proofs of Lemmas \ref{lemma-local-fourth-stein}--\ref{lemma-cw-fourth-stein} follow the same two steps. Step 1 constructs the higher-order Stein identity, and Step 2 relates its coefficients to the cumulants and establishes the stated bounds.

\begin{proof}[Proof of Lemma \ref{lemma-local-fourth-stein}]
\par\smallskip\noindent\textit{Step 1.}\quad We first write the Taylor remainders in a common kernel form: for $0\leq r\leq3$, define $T_{r,y}(u)\coloneqq(-1)^{3-r}(u+y)^{3-r}H_y(u)/(3-r)!$. Taylor's expansion with integral remainder gives the full expansion
\begin{equation}\label{local-direct-taylor}
 f^{(r)}(w)-f^{(r)}(w-y)
 =\sum_{q=1}^{3-r}\frac{(-1)^{q+1}y^q}{q!}f^{(r+q)}(w)
 +\int_{|u|\leq\kappa\delta}f^{(4)}(w+u)T_{r,y}(u)du,
 \qquad |y|\leq\kappa\delta.
\end{equation}
In particular, $T_{3,y}(u)=H_y(u)$.  Each kernel is supported between $-y$ and $0$, and
\begin{equation}\label{local-taylor-kernel-mass}
 \int_{|u|\leq\kappa\delta}|T_{r,y}(u)|du
 =\frac{|y|^{4-r}}{(4-r)!}.
\end{equation}
Each neighborhood sum in \eqref{local-neighborhood-sums} is bounded in absolute value by $\kappa\delta$, so \eqref{local-direct-taylor} applies to every neighborhood shift. For each centered coefficient, local independence allows us to subtract a test function evaluated at the sum outside the next neighborhood. The coefficients below collect the terms produced by Taylor's expansion of these differences. In the following sums, the indices range over $i\in\mathcal J$, $j\in A_i$, $k\in A_{ij}$ and $l\in A_{ijk}$, with unused trailing indices omitted. Define
\begin{align}
 Q_{ijk}&\coloneqq\{X_iX_j-\E(X_iX_j)\}X_k
 -\frac12 X_iX_jX_k\I(k\in A_i),\label{local-triple-coefficient}\displaybreak[1]\\
 V_{ijkl}&\coloneqq\{Q_{ijk}-\E Q_{ijk}\}X_l
 -\frac12\{X_iX_j-\E(X_iX_j)\}X_kX_l\I(l\in A_{ij}) \notag\\
 &\quad+\frac16 X_iX_jX_kX_l\I(k,l\in A_i).
 \label{local-quadruple-coefficient}
\end{align}
The resulting fourth-order remainder kernel is
\begin{align}\label{local-fourth-kernel-definitions}
 \mathcal K_{4,1}(u)\coloneqq{}&\sum_i X_iT_{0,Y_i}(u)
 +\sum_{i,j}\{X_iX_j-\E(X_iX_j)\}T_{1,Y_{ij}}(u)\notag\\*
 &+\sum_{i,j,k}\{Q_{ijk}-\E Q_{ijk}\}T_{2,Y_{ijk}}(u)
 +\sum_{i,j,k,l}\{V_{ijkl}-\E V_{ijkl}\}H_{Y_{ijkl}}(u).
\end{align}
To verify \eqref{local-fourth-kernel-definitions}, first use {\rm (LD1)} and $\E X_i=0$ to write
\begin{equation}\label{local-basic-stein}
 \E\{Wf(W)\}
 =\sum_i\E\bigl[X_i\{f(W)-f(W-Y_i)\}\bigr].
\end{equation}
Apply \eqref{local-direct-taylor} with $r=0$ in \eqref{local-basic-stein}. The coefficient of $f'(W)$ is $\sum_{i,j}X_iX_j$, whose expectation is $\E W^2=1$. Condition {\rm (LD$^*$2)} gives
\begin{equation}\label{local-pair-centering}
 \E\bigl[\{X_iX_j-\E(X_iX_j)\}f'(W)\bigr]
 =\E\bigl[\{X_iX_j-\E(X_iX_j)\}\{f'(W)-f'(W-Y_{ij})\}\bigr].
\end{equation}
Expanding \eqref{local-pair-centering} by \eqref{local-direct-taylor} with $r=1$ gives the coefficient $\sum_{i,j,k}Q_{ijk}$ of $f''(W)$. The remaining centered coefficients satisfy
\begin{align}
 \E\bigl[\{Q_{ijk}-\E Q_{ijk}\}f''(W)\bigr]
 &=\E\bigl[\{Q_{ijk}-\E Q_{ijk}\}
 \{f''(W)-f''(W-Y_{ijk})\}\bigr],\label{local-triple-centering}\displaybreak[1]\\
 \E\bigl[\{V_{ijkl}-\E V_{ijkl}\}f^{(3)}(W)\bigr]
 &=\E\bigl[\{V_{ijkl}-\E V_{ijkl}\}
 \{f^{(3)}(W)-f^{(3)}(W-Y_{ijkl})\}\bigr].\label{local-quadruple-centering}
\end{align}
Here \eqref{local-triple-centering} and \eqref{local-quadruple-centering} follow from {\rm (LD$^*$3)} and {\rm (LD$^*$4)}, respectively. Applying \eqref{local-direct-taylor} with $r=2$ in \eqref{local-triple-centering} gives the coefficient $\sum_{i,j,k,l}V_{ijkl}$ of $f^{(3)}(W)$. Finally, use $r=3$ in \eqref{local-quadruple-centering} and collect the four integral remainders to obtain
\begin{align}\label{local-fourth-stein-pre}
 \E\{Wf(W)\}
 &=\E f'(W)+\biggl(\sum_{i,j,k}\E Q_{ijk}\biggr)\E f''(W)
 +\biggl(\sum_{i,j,k,l}\E V_{ijkl}\biggr)\E f^{(3)}(W)\notag\\*
 &\quad+\E\int_{|u|\leq\kappa\delta}f^{(4)}(W+u)\mathcal K_{4,1}(u)du.
\end{align}
\par\smallskip\noindent\textit{Step 2.}\quad Taking $f(w)=w^2$ and $f(w)=w^3$ in \eqref{local-fourth-stein-pre} identifies the coefficients as
\begin{equation}\label{local-cumulant-coefficients}
 \Gamma_3=2\sum_{i,j,k}\E Q_{ijk},\qquad
 \Gamma_4=6\sum_{i,j,k,l}\E V_{ijkl}.
\end{equation}
Substituting \eqref{local-cumulant-coefficients} into \eqref{local-fourth-stein-pre} proves \eqref{local-fourth-stein}.

It remains to bound the coefficients and kernel. Boundedness gives $|X_iX_j-\E(X_iX_j)|\leq2\delta^2$, $|Q_{ijk}|\leq3\delta^3$ and $|V_{ijkl}|\leq8\delta^4$. The respective numbers of summands are at most $n\kappa$, $n\kappa^2$ and $n\kappa^3$. Consequently, \eqref{local-cumulant-coefficients} gives the cumulant bounds in \eqref{local-fourth-size-bounds}. Since every neighborhood sum has absolute value at most $\kappa\delta$, the support assertion follows from \eqref{local-fourth-kernel-definitions}. Moreover, \eqref{local-taylor-kernel-mass} yields the pathwise bound
\begin{equation}\label{local-kernel-size-bound}
 \int_{|u|\leq\kappa\delta}|\mathcal K_{4,1}(u)|du
 \leq C\{n\delta(\kappa\delta)^4+n\kappa\delta^2(\kappa\delta)^3
 +n\kappa^2\delta^3(\kappa\delta)^2+n\kappa^3\delta^4\kappa\delta\}
 \leq Cn\kappa^4\delta^5.
\end{equation}
This proves \eqref{local-fourth-size-bounds}.
\end{proof}

		\Needspace{6\baselineskip}
		\section{Proof of Lemma \ref{lemma-cclt-third-stein}}\label{appC}

		\begin{proof}[Proof of Lemma \ref{lemma-cclt-third-stein}]
			\leavevmode
			\par\smallskip\noindent\textit{Step 1.}\quad To expand the kernel-mass fluctuations, we need coupled permutations with one, two, or three prescribed matches. We construct them using the prescribed-pair coupling of Lemma 7.1 in \citet{LZ2023}. Throughout the proof, we retain the notation of Sections \ref{subsubsec-third-order-Stein-derivation} and \ref{subsec-proof-theorem33}, and write $b_{i,j}\coloneqq\E X_{i,j}^{2}$. Denote by $\tau_{i,j}$ the transposition of $i$ and $j$, and read compositions from right to left. For one prescribed pair $(i,j)$, define
			\[
			\pi_{i,j}\coloneqq
			\begin{cases}
				\pi, & \text{if }\pi(i)=j,\\
				\pi\circ\tau_{i,\pi^{-1}(j)}, & \text{if }\pi(i)\neq j.
			\end{cases}
			\]
			Then $\pi_{i,j}(i)=j$.  For two prescribed pairs $\mathbf i=(i_1,i_2)\in[n]_2$ and $\mathbf j=(j_1,j_2)\in[n]_2$, define
			\[
			\pi_{\mathbf i,\mathbf j}\coloneqq
			\begin{cases}
				\pi, & \text{if }\pi(i_1)=j_1,\ \pi(i_2)=j_2,\\
				\pi\circ\tau_{i_1,\pi^{-1}(j_1)}, & \text{if }\pi(i_1)\neq j_1,\ \pi(i_2)=j_2,\\
				\pi\circ\tau_{i_2,\pi^{-1}(j_2)}, & \text{if }\pi(i_1)=j_1,\ \pi(i_2)\neq j_2,\\
				\pi\circ\tau_{i_1,\pi^{-1}(j_2)}\circ\tau_{i_2,\pi^{-1}(j_1)}
				\circ\tau_{i_1,i_2}, & \text{if }\pi(i_1)\neq j_1,\ \pi(i_2)\neq j_2.
			\end{cases}
			\]
			Then $\pi_{\mathbf i,\mathbf j}(i_1)=j_1$ and $\pi_{\mathbf i,\mathbf j}(i_2)=j_2$.  For three prescribed pairs $\mathbf r=(r_1,r_2,r_3)\in[n]_3$ and $\mathbf s=(s_1,s_2,s_3)\in[n]_3$, define the sequence recursively by $\pi^{(0)}\coloneqq\pi$ and, for $q=1,2,3$,
			\[
			\pi^{(q)}\coloneqq
			\begin{cases}
				\pi^{(q-1)}, & \text{if }\pi^{(q-1)}(r_q)=s_q,\\
				\pi^{(q-1)}\circ\tau_{r_q,(\pi^{(q-1)})^{-1}(s_q)},
				& \text{if }\pi^{(q-1)}(r_q)\neq s_q.
			\end{cases}
			\]
			Define $\pi_{\mathbf r,\mathbf s}\coloneqq\pi^{(3)}$.  Then $\pi_{\mathbf r,\mathbf s}(r_q)=s_q$ for $q=1,2,3$.  With these constructions,
			\begin{equation}\label{cclt-prescribed-pair-laws}
			\begin{aligned}
				\mathcal L(\pi_{i,j})&=\mathcal L\bigl(\pi\mid\pi(i)=j\bigr),\qquad
				\mathcal L(\pi_{\mathbf i,\mathbf j})=\mathcal L\bigl(\pi\mid\pi(i_q)=j_q,\ 1\leq q\leq2\bigr),\\
				\mathcal L(\pi_{\mathbf r,\mathbf s})&=\mathcal L\bigl(\pi\mid\pi(r_q)=s_q,\ 1\leq q\leq3\bigr).
			\end{aligned}
			\end{equation}
			Here $\mathcal L(\cdot)$ denotes the probability distribution.
			
			For one prescribed pair $(i,j)$, define $W_{i,j}\coloneqq\sum_{k=1}^{n}X_{k,\pi_{i,j}(k)}$ and $W_{i,j}^{(i)}\coloneqq W_{i,j}-X_{i,j}$. Then $\mathcal L(W_{i,j})=\mathcal L(W\mid \pi(i)=j)$, and $X_{i,j}$ is independent of $W_{i,j}^{(i)}$.  For $\mathbf i=(i_1,i_2)\in[n]_2$ and $\mathbf j=(j_1,j_2)\in[n]_2$, define
			\[
			W_{\mathbf i,\mathbf j}\coloneqq\sum_{k=1}^{n}X_{k,\pi_{\mathbf i,\mathbf j}(k)},\qquad
			W_{\mathbf i,\mathbf j}^{(\mathbf i)}\coloneqq W_{\mathbf i,\mathbf j}-X_{i_1,j_1}-X_{i_2,j_2}.
			\]
			Thus $\mathcal L(W_{\mathbf i,\mathbf j})=\mathcal L(W\mid \pi(i_1)=j_1,\pi(i_2)=j_2)$, and the selected variables $X_{i_1,j_1},X_{i_2,j_2}$ are independent of $W_{\mathbf i,\mathbf j}^{(\mathbf i)}$.  Define
			\[
			\Delta_{i,j}\coloneqq W-W_{i,j},\quad
			\Delta_{i,j}^{(i)}\coloneqq W-W_{i,j}^{(i)},\quad
			\Delta_{\mathbf i,\mathbf j}\coloneqq W-W_{\mathbf i,\mathbf j},\quad
			\Delta_{\mathbf i,\mathbf j}^{(\mathbf i)}\coloneqq W-W_{\mathbf i,\mathbf j}^{(\mathbf i)}.
			\]
			For the three-pair construction, define $W_{\mathbf r,\mathbf s}\coloneqq\sum_{k=1}^{n}X_{k,\pi_{\mathbf r,\mathbf s}(k)}$ and $\Delta_{\mathbf r,\mathbf s}\coloneqq W-W_{\mathbf r,\mathbf s}$. By \eqref{cclt-bounded}, the changes in $W$ caused by imposing one, two, or three prescribed pairs are bounded in absolute value by $\delta/3$, $2\delta/3$, and $\delta$, respectively.
			
			Using \eqref{cclt-prescribed-pair-laws}, the independence of $X_{i,j}$ from $W_{i,j}^{(i)}$, and the independence of $X_{i_1,j_1},X_{i_2,j_2}$ from $W_{\mathbf i,\mathbf j}^{(\mathbf i)}$, we now expand \eqref{cclt-basic-stein} further.  By \eqref{cclt-K},
			\[
			\hat L=\int_{|u|\leq\delta/3}\hat K(u)du-1
			=\frac1{4n}\sum_{\mathbf i\in[n]_2}D_{\mathbf i}^2-1.
			\]
			Furthermore,
			\[
			|\E \hat L|=|\E\{WR\}|
			=\frac1{n^2}\sum_{i=1}^{n}\sum_{j=1}^{n}c_{i,j}^{2}
			\leq \frac1n,
			\]
			and a direct expansion of $D_{\mathbf i}^{2}$ yields
			\begin{align}\label{cclt-D-square}
				\sum_{\mathbf i\in[n]_2}D_{\mathbf i}^{2}
				={}&2n\sum_{i=1}^{n}X_{i,\pi(i)}^2
				+2\sum_{i=1}^{n}\sum_{j=1}^{n}X_{i,j}^2+2\sum_{i_1=1}^{n}\sum_{i_2=1}^{n}
				\{X_{i_1,\pi(i_1)}X_{i_2,\pi(i_2)}
				+X_{i_1,\pi(i_2)}X_{i_2,\pi(i_1)}\}\notag\\
				&\quad-2\sum_{i_1=1}^{n}\sum_{i_2=1}^{n}
				\bigl\{X_{i_1,\pi(i_2)}X_{i_2,\pi(i_2)}
				+X_{i_2,\pi(i_1)}X_{i_2,\pi(i_2)}\notag\\
				&\hspace{5.5em}+X_{i_1,\pi(i_2)}X_{i_1,\pi(i_1)}
				+X_{i_2,\pi(i_1)}X_{i_1,\pi(i_1)}\bigr\}.
			\end{align}
			We now expand $\E\{\hat Lf'(W)\}$ by conditioning on the matches selected by each product in \eqref{cclt-D-square}. For the square term, since $X_{i,j}$ is independent of $W_{i,j}^{(i)}$, we have
			\begin{align}\label{cclt-one-pair-example}
				&\E\{(X_{i,\pi(i)}^2-\E X_{i,\pi(i)}^2)f(W)\} \notag\\
				={}&\frac1n\sum_{j=1}^{n}\E\{X_{i,j}^2f(W)\mid \pi(i)=j\}
				-\frac1n\sum_{j=1}^{n}b_{i,j}\E f(W)\\
				={}&\frac1n\sum_{j=1}^{n}\E\{X_{i,j}^2[f(W_{i,j})-f(W_{i,j}^{(i)})]\}
				-\frac1n\sum_{j=1}^{n}b_{i,j}\E\{f(W)-f(W_{i,j}^{(i)})\} \notag\\
				={}&\frac1n\sum_{j=1}^{n}\E\int_{|u|\leq\delta/3}
				f'(W+u)\{X_{i,j}^2[H_{\Delta_{i,j}^{(i)}}(u)-H_{\Delta_{i,j}}(u)]
				-b_{i,j}H_{\Delta_{i,j}^{(i)}}(u)\}du. \notag
			\end{align}
			For the term involving $X_{i_1,\pi(i_1)}X_{i_2,\pi(i_2)}$ in \eqref{cclt-D-square}, conditioning on $\{\pi(i_1)=j_1,\pi(i_2)=j_2\}$ yields
			\begin{align}\label{cclt-two-pair-direct}
				&\E\{[X_{i_1,\pi(i_1)}X_{i_2,\pi(i_2)}
				-\E X_{i_1,\pi(i_1)}X_{i_2,\pi(i_2)}]f(W)\}\\
				={}&\frac1{n(n-1)}\sum_{\mathbf j\in[n]_2}
				\E\!\int_{|u|\leq2\delta/3}\!f'(W+u)
				\{X_{i_1,j_1}X_{i_2,j_2}
				[H_{\Delta_{\mathbf i,\mathbf j}^{(\mathbf i)}}(u)-H_{\Delta_{\mathbf i,\mathbf j}}(u)] 
				-a_{i_1,j_1}a_{i_2,j_2}
				H_{\Delta_{\mathbf i,\mathbf j}^{(\mathbf i)}}(u)\}du. \notag
			\end{align}
			Similarly, we obtain
			\begin{align}\label{cclt-two-pair-cross}
				&\E\{[X_{i_1,\pi(i_2)}X_{i_2,\pi(i_1)}
				-\E X_{i_1,\pi(i_2)}X_{i_2,\pi(i_1)}]f(W)\}\\*
				={}&\frac1{n(n-1)}\sum_{\mathbf j\in[n]_2}
				\E\!\int_{|u|\leq2\delta/3}\!f'(W+u)
				\{X_{i_1,j_2}X_{i_2,j_1}
				[H_{\Delta_{\mathbf i,\mathbf j}^{(\mathbf i)}}(u)-H_{\Delta_{\mathbf i,\mathbf j}}(u)] 
				-a_{i_1,j_2}a_{i_2,j_1}
				H_{\Delta_{\mathbf i,\mathbf j}^{(\mathbf i)}}(u)\}du. \notag
			\end{align}
			The remaining four products in \eqref{cclt-D-square} are handled analogously.  In the corresponding versions of \eqref{cclt-two-pair-direct}--\eqref{cclt-two-pair-cross}, the random products are, respectively,
			\begin{equation}\label{cclt-two-pair-analogues}
			X_{i_1,j_2}X_{i_2,j_2},\quad
			X_{i_2,j_1}X_{i_2,j_2},\quad
			X_{i_1,j_2}X_{i_1,j_1},\quad
			X_{i_2,j_1}X_{i_1,j_1}.
			\end{equation}
			Each centering term is the corresponding product of means. Substituting \eqref{cclt-one-pair-example} and \eqref{cclt-two-pair-direct}--\eqref{cclt-two-pair-cross}, together with the four analogues in \eqref{cclt-two-pair-analogues}, into \eqref{cclt-D-square} and collecting the integral terms, we obtain
			\begin{equation}\label{cclt-L-K1}
				\E\{\hat Lf'(W)\}
				=\E\int_{|u|\leq2\delta/3}f''(W+u)\hat K_1(u)du
				+\E\{R_1f'(W)\},
			\end{equation}
			where\nopagebreak[4]
			\begin{align}\label{cclt-K1-explicit}
				\hat K_1(u)
				&=\frac{1}{2n}\sum_{i=1}^{n}\sum_{j=1}^{n}
				\{X_{i,j}^{2}[H_{\Delta_{i,j}^{(i)}}(u)-H_{\Delta_{i,j}}(u)]
				-b_{i,j}H_{\Delta_{i,j}^{(i)}}(u)\}\\
				&\quad+\frac{1}{2n^2(n-1)}
				\sum_{\mathbf i\in[n]_2}\sum_{\mathbf j\in[n]_2}
				\{B_{\mathbf i,\mathbf j}[H_{\Delta_{\mathbf i,\mathbf j}^{(\mathbf i)}}(u)
				-H_{\Delta_{\mathbf i,\mathbf j}}(u)]
				-\E[B_{\mathbf i,\mathbf j}]H_{\Delta_{\mathbf i,\mathbf j}^{(\mathbf i)}}(u)\}. \notag
			\end{align}
			Here
			\begin{align*}
				B_{\mathbf i,\mathbf j}
				&\coloneqq X_{i_1,j_1}X_{i_2,j_2}+X_{i_1,j_2}X_{i_2,j_1}
				-X_{i_1,j_2}X_{i_2,j_2}-X_{i_2,j_1}X_{i_2,j_2}-X_{i_1,j_2}X_{i_1,j_1}-X_{i_2,j_1}X_{i_1,j_1}.
			\end{align*}
			The remainder in \eqref{cclt-L-K1} is
			\begin{equation}\label{cclt-R1-explicit}
				R_1=\frac1{4n}\sum_{\mathbf i\in[n]_2}\E D_{\mathbf i}^{2}-1
				+\frac1{2n}\sum_{i=1}^{n}\sum_{j=1}^{n}(X_{i,j}^{2}-b_{i,j})
				-\frac1n\sum_{i=1}^{n}\{X_{i,\pi(i)}^{2}-\E X_{i,\pi(i)}^{2}\}.
			\end{equation}
			Furthermore, by \eqref{cclt-K1-explicit},
			\begin{align}\label{cclt-K1-integral}
				\int_{|u|\leq2\delta/3}\hat K_1(u)du
				&=\frac1{2n}\sum_{i=1}^{n}\sum_{j=1}^{n}
				\{X_{i,j}^{3}-b_{i,j}\Delta_{i,j}^{(i)}\}\\*
				&\quad+\frac{1}{2n^2(n-1)}
				\sum_{\mathbf i\in[n]_2}\sum_{\mathbf j\in[n]_2}
				\{B_{\mathbf i,\mathbf j}(X_{i_1,j_1}+X_{i_2,j_2})
				-\E[B_{\mathbf i,\mathbf j}]\Delta_{\mathbf i,\mathbf j}^{(\mathbf i)}\}.\notag
			\end{align}
			Here
			\begin{align}
				\Delta_{i,j}^{(i)}
				&=\I(\pi(i)\neq j)\{X_{i,\pi(i)}+X_{\pi^{-1}(j),j}
				-X_{\pi^{-1}(j),\pi(i)}\}
				+\I(\pi(i)=j)X_{i,\pi(i)} \notag\\
				&=X_{i,\pi(i)}+X_{\pi^{-1}(j),j}
				-X_{\pi^{-1}(j),\pi(i)},\label{cclt-Delta-one-explicit}\displaybreak[1]\\
				\Delta_{\mathbf i,\mathbf j}^{(\mathbf i)}
				&=\I(\pi(i_1)=j_1,\pi(i_2)=j_2)
				\{X_{i_1,\pi(i_1)}+X_{i_2,\pi(i_2)}\} \notag\\
				&\quad+\I(\pi(i_1)=j_1,\pi(i_2)\neq j_2)
				\bigl\{X_{i_1,\pi(i_1)}+X_{i_2,\pi(i_2)} \notag\\
				&\hspace{7em}+X_{\pi^{-1}(j_2),j_2}-X_{\pi^{-1}(j_2),\pi(i_2)}\bigr\}\notag\\
				&\quad+\I(\pi(i_1)\neq j_1,\pi(i_2)=j_2)
				\bigl\{X_{i_1,\pi(i_1)}+X_{i_2,\pi(i_2)} \notag\\
				&\hspace{7em}+X_{\pi^{-1}(j_1),j_1}-X_{\pi^{-1}(j_1),\pi(i_1)}\bigr\}\notag\\
				&\quad+\I(\pi(i_1)\neq j_1,\pi(i_2)\neq j_2)
				\bigl\{X_{i_1,\pi(i_1)}+X_{i_2,\pi(i_2)} \notag\\
				&\hspace{7em}
				+X_{\pi^{-1}(j_1),j_1}+X_{\pi^{-1}(j_2),j_2}
				-X_{\pi^{-1}(j_1),\pi(i_2)}
				-X_{\pi^{-1}(j_2),\pi(i_1)}\bigr\} \notag\\
				&=X_{i_1,\pi(i_1)}+X_{i_2,\pi(i_2)}
				+X_{\pi^{-1}(j_1),j_1}+X_{\pi^{-1}(j_2),j_2} \notag\\
				&\quad-X_{\pi^{-1}(j_1),\pi(i_2)}
				-X_{\pi^{-1}(j_2),\pi(i_1)} \notag\\
				&\quad+\I(\pi(i_1)=j_1\text{ or }\pi(i_2)=j_2)\bigl\{X_{\pi^{-1}(j_1),\pi(i_2)}
				+X_{\pi^{-1}(j_2),\pi(i_1)} \notag\\
				&\hspace{7em}-X_{\pi^{-1}(j_1),\pi(i_1)}
				-X_{\pi^{-1}(j_2),\pi(i_2)}\bigr\}. \label{cclt-Delta-two-explicit}
			\end{align}

			Set $\gamma_{3,1}=2\E\int_{|u|\leq2\delta/3}\hat K_1(u)du$. To obtain the next derivative in the Stein identity, we expand the centered part of $\int_{|u|\leq2\delta/3}\hat K_1(u)du$ in \eqref{cclt-K1-integral}. This is $\hat L_1$ from \eqref{third-L1-definition}. For the centered term involving $X_{i,j}^3$ in \eqref{cclt-K1-integral}, the law of total expectation yields
			\begin{align}\notag
				\E\{(X_{i,j}^3-\E X_{i,j}^3)h(W)\}
				&=\sum_{k=1}^{n}\E\{(X_{i,j}^3-\E X_{i,j}^3)h(W)\mid\pi(i)=k\}\p(\pi(i)=k)\\*
				&=\frac1n\E\{(X_{i,j}^3-\E X_{i,j}^3)h(W)\mid\pi(i)=j\}, \notag
			\end{align}
			where the second equality follows from the fact that $X_{i,j}$ is independent of $W$ given $\pi(i)\neq j$.  Since $X_{i,j}$ is independent of $W_{i,j}^{(i)}$, applying \eqref{difference-formula} to $h(W_{i,j})-h(W_{i,j}^{(i)})$ yields
			\begin{align}\notag
				\E\{(X_{i,j}^3-\E X_{i,j}^3)h(W_{i,j})\}
				=\E\int_{|u|\leq\delta/3}h'(W+u)(X_{i,j}^3-\E X_{i,j}^3)
				[H_{\Delta_{i,j}^{(i)}}(u)-H_{\Delta_{i,j}}(u)]du.
			\end{align}
			Therefore, we obtain
			\begin{align}\label{cclt-X3-sum}
				&\frac1{2n}\sum_{i=1}^{n}\sum_{j=1}^{n}
				\E\{(X_{i,j}^3-\E X_{i,j}^3)h(W)\}\\*
				={}&\frac1{2n^2}\sum_{i=1}^{n}\sum_{j=1}^{n}
				\E\int_{|u|\leq\delta/3}h'(W+u)(X_{i,j}^3-\E X_{i,j}^3)
				[H_{\Delta_{i,j}^{(i)}}(u)-H_{\Delta_{i,j}}(u)]du.\notag
			\end{align}
			For the centered term involving $\Delta_{i,j}^{(i)}$ in \eqref{cclt-K1-integral}, substitute \eqref{cclt-Delta-one-explicit} and apply the conditioning calculation used in \eqref{cclt-one-pair-example}:
			\begin{align}\label{cclt-Delta-center-expand}
				&\E\{[\E\Delta_{i,j}^{(i)}-\Delta_{i,j}^{(i)}]h(W)\}\\
				={}&\E\{[\E X_{i,\pi(i)}-X_{i,\pi(i)}]h(W)\}
				+\E\{[\E X_{\pi^{-1}(j),j}-X_{\pi^{-1}(j),j}]h(W)\} \notag\\
				&\quad+\E\{[X_{\pi^{-1}(j),\pi(i)}
				-\E X_{\pi^{-1}(j),\pi(i)}]h(W)\} \notag\\
				={}&\frac1n\sum_{l=1}^{n}
				\E\int_{|u|\leq\delta/3}h'(W+u)
				\{a_{i,l}H_{\Delta_{i,l}^{(i)}}(u)
				-X_{i,l}[H_{\Delta_{i,l}^{(i)}}(u)-H_{\Delta_{i,l}}(u)]\}du \notag\\
				&\quad+\frac1n\sum_{k=1}^{n}
				\E\int_{|u|\leq\delta/3}h'(W+u)
				\{a_{k,j}H_{\Delta_{k,j}^{(k)}}(u)
				-X_{k,j}[H_{\Delta_{k,j}^{(k)}}(u)-H_{\Delta_{k,j}}(u)]\}du \notag\\
				&\quad+\E\{[X_{\pi^{-1}(j),\pi(i)}
				-\E X_{\pi^{-1}(j),\pi(i)}]h(W)\}. \notag
			\end{align}
			The last term in \eqref{cclt-Delta-center-expand} is expanded as
			\begin{align}\label{cclt-last-cross-expand}
				&\E\{[X_{\pi^{-1}(j),\pi(i)}
				-\E X_{\pi^{-1}(j),\pi(i)}]h(W)\}\\
				={}&\frac1n\E\{X_{i,j}h(W_{i,j})\}
				+\frac1{n(n-1)}
				\sum_{\substack{k=1\\ k\neq i}}^{n}
				\sum_{\substack{l=1\\ l\neq j}}^{n}
				a_{k,l}\E h(W_{(i,k),(l,j)})
				-\frac{a_{i,j}}{n-1}\E h(W) \notag\\
				={}&\frac1n\E\int_{|u|\leq\delta/3}h'(W+u)
				\{X_{i,j}[H_{\Delta_{i,j}^{(i)}}(u)-H_{\Delta_{i,j}}(u)]
				-a_{i,j}H_{\Delta_{i,j}^{(i)}}(u)\}du \notag\\
				&\quad-\frac1{n(n-1)}
				\sum_{\substack{k=1\\ k\neq i}}^{n}
				\sum_{\substack{l=1\\ l\neq j}}^{n}a_{k,l}
				\E\int_{|u|\leq2\delta/3}h'(W+u)H_{\Delta_{(i,k),(l,j)}}(u)du. \notag
			\end{align}
			
			The six products in $B_{\mathbf i,\mathbf j}$ are treated in the same way.  We present the details for the centered term corresponding to $X_{i_1,j_1}X_{i_2,j_2}(X_{i_1,j_1}+X_{i_2,j_2})$.  Let $\mathbf l=(l_1,l_2)$.  Then
			\begin{align}\label{cclt-two-pair-cubic}
				&\E\{[X_{i_1,j_1}X_{i_2,j_2}(X_{i_1,j_1}+X_{i_2,j_2})
				-(b_{i_1,j_1}a_{i_2,j_2}+b_{i_2,j_2}a_{i_1,j_1})]h(W)\}\\
				={}&\frac1{n(n-1)}
				\sum_{\substack{\mathbf l\in[n]_2\\ l_1=j_1\text{ or }l_2=j_2}}
				\E\{[X_{i_1,j_1}X_{i_2,j_2}(X_{i_1,j_1}+X_{i_2,j_2})-(b_{i_1,j_1}a_{i_2,j_2}+b_{i_2,j_2}a_{i_1,j_1})]h(W_{\mathbf i,\mathbf l})\} \notag\\
				={}&\frac1{n(n-1)}
				\sum_{\substack{\mathbf l\in[n]_2\\ l_1=j_1\text{ or }l_2=j_2}}
				\E\int_{|u|\leq2\delta/3}h'(W+u)\{[(b_{i_1,j_1}a_{i_2,j_2}+b_{i_2,j_2}a_{i_1,j_1}) \notag\\
				&\qquad\qquad\qquad\qquad\qquad
				-X_{i_1,j_1}X_{i_2,j_2}(X_{i_1,j_1}+X_{i_2,j_2})]
				H_{\Delta_{\mathbf i,\mathbf l}}(u)\}du \notag\\
				&\quad+\frac1{n(n-1)}
				\sum_{\substack{\mathbf l\in[n]_2\\ l_1=j_1\text{ or }l_2=j_2}}
				\E\{[X_{i_1,j_1}X_{i_2,j_2}(X_{i_1,j_1}+X_{i_2,j_2})-(b_{i_1,j_1}a_{i_2,j_2}+b_{i_2,j_2}a_{i_1,j_1})]h(W)\}. \notag
			\end{align}
			The last term in \eqref{cclt-two-pair-cubic} is included in $R_4$.

			For the centered terms involving $\Delta_{\mathbf i,\mathbf j}^{(\mathbf i)}$ in \eqref{cclt-K1-integral}, we expand the terms without indicators in \eqref{cclt-Delta-two-explicit} by the conditioning calculations used in \eqref{cclt-Delta-center-expand} and \eqref{cclt-last-cross-expand}. For the indicator terms, separate the cases where only the first match holds, only the second holds, or both hold. We present the details for two of these terms and treat the others analogously. First,
			\begin{align}\label{cclt-indicator-I1}
				&\E\{X_{\pi^{-1}(j_1),\pi(i_2)}\I(\pi(i_1)=j_1,\pi(i_2)\neq j_2)h(W)\}\\
				={}&\frac1{n(n-1)}
				\sum_{l\in[n]\setminus\{j_1,j_2\}}a_{i_1,l}
				\E\{h(W)\mid \pi(i_1)=j_1,\pi(i_2)=l\} \notag\\
				={}&-\frac1{n(n-1)}
				\sum_{l\in[n]\setminus\{j_1,j_2\}}a_{i_1,l}
				\E\int_{|u|\leq2\delta/3}h'(W+u)
				H_{\Delta_{\mathbf i,(j_1,l)}}(u)du -\frac{a_{i_1,j_1}+a_{i_1,j_2}}{n(n-1)}\E h(W). \notag
			\end{align}
			The same term with only the second match or both matches is handled in the same way. Moreover,
			\begin{align}\label{cclt-indicator-I2}
				&\E\{X_{\pi^{-1}(j_1),\pi(i_1)}\I(\pi(i_1)\neq j_1,\pi(i_2)=j_2)h(W)\}\\
				={}&\frac1{n(n-1)(n-2)}
				\sum_{k\in[n]\setminus\{i_1,i_2\}}
				\sum_{l\in[n]\setminus\{j_1,j_2\}}a_{k,l} \notag\\
				&\qquad\times\E\{h(W)\mid
				\pi(i_1)=l,\pi(k)=j_1,\pi(i_2)=j_2\} \notag\\
				={}&-\frac1{n(n-1)(n-2)}
				\sum_{k\in[n]\setminus\{i_1,i_2\}}
				\sum_{l\in[n]\setminus\{j_1,j_2\}}a_{k,l} \notag\\
				&\qquad\times\{\E h(W)
				-\E[h(W)\mid\pi(i_1)=l,\pi(k)=j_1,\pi(i_2)=j_2]\} \notag\\
				&\quad+\frac{a_{i_1,j_1}+a_{i_1,j_2}+a_{i_2,j_1}+a_{i_2,j_2}}
				{n(n-1)(n-2)}\E h(W). \notag
			\end{align}
			For the last difference in \eqref{cclt-indicator-I2}, let $\mathbf r=(i_1,k,i_2)$ and $\mathbf s=(l,j_1,j_2)$.  By the definition of $W_{\mathbf r,\mathbf s}$, $\mathcal L(W_{\mathbf r,\mathbf s}) =\mathcal L(W\mid\pi(i_1)=l,\pi(k)=j_1,\pi(i_2)=j_2)$, and hence
			\begin{align}\label{cclt-three-pair-difference}
				\E h(W)-\E\{h(W)\mid \pi(i_1)=l,\pi(k)=j_1,\pi(i_2)=j_2\}
				=\E\int_{|u|\leq\delta}h'(W+u)
				H_{\Delta_{\mathbf r,\mathbf s}}(u)du.
			\end{align}
			For every fixed $\mathbf i\in[n]_2$ and $\pi$, $\sum_{\mathbf j\in[n]_2}\I(\pi(i_1)=j_1\text{ or }\pi(i_2)=j_2)=2n-3$. Therefore, by \eqref{cclt-bounded},
			\begin{align}\label{cclt-indicator-remainder-bound}
				&\biggl|\E\biggl\{\frac{1}{2n^2(n-1)}
				\sum_{\mathbf i\in[n]_2}\sum_{\mathbf j\in[n]_2}
				\E[B_{\mathbf i,\mathbf j}]
				\I(\pi(i_1)=j_1\text{ or }\pi(i_2)=j_2)\notag\\
				&\qquad\times\bigl\{
				X_{\pi^{-1}(j_1),\pi(i_2)}
				+X_{\pi^{-1}(j_2),\pi(i_1)}
				-X_{\pi^{-1}(j_1),\pi(i_1)}
				-X_{\pi^{-1}(j_2),\pi(i_2)}
				\bigr\}\biggr\}\biggr|
				\leq C\delta^3.
			\end{align}
			The expectation bounded in \eqref{cclt-indicator-remainder-bound} is included in $R_4$. We collect in $\hat K_4(u)$ the integral terms in \eqref{cclt-X3-sum}--\eqref{cclt-three-pair-difference}, retaining their original normalizing constants, and place the remaining terms in $R_4$.  With this notation,
			\begin{equation}\label{cclt-L1-K4}
				\E\{\hat L_1h(W)\}
				=\E\int_{|u|\leq\delta}h'(W+u)\hat K_4(u)du
				+\E\{R_4h(W)\}.
			\end{equation}
			We next expand $\hat L_4$.  Remark \ref{remark-exchangeable-third-expansion} shows that $\gamma_{3,3}=0$.  Hence, by \eqref{third-L4-definition} and \eqref{cclt-K},
			\[
			\hat L_4=\int_{|u|\leq\delta/3}u\hat K(u)du
			=-\frac1{8n}\sum_{\mathbf i\in[n]_2}D_{\mathbf i}^{3}.
			\]
			Applying the exchangeability calculation described in the same remark to each fixed $\mathbf i$, we obtain
			\begin{equation}\label{cclt-L4-K5}
				\E\{\hat L_4h(W)\}
				=\E\int_{|u|\leq\delta/3}h'(W+u)\hat K_5(u)du,
			\end{equation}
			where
			\begin{equation}\label{cclt-K5-explicit}
				\hat K_5(u)=-\frac1{16n}\sum_{\mathbf i\in[n]_2}
				D_{\mathbf i}^{3}H_{D_{\mathbf i}}(u),\qquad R_5=0.
			\end{equation}
			
			To apply \eqref{third-canonical-direct-identification}, set $R_2=R$, $R_3=0$, $\hat K_2(u)=\hat K_3(u)=0$, $\hat L_{21}=\hat L_3=0$, and $\gamma_{3,2}=0$. Together with $\gamma_{3,3}=0$ and $R_5=0$, equations \eqref{cclt-basic-stein}, \eqref{cclt-L-K1}, \eqref{cclt-L1-K4}, and \eqref{cclt-L4-K5} then yield \eqref{third-order-Stein-identity} with $\mathcal R_{3,j}$ and $\mathcal K_{3,j}(u)$ as stated in Lemma \ref{lemma-cclt-third-stein}.

			\par\smallskip\noindent\textit{Step 2.}\quad We first compare the kernel mass with the third cumulant. Expanding $W^3$ and using the row and column centering in \eqref{cclt-row-column}, we obtain
			\begin{align}\label{cclt-gamma-proof}
				\Gamma_3=\E W^3
				&=\frac{n}{(n-1)(n-2)}
				\sum_{i=1}^{n}\sum_{j=1}^{n}a_{i,j}^{3}
				+\frac{3}{n-1}
				\sum_{i=1}^{n}\sum_{j=1}^{n}a_{i,j}c_{i,j}^{2}
				+\frac1n
				\sum_{i=1}^{n}\sum_{j=1}^{n}\E\bar X_{i,j}^{3}\\
				&=\frac1n\sum_{i=1}^{n}\sum_{j=1}^{n}\E X_{i,j}^{3}
				+\frac{3}{n(n-1)}
				\sum_{i=1}^{n}\sum_{j=1}^{n}a_{i,j}b_{i,j}
				+\frac{4}{n(n-1)(n-2)}
				\sum_{i=1}^{n}\sum_{j=1}^{n}a_{i,j}^{3}. \notag
			\end{align}
			The same centering condition gives
			\begin{align}\label{cclt-Delta-expectations}
				\E\Delta_{i,j}^{(i)}
				=-\frac{a_{i,j}}{n-1},\quad
				\E\Delta_{\mathbf i,\mathbf j}^{(\mathbf i)}
				=-\frac{a_{i_1,j_1}+a_{i_1,j_2}+a_{i_2,j_1}+a_{i_2,j_2}}{n-2}.
			\end{align}
			Then \eqref{cclt-K1-integral}, \eqref{cclt-Delta-expectations}, and \eqref{cclt-gamma-proof} yield
			\begin{align}\notag
				\E\int_{|u|\leq2\delta/3}\hat K_1(u)du
				=&\frac1{2n}\sum_{i=1}^{n}\sum_{j=1}^{n}\E X_{i,j}^3
				+\frac{5n-2}{2n^2(n-1)}
				\sum_{i=1}^{n}\sum_{j=1}^{n}a_{i,j}b_{i,j}
				+\frac{3n-2}{n^2(n-1)(n-2)}
				\sum_{i=1}^{n}\sum_{j=1}^{n}a_{i,j}^{3}\notag\\
				=&\frac{\Gamma_3}{2}
				+\frac1{n^2}\sum_{i=1}^{n}\sum_{j=1}^{n}a_{i,j}b_{i,j}
				+\frac1{n^2(n-1)}
				\sum_{i=1}^{n}\sum_{j=1}^{n}a_{i,j}^{3}.
			\end{align}
			Thus the definition of $\gamma_{3,1}$ and \eqref{cclt-bounded} imply
			\begin{align*}
				|\Gamma_3-\gamma_{3,1}|
				&\leq C\biggl\{\frac1{n^2}
				\sum_{i=1}^{n}\sum_{j=1}^{n}|a_{i,j}|b_{i,j}
				+\frac1{n^2(n-1)}
				\sum_{i=1}^{n}\sum_{j=1}^{n}|a_{i,j}|^3\biggr\}
				\leq C\delta^3.
			\end{align*}
			Each of the one-, two-, and three-pair differences in \eqref{cclt-X3-sum}--\eqref{cclt-three-pair-difference} is bounded in absolute value by $\delta$. Hence $\hat K_4(u)$ is supported on $[-\delta,\delta]$. Since $|t|\delta\leq\sqrt\delta\leq1$, counting the corresponding one-, two-, and three-pair terms in \eqref{cclt-X3-sum}--\eqref{cclt-three-pair-difference}, we obtain
			\begin{equation}\label{cclt-K4-proof-bound}
				\E\biggl\{\int_{|u|\leq\delta}e^{2|tu|}|\hat K_4(u)|du\,\bigg|\,W\biggr\}
				\leq C(\delta^2+n\delta^4),\qquad
				|t|\leq \delta^{-1/2}.
			\end{equation}
			We next verify \eqref{cclt-kernel-bound}.  By \eqref{cclt-bounded}, $|D_{\mathbf i}|\leq\delta/3$.  Thus, for $|t|\leq\delta^{-1/2}$,
			\[
			\int_{|u|\leq\delta} e^{2|tu|}|u^2\hat K(u)|du
			\leq \frac{C}{n}\sum_{\mathbf i\in[n]_2}|D_{\mathbf i}|^4
			\leq Cn\delta^4,
			\]
			and, by \eqref{cclt-K5-explicit},
			\[
			\int_{|u|\leq\delta} e^{2|tu|}|\hat K_5(u)|du
			\leq \frac{C}{n}\sum_{\mathbf i\in[n]_2}|D_{\mathbf i}|^4
			\leq Cn\delta^4.
			\]
			The explicit formula \eqref{cclt-K1-explicit} similarly yields $\int_{|u|\leq\delta}e^{2|tu|}|u\hat K_1(u)|du\leq Cn\delta^4$. Since $0\leq1-U\leq1$, the bounds just obtained for $u^2\hat K(u)$, $\hat K_5(u)$, and $u\hat K_1(u)$, together with \eqref{cclt-K4-proof-bound}, prove \eqref{cclt-kernel-bound}.

			It remains to bound the scalar remainders in \eqref{cclt-R-bound}--\eqref{cclt-R4-bound} for $|t|\leq\delta^{-1/2}$. Conditioning on the permutation separates the entries that occur in $W$ from those that do not. For this purpose, define $\mathcal I_{i,j}\coloneqq\{\pi(i)=j\}$. By the Cauchy--Schwarz inequality, $\E_t|R|\leq(\E_tR^2)^{1/2}$. Since
			\[
			R=n^{-1}\sum_{i=1}^{n}\sum_{j=1}^{n}\bar X_{i,j},
			\]
			\begin{equation}\label{cclt-R-square-expand}
			\begin{aligned}
				\E_tR^2
				={}&\frac1{n^2}\sum_{i=1}^{n}\sum_{j=1}^{n}
				\sum_{i'=1}^{n}\sum_{j'=1}^{n}\E_t\biggl\{\bar X_{i,j}\bar X_{i',j'}
				\bigl[
				\I(\mathcal I_{i,j})\I(\mathcal I_{i',j'})
				+\I(\mathcal I_{i,j})\I(\mathcal I_{i',j'}^c)\\
				&\hspace{6em}+\I(\mathcal I_{i,j}^c)\I(\mathcal I_{i',j'})
				+\I(\mathcal I_{i,j}^c)\I(\mathcal I_{i',j'}^c)
				\bigr]\biggr\}.
			\end{aligned}
			\end{equation}
			For the first term in \eqref{cclt-R-square-expand}, which contains $\I(\mathcal I_{i,j})\I(\mathcal I_{i',j'})$, the bound $|\bar X_{i,j}|\leq C\delta$ implies
			\[
			\bigl|\E_t\{\bar X_{i,j}\bar X_{i',j'}
			\I(\mathcal I_{i,j})\I(\mathcal I_{i',j'})\}\bigr|
			\leq C\delta^2\E_t\{
			\I(\mathcal I_{i,j})\I(\mathcal I_{i',j'})\}.
			\]
			Summing over the indices and using $\sum_{i=1}^{n}\sum_{j=1}^{n}\I(\mathcal I_{i,j})=n$ yields the bound $C\delta^2$.  For the second term, containing $\I(\mathcal I_{i,j})\I(\mathcal I_{i',j'}^c)$, condition on $\pi$.  On $\mathcal I_{i,j}\cap\mathcal I_{i',j'}^c$, the variable $\bar X_{i',j'}$ does not occur in $W$ and is independent of the remaining factors.  Since $\E\bar X_{i',j'}=0$, $\E_t\{\bar X_{i,j}\bar X_{i',j'}\mid\pi\}=0$ on $\mathcal I_{i,j}\cap\mathcal I_{i',j'}^c$. The third term is handled symmetrically.  In the last term, neither $\bar X_{i,j}$ nor $\bar X_{i',j'}$ appears in $W$.  Conditional on $\pi$, the expectation therefore vanishes unless $(i,j)=(i',j')$, while in the diagonal case $|\E_t\{\bar X_{i,j}^2\I(\mathcal I_{i,j}^c)\}|\leq C\delta^2\E_t\I(\mathcal I_{i,j}^c)$. After summing over the indices, this term is also bounded by $C\delta^2$.  Therefore
			\begin{equation}\label{cclt-R-second-moment-bound}
				\E_tR^2\leq C\delta^2.
			\end{equation}
			Hence $\E_t|R|\leq C\delta$.
			
			We now estimate $\E_t|R_1|$ using \eqref{cclt-R1-explicit}.  The deterministic term satisfies
			\begin{equation}\label{cclt-R1-deterministic-bound}
			\biggl|\frac1{4n}\sum_{\mathbf i\in[n]_2}\E D_{\mathbf i}^{2}-1\biggr|
			=|\E\hat L|\leq n^{-1}.
			\end{equation}
			Moreover, by \eqref{cclt-bounded},
			\begin{equation}\label{cclt-R1-first-sum-bound}
			\frac1n\E_t\left|\sum_{i=1}^{n}
			\{X_{i,\pi(i)}^{2}-\E X_{i,\pi(i)}^{2}\}\right|
			\leq C\delta^2.
			\end{equation}
			It remains to estimate the second sum on the right-hand side of \eqref{cclt-R1-explicit}.  By the Cauchy--Schwarz inequality,
			\begin{align*}
				\E_t\biggl|\sum_{i=1}^{n}\sum_{j=1}^{n}
				(X_{i,j}^{2}-b_{i,j})\biggr|
				\leq{}
				\biggl[\E_t\biggl(\sum_{i=1}^{n}\sum_{j=1}^{n}
				(X_{i,j}^{2}-b_{i,j})\biggr)^2\biggr]^{1/2}.
			\end{align*}
			Repeating the diagonal--off-diagonal argument used to prove \eqref{cclt-R-second-moment-bound} yields
			\begin{equation}\label{cclt-R1-second-sum-moment-bound}
			\E_t\biggl(\sum_{i=1}^{n}\sum_{j=1}^{n}
			(X_{i,j}^{2}-b_{i,j})\biggr)^2
			\leq Cn^2\delta^4,
			\end{equation}
			and hence
			\begin{equation}\label{cclt-R1-second-sum-bound}
			\frac1{2n}\E_t\biggl|\sum_{i=1}^{n}\sum_{j=1}^{n}
			(X_{i,j}^{2}-b_{i,j})\biggr|
			\leq C\delta^2.
			\end{equation}
			Combining \eqref{cclt-R1-deterministic-bound}, \eqref{cclt-R1-first-sum-bound}, and \eqref{cclt-R1-second-sum-bound}, and using $\mathcal R_{3,2}=R_1$, proves \eqref{cclt-R1-bound}. Finally, recalling \eqref{cclt-two-pair-cubic} and \eqref{cclt-indicator-remainder-bound}, we have
			\begin{align}\label{cclt-R4-intermediate-bound}
				\E_t|R_4|
				\leq\frac{1}{n^3(n-1)^2}
				\sum_{\mathbf i\in[n]_2}\sum_{\mathbf j\in[n]_2}
				\sum_{\substack{\mathbf l\in[n]_2\\
						l_1=j_1\text{ or }l_2=j_2}}C\delta^3+C\delta^3
				\leq C\delta^3.
			\end{align}
			Combining \eqref{cclt-R4-intermediate-bound} with the bound $|\Gamma_3-\gamma_{3,1}|\leq C\delta^3$ obtained from \eqref{cclt-K1-integral} and \eqref{cclt-Delta-expectations} yields \eqref{cclt-R4-bound} for $\mathcal R_{3,3}=R_4+(\gamma_{3,1}-\Gamma_3)/2$.
		\end{proof}

\section{Proof of Lemma \ref{lemma-cw-fourth-stein}}\label{appD}

\begin{proof}[Proof of Lemma \ref{lemma-cw-fourth-stein}]
Throughout the proof, constants depend only on $\beta$ and $h$. We use the Curie--Weiss exchangeable pair from the proof of Theorem \ref{theorem32}.

\par\smallskip\noindent\textit{Step 1.}\quad Write $\hat L=(2\lambda)^{-1}\E(\Delta^2\mid\Sigma)-1$. Since $\sigma_i^2=(\sigma_i')^2=1$, the difference satisfies $\Delta^3=4\sigma^{-2}\Delta$, $\Delta^4=4\sigma^{-2}\Delta^2$, and $\Delta^5=16\sigma^{-4}\Delta$. Fix a three times differentiable function $f(w)$ whose third derivative $f^{(3)}(w)$ is absolutely continuous and for which the expectations and integrals below are finite. Applying Taylor's expansion in \eqref{cw-exchangeable-stein} and using the definition of $\hat L$ and the identities for powers of $\Delta$, we obtain the exact identity
\begin{align}
\E\{Wf(W)\}
={}&\E f'(W)+\E\{\hat Lf'(W)\}+\E\{Rf(W)\}-\frac1{\sigma^2}\E\{(W-R)f''(W)\}
\label{cw-fourth-Wf}\\
&
+\frac2{3\sigma^2}\E\{(1+\hat L)f^{(3)}(W)\}-\frac1{12\lambda}\E\{\Delta^5(1-U)^3
f^{(4)}(W-U\Delta)\}.\notag
\end{align}
We include the last term on the right-hand side of \eqref{cw-fourth-Wf} in the kernel. All further terms involving $f^{(4)}(w)$ are treated in the same way.  We now expand the other terms.

Let $S_i=S-\sigma_i=nm+\sigma W-\sigma_i$. By Taylor's expansion, we obtain
\begin{align}
 \tanh(\beta S_i/n+\beta h)
 ={}&m+\frac{\beta(1-m^2)}n(\sigma W-\sigma_i)
 -\frac{\beta^2m(1-m^2)}{n^2}(\sigma W-\sigma_i)^2\notag\\*
 &+\frac{\beta^3(1-m^2)(3m^2-1)}{3n^3}
   (\sigma W-\sigma_i)^3+\mathcal E_{i,4},
 \label{cw-tanh-expansion}
\end{align}
where $|\mathcal E_{i,4}|\leq C_{\beta,h}n^{-4}(\sigma|W|+1)^4$. Using $R=W-\lambda^{-1}\E(\Delta\mid\Sigma)$ and $\hat L=(2\lambda)^{-1}\E(\Delta^2\mid\Sigma)-1$, we obtain
\begin{align}
 R&=W-\frac1{\lambda n\sigma}\sum_{i=1}^n
       \{\sigma_i-\tanh(\beta S_i/n+\beta h)\},
       \label{cw-fourth-R-definition}\displaybreak[1]\\
 \hat L&=\frac1{n(1-m^2)}\sum_{i=1}^n
       \{m^2-\sigma_i\tanh(\beta S_i/n+\beta h)\}.
       \label{cw-L-basic}
\end{align}
Substitution of \eqref{cw-tanh-expansion} into \eqref{cw-fourth-R-definition}, together with $\sum_{i=1}^n\sigma_i=nm+\sigma W$, yields
\begin{align}
R
={}&-\frac1{\lambda n\sigma}\Biggl[
\beta^2m(1-m^2)
\left(\frac{\sigma^2}{n}-\frac{2\sigma^2}{n^2}\right)W^2
+\beta m(1-m^2)
\label{cw-R-expansion}\\*
&\quad+\{\beta(1-m^2)-2\beta^2m^2(1-m^2)\}
\frac{\sigma}{n}W
+\frac{\beta^2m(1-m^2)}n\notag\\*
&\;-\frac{\beta^3(1-m^2)(3m^2-1)}3
\left\{
\left(\frac{\sigma^3}{n^2}-\frac{3\sigma^3}{n^3}\right)W^3
-\frac{3m\sigma^2}{n^2}W^2
+\left(\frac{3\sigma}{n^2}-\frac{\sigma}{n^3}\right)W
-\frac{m}{n^2}\right\}
+\mathcal E_{R,4}\Biggr].\notag
\end{align}
Similarly, substituting \eqref{cw-tanh-expansion} into \eqref{cw-L-basic} and using $\sum_{i=1}^n\sigma_i=nm+\sigma W$, we obtain
\begin{align}
\hat L
={}&\frac1{n(1-m^2)}\Biggl[
-m\sigma\{1+\beta(1-m^2)\}W
-\frac{\beta(1-m^2)\sigma^2}{n}W^2
+\beta(1-m^2)
\label{cw-L-expansion}\\
&\quad+\beta^2m(1-m^2)\left\{
\frac{\sigma^3}{n^2}W^3
+\frac{m\sigma^2}{n}W^2
-\frac{(2n-1)\sigma}{n^2}W+\frac mn
\right\}\notag\\
&\quad-\frac{\beta^3(1-m^2)(3m^2-1)}3
\left\{
\frac{\sigma^4}{n^3}W^4
+\frac{m\sigma^3}{n^2}W^3
-\frac{3(n-1)\sigma^2}{n^3}W^2
+\frac{3m\sigma}{n^2}W-\frac1{n^2}
\right\}
+\mathcal E_{L,4}\Biggr].\notag
\end{align}
The Taylor remainders inside the brackets in \eqref{cw-R-expansion} and \eqref{cw-L-expansion} are
\[
 \mathcal E_{R,4}=-\sum_{i=1}^n\mathcal E_{i,4},\qquad
 \mathcal E_{L,4}=-\sum_{i=1}^n\sigma_i\mathcal E_{i,4}.
\]
The bound following \eqref{cw-tanh-expansion} gives $|\mathcal E_{R,4}|+|\mathcal E_{L,4}| \leq C_{\beta,h}n^{-3}(\sigma|W|+1)^4$. Separating the leading terms of $\hat L$ and $R$, we have
\begin{align}
R
={}&-\frac{\beta^2m(1-m^2)\sigma}{\lambda n^2}W^2
+\frac{\beta^3(1-m^2)(3m^2-1)\sigma^2}
{3\lambda n^3}W^3+\mathcal R_{4,1,1},
\label{cw-fourth-R-local}\displaybreak[1]\\
\hat L
={}&-\frac{m\sigma\{1+\beta(1-m^2)\}}
{n(1-m^2)}W
-\frac{\beta(1-\beta m^2)\sigma^2}{n^2}W^2
+\mathcal R_{4,2,1},
\label{cw-fourth-L-local}
\end{align}
where the two local remainders satisfy
\begin{align}
|\mathcal R_{4,1,1}|
&\leq C_{\beta,h}\bigl[
n^{-1/2}+n^{-1}|W|+n^{-3/2}W^2
+n^{-2}|W|^3+n^{-3/2}W^4\bigr],
\label{cw-fourth-local-bounds}\\
|\mathcal R_{4,2,1}|
&\leq C_{\beta,h}\bigl[
n^{-1}+n^{-3/2}|W|+n^{-2}W^2
+n^{-3/2}|W|^3+n^{-2}W^4\bigr].\notag
\end{align}

{\tolerance=9999\emergencystretch=3em We next present four identities for $\E\{Wf''(W)\}$, $\E\{W^2f'(W)\}$, $\E\{W^2f(W)\}$, and $\E\{W^3f(W)\}$. Repeated applications of Taylor's expansion to the exchangeable-pair identity yield the first two:\par}
\begin{align}
\begin{split}
\E\{Wf''(W)\}
={}&\E f^{(3)}(W)+\E\{\hat Lf^{(3)}(W)\}
+\E\{Rf''(W)\}\\
&-\frac1{2\lambda}\E\{\Delta^3(1-U)
f^{(4)}(W-U\Delta)\},
\end{split}\label{cw-fourth-Wf2}\\[2pt]
\begin{split}
\E\{W^2f'(W)\}
={}&\E f'(W)+\E f^{(3)}(W)
+\E\{(\hat L+RW)f'(W)\}
\\
&+\E\{(W\hat L+R)f''(W)\}
+\E\left\{\left[
\hat L-\frac1{\sigma^2}W(W-R)\right]f^{(3)}(W)\right\}\\
&-\frac1{2\lambda}\E\{\Delta^3(1-U)
f^{(4)}(W-U\Delta)\}\\
&+\frac1{4\lambda}\E\{W\Delta^4(1-U)^2
f^{(4)}(W-U\Delta)\}.
\end{split}\label{cw-fourth-W2f1}
\end{align}
Here \eqref{cw-fourth-W2f1} follows by applying Taylor's expansion directly to the exchangeable-pair identity \eqref{cw-exchangeable-stein}. The identity for $\E\{W^2f(W)\}$, which will be applied to the quadratic term of $R$ in \eqref{cw-fourth-R-local}, follows by antisymmetrization:
\begin{align}
\E\{W^2f(W)\}
={}&\E f(W)+\E f''(W)
+\E\{(\hat L+RW)f(W)\}
+\E\{(W\hat L+R)f'(W)\}
\label{cw-fourth-W2f}\\
&+\E\left\{\left[
\hat L-\frac1{\sigma^2}W(W-R)\right]f''(W)\right\}
+\frac2{3\sigma^2}\E\{W(1+\hat L)f^{(3)}(W)\}\notag\\
&-\frac1{4\lambda}\E\left\{\Delta^4
\int_0^1u(1-u)f^{(4)}(W-u\Delta)\,du\right\}\notag\\
&-\frac1{12\lambda}\E\{W\Delta^5(1-U)^3
f^{(4)}(W-U\Delta)\}.\notag
\end{align}
For $\E\{W^3f(W)\}$, which will be applied to the cubic term of $R$ in \eqref{cw-fourth-R-local}, define
\begin{align}
\mathcal G_0
\coloneqq{}&2W\hat L-\frac2{\sigma^2}(W-R)+2R+RW^2,\qquad
\mathcal G_1
\coloneqq 3\hat L+RW+W^2\hat L,
\label{cw-fourth-G}\\*
\mathcal G_2
\coloneqq{}&W\hat L+R-\frac2{\sigma^2}(W-R)
-\frac1{\sigma^2}W^2(W-R),
\notag\\*
\mathcal G_3
\coloneqq{}&\hat L+\frac4{3\sigma^2}(1+\hat L)
-\frac1{\sigma^2}W(W-R)
+\frac2{3\sigma^2}W^2(1+\hat L).\notag
\end{align}
Then
\begin{align}
\E\{W^3f(W)\}
={}&3\E f'(W)+\E f^{(3)}(W)
+\sum_{j=0}^{3}\E\{\mathcal G_jf^{(j)}(W)\}
+\mathcal T_3(f^{(4)}),
\label{cw-fourth-W3f}\displaybreak[1]\\
\mathcal T_3(f^{(4)})
\coloneqq{}&-\frac1{6\lambda}\E\{\Delta^5(1-U)^3
f^{(4)}(W-U\Delta)\}-\frac1{2\lambda}\E\{\Delta^3(1-U)
f^{(4)}(W-U\Delta)\}\label{cw-fourth-T3}\\
&+\frac1{4\lambda}\E\{W\Delta^4(1-U)^2
f^{(4)}(W-U\Delta)\}\notag\\
&-\frac1{12\lambda}\E\{W^2\Delta^5(1-U)^3
f^{(4)}(W-U\Delta)\}.\notag
\end{align}

Substituting \eqref{cw-fourth-Wf2} into \eqref{cw-fourth-Wf} cancels the two terms containing $Rf''(W)$, and we obtain
\begin{align}
\E\{Wf(W)\}
={}&\E f'(W)-\frac1{3\sigma^2}\E f^{(3)}(W)
+\E\{\hat Lf'(W)\}
-\frac1{3\sigma^2}\E\{\hat Lf^{(3)}(W)\}
\label{cw-fourth-working-identity}\\
&+\E\{Rf(W)\}+\frac1{2\lambda\sigma^2}
\E\{\Delta^3(1-U)f^{(4)}(W-U\Delta)\}\notag\\
&-\frac1{12\lambda}\E\{\Delta^5(1-U)^3
f^{(4)}(W-U\Delta)\}.\notag
\end{align}
It remains to expand $\E\{\hat Lf'(W)\}$ and $\E\{Rf(W)\}$.

We begin by expanding $\E\{\hat Lf'(W)\} -\{3\sigma^2\}^{-1}\E\{\hat Lf^{(3)}(W)\}$.  By \eqref{cw-fourth-L-local},
\begin{align}
\E\{\hat Lf'(W)\}
={}&-\frac{m\sigma\{1+\beta(1-m^2)\}}
{n(1-m^2)}\E\{Wf'(W)\}
-\frac{\beta(1-\beta m^2)\sigma^2}{n^2}
\E\{W^2f'(W)\}\notag\\
&+\E\{\mathcal R_{4,2,1}f'(W)\}.\label{cw-fourth-L-first}
\end{align}
After replacing $f(w)$ by $f'(w)$ in \eqref{cw-exchangeable-stein} and applying Taylor's expansion, we antisymmetrize $\E\{\Delta^3f^{(3)}(W)\}$ to obtain
\begin{align}
\E\{Wf'(W)\}
={}\E\{(1+\hat L)f''(W)\}+\E\{Rf'(W)\}
-\frac1{4\lambda}\E\{\Delta^4U(1-U)
f^{(4)}(W-U\Delta)\}.
\label{cw-fourth-Wf1}
\end{align}
Substitute \eqref{cw-fourth-Wf1} into \eqref{cw-fourth-L-first}, expand $R$ and $\hat L$ by \eqref{cw-fourth-R-local} and \eqref{cw-fourth-L-local}, and apply \eqref{cw-fourth-Wf2} and \eqref{cw-fourth-W2f1} to the resulting $Wf''(W)$ and $W^2f'(W)$ terms. For the remaining terms with $Wf^{(3)}(W)$, use
\begin{equation}\label{cw-fourth-Wf3}
 \E\{Wf^{(3)}(W)\}
 =\frac1{2\lambda}\E\{\Delta^2f^{(4)}(W-U\Delta)\}
  +\E\{Rf^{(3)}(W)\}.
\end{equation}
The leading deterministic terms in $\E\{\hat Lf'(W)\}-(3\sigma^2)^{-1}\E\{\hat Lf^{(3)}(W)\}$ are
\begin{equation}\label{cw-fourth-L-deterministic}
\begin{aligned}
&-\frac{m\sigma\{1+\beta(1-m^2)\}}
{n(1-m^2)}\E f''(W)\\
&\quad+\left[
\left\{\frac{m\sigma\{1+\beta(1-m^2)\}}
{n(1-m^2)}\right\}^{\!2}
-\frac{\beta\sigma^2\{1-\beta(1+m^2)\}}
{n^2\{1-\beta(1-m^2)\}}
\right]\E f^{(3)}(W).
\end{aligned}
\end{equation}
The deterministic correction multiplying $\E f'(W)$ is included in $\mathcal R_{4,2}^{(L)}$, while the two coefficients in \eqref{cw-fourth-L-deterministic} enter $\gamma_{3,n}/2$ and $\gamma_{4,n}/6$, respectively.  The full remainder factor multiplying $f'(W)$ inside the expectation is
\begin{align}
\mathcal R_{4,2}^{(L)}
\coloneqq{}&-\frac{\beta\sigma^2\{1-\beta(1+m^2)\}}
{n^2\{1-\beta(1-m^2)\}}
+\mathcal R_{4,2,1}
-\frac{m\sigma\{1+\beta(1-m^2)\}}
{n(1-m^2)}
\label{cw-fourth-R42-L}\\
&\times
\left[
\frac{\beta^3(1-m^2)(3m^2-1)\sigma^2}
{3\lambda n^3}W^3+\mathcal R_{4,1,1}\right]-\frac{\beta\sigma^2\{1-\beta(1+m^2)\}}
{n^2\{1-\beta(1-m^2)\}}(\hat L+RW).\notag
\end{align}
The remainder factor multiplying $f''(W)$ inside the expectation is
\begin{align}
\mathcal R_{4,3}^{(L)}
\coloneqq{}&
\frac{m\sigma\{1+\beta(1-m^2)\}}
{n(1-m^2)}
\frac{\beta(1-\beta m^2)\sigma^2}{n^2}W^2
-\frac{m\sigma\{1+\beta(1-m^2)\}}
{n(1-m^2)}\mathcal R_{4,2,1}
\label{cw-fourth-R43-L}\\
&
+\left[
\frac{m\sigma\{1+\beta(1-m^2)\}}
{n(1-m^2)}\right]^2R-\frac{\beta\sigma^2\{1-\beta(1+m^2)\}}
{n^2\{1-\beta(1-m^2)\}}(W\hat L+R).\notag
\end{align}
After including $-\{3\sigma^2\}^{-1}\E\{\hat Lf^{(3)}(W)\}$ from \eqref{cw-fourth-working-identity}, the remainder factor multiplying $f^{(3)}(W)$ inside the expectation is
\begin{align}
\mathcal R_{4,4}^{(L)}
\coloneqq{}&\left[
\left\{\frac{m\sigma\{1+\beta(1-m^2)\}}
{n(1-m^2)}\right\}^{\!2}
-\frac{\beta\sigma^2\{1-\beta(1+m^2)\}}
{n^2\{1-\beta(1-m^2)\}}
-\frac1{3\sigma^2}\right]
\label{cw-fourth-R44-L}\\*
&\quad\times\left[
\mathcal R_{4,2,1}
-\frac{m\sigma\{1+\beta(1-m^2)\}}
{n(1-m^2)}R
-\frac{\beta(1-\beta m^2)\sigma^2}{n^2}W^2
\right]\notag\\*
&+\frac{\beta\sigma^2\{1-\beta(1+m^2)\}}
{n^2\{1-\beta(1-m^2)\}}\,
\frac1{\sigma^2}W(W-R).\notag
\end{align}

For $\E\{Rf(W)\}$, apply \eqref{cw-fourth-W2f} and \eqref{cw-fourth-W3f} to the quadratic and cubic terms of \eqref{cw-fourth-R-local}. Expand $R$ and $\hat L$ in the resulting expressions once more. The terms involving $Wf(W)$, $Wf''(W)$, $W^2f'(W)$ and $W^3f(W)$ are then reduced by \eqref{cw-fourth-Wf}, \eqref{cw-fourth-Wf2}, \eqref{cw-fourth-W2f1} and \eqref{cw-fourth-W3f}, respectively. Grouping by derivatives of $f(w)$ gives the following scalar remainder factors, with every fourth-derivative term reserved for the kernel:
\begin{align}
\mathcal R_{4,1}^{(R)}
\coloneqq{}&-\frac{\beta^2m(1-m^2)\sigma}{\lambda n^2}
+\mathcal R_{4,1,1}
+\frac{\beta^2m(1-m^2)\sigma}{\lambda n^2}
\left[
\frac{\beta(1-\beta m^2)\sigma^2}{n^2}W^2
-\mathcal R_{4,2,1}\right.
\label{cw-fourth-R41-R}\\*
&\quad\left.
-\frac{\beta^3(1-m^2)(3m^2-1)\sigma^2}
{3\lambda n^3}W^4
-W\mathcal R_{4,1,1}\right]
+\frac{\beta^2m^2\{1+\beta(1-m^2)\}\sigma^2}
{\lambda n^3}R\notag\\*
&+\left[
\frac{\beta^3(1-m^2)(3m^2-1)\sigma^2}{3\lambda n^3}
+\frac{\beta^4m^2(1-m^2)^2\sigma^2}{\lambda^2n^4}
\right]\mathcal G_0 .\notag
\end{align}
The full remainder factor multiplying $f'(W)$ inside the expectation is
\begin{align}
\mathcal R_{4,2}^{(R)}
\coloneqq{}&
\frac{\beta^3(1-m^2)(3m^2-1)\sigma^2}{\lambda n^3}
+\frac{2\beta^2m^2\{1+\beta(1-m^2)\}\sigma^2}
{\lambda n^3}
+\frac{4\beta^4m^2(1-m^2)^2\sigma^2}{\lambda^2n^4}
\label{cw-fourth-R42-R}\\*
&+\frac{\beta^2m(1-m^2)\sigma}{\lambda n^2}
\left[
\left\{\frac{\beta(1-\beta m^2)\sigma^2}{n^2}
-\frac{\beta^3(1-m^2)(3m^2-1)\sigma^2}
{3\lambda n^3}\right\}W^3\right.\notag\\*
&\quad\left.
-W\mathcal R_{4,2,1}-\mathcal R_{4,1,1}\right]
+\frac{\beta^2m^2\{1+\beta(1-m^2)\}\sigma^2}
{\lambda n^3}\hat L\notag\\*
&+\left[
\frac{\beta^2m^2\{1+\beta(1-m^2)\}\sigma^2}
{\lambda n^3}
+\frac{\beta^4m^2(1-m^2)^2\sigma^2}
{\lambda^2n^4}\right](\hat L+RW)\notag\\*
&+\left[
\frac{\beta^3(1-m^2)(3m^2-1)\sigma^2}
{3\lambda n^3}
+\frac{\beta^4m^2(1-m^2)^2\sigma^2}
{\lambda^2n^4}\right]\mathcal G_1 .\notag
\end{align}
The full remainder factor multiplying $f''(W)$ inside the expectation is
\begin{align}
\mathcal R_{4,3}^{(R)}
\coloneqq{}&\frac{\beta^2m(1-m^2)\sigma}{\lambda n^2}
\left[
\frac{\beta(1-\beta m^2)\sigma^2}{n^2}W^2
-\mathcal R_{4,2,1}
+\frac1{\sigma^2}W(W-R)\right]
\label{cw-fourth-R43-R}\\
&+\frac{\beta^2m^2\{1+\beta(1-m^2)\}\sigma^2}
{\lambda n^3}
\left[R-\frac1{\sigma^2}(W-R)\right]\notag\\
&+\left[
\frac{\beta^2m^2\{1+\beta(1-m^2)\}\sigma^2}
{\lambda n^3}
+\frac{\beta^4m^2(1-m^2)^2\sigma^2}
{\lambda^2n^4}\right](W\hat L+R)\notag\\
&+\left[
\frac{\beta^3(1-m^2)(3m^2-1)\sigma^2}
{3\lambda n^3}
+\frac{\beta^4m^2(1-m^2)^2\sigma^2}
{\lambda^2n^4}\right]\mathcal G_2 .\notag
\end{align}
For the terms involving $f^{(3)}(W)$, the full remainder factor inside the expectation is
\begin{align}
\mathcal R_{4,4}^{(R,0)}
\coloneqq{}&-\frac{2\beta^2m(1-m^2)}
{3\lambda n^2\sigma}W(1+\hat L)
\label{cw-fourth-R44-R0}\\*
&+\frac{2\beta^2m^2\{1+\beta(1-m^2)\}}
{3\lambda n^3}(1+\hat L)
+\frac{\beta^2m^2\{1+\beta(1-m^2)\}\sigma^2}
{\lambda n^3}\hat L\notag\\*
&+\left[
\frac{\beta^2m^2\{1+\beta(1-m^2)\}\sigma^2}
{\lambda n^3}
+\frac{\beta^4m^2(1-m^2)^2\sigma^2}
{\lambda^2n^4}\right]
\left[\hat L-\frac1{\sigma^2}W(W-R)\right]\notag\\*
&+\left[
\frac{\beta^3(1-m^2)(3m^2-1)\sigma^2}
{3\lambda n^3}
+\frac{\beta^4m^2(1-m^2)^2\sigma^2}
{\lambda^2n^4}\right]\mathcal G_3 .\notag
\end{align}
Define the deterministic coefficient $\chi_{4,n}$ by
\begin{align}\label{cw-fourth-chi}
\chi_{4,n}
\coloneqq{}&
\frac{2\beta^2m(1-m^2)}{3\lambda n^2\sigma}
+\frac{m\sigma\{1+\beta(1-m^2)\}}{n(1-m^2)}
\Biggl[
\left(2+\frac{2}{3\sigma^2}\right)
\frac{m\sigma\{1+\beta(1-m^2)\}}{n(1-m^2)}
\frac{\beta^2m(1-m^2)\sigma}{\lambda n^2} \notag\\*
&\quad
+\left(2+\frac{4}{3\sigma^2}\right)
\left\{\frac{\beta^2m(1-m^2)\sigma}
{\lambda n^2}\right\}^{\!2}
+\left(1+\frac{4}{3\sigma^2}\right)
\frac{\beta^3(1-m^2)(3m^2-1)\sigma^2}
{3\lambda n^3}
\Biggr].
\end{align}
By \eqref{cw-fourth-L-local}, the deterministic linear term in $\mathcal R_{4,4}^{(R,0)}$ is $-\chi_{4,n}W$.  We therefore set
\begin{equation}
\mathcal R_{4,4}^{(R)}
\coloneqq\mathcal R_{4,4}^{(R,0)}
+\chi_{4,n}(W-R).
\label{cw-fourth-R44-R}
\end{equation}
Since $\mathcal R_{4,4}^{(R,0)} =\mathcal R_{4,4}^{(R)}-\chi_{4,n}(W-R)$, applying \eqref{cw-fourth-Wf3} to the term involving $Wf^{(3)}(W)$ transforms $-\chi_{4,n}\E\{(W-R)f^{(3)}(W)\}$ into $-\chi_{4,n}(2\lambda)^{-1} \E\{\Delta^2f^{(4)}(W-U\Delta)\}$. The same substitutions give the leading deterministic terms
\begin{align}
&-\frac{\beta^2m(1-m^2)\sigma}{\lambda n^2}\E f''(W)
\label{cw-fourth-R-deterministic}\\*
&+\Biggl[
\frac{\beta^3(1-m^2)(3m^2-1)\sigma^2}
{3\lambda n^3}
+\frac{2\beta^2m^2\{1+\beta(1-m^2)\}\sigma^2}
{\lambda n^3}
+\frac{2\beta^4m^2(1-m^2)^2\sigma^2}
{\lambda^2n^4}\Biggr]\E f^{(3)}(W).\notag
\end{align}
The deterministic corrections multiplying $\E f(W)$ and $\E f'(W)$ are included in $\mathcal R_{4,1}^{(R)}$ and $\mathcal R_{4,2}^{(R)}$, respectively, while the two coefficients in \eqref{cw-fourth-R-deterministic} enter $\gamma_{3,n}/2$ and $\gamma_{4,n}/6$.  Thus, \eqref{cw-fourth-R41-R}--\eqref{cw-fourth-R44-R} specify all scalar remainders in the expansion of $\E\{Rf(W)\}$.

The scalar remainders from the expansions of $\E\{\hat Lf'(W)\}$ and $\E\{Rf(W)\}$ are combined by setting
\begin{equation}\label{cw-fourth-global-remainders}
 \mathcal R_{4,1}=\mathcal R_{4,1}^{(R)},\qquad
 \mathcal R_{4,j}=\mathcal R_{4,j}^{(L)}+\mathcal R_{4,j}^{(R)}
 \quad(j=2,3,4).
\end{equation}
The deterministic multiples of $\E f''(W)$ and $\E f^{(3)}(W)$ obtained from \eqref{cw-fourth-working-identity}, \eqref{cw-fourth-L-deterministic}, and \eqref{cw-fourth-R-deterministic} combine into $\gamma_{3,n}/2$ and $\gamma_{4,n}/6$, respectively:
\begin{align}
\frac{\gamma_{3,n}}2
={}&-\frac{m\sigma\{1+\beta(1-m^2)\}}
{n(1-m^2)}
-\frac{\beta^2m(1-m^2)\sigma}{\lambda n^2},
\label{cw-gamma3-raw}\displaybreak[1]\\
\frac{\gamma_{4,n}}6
={}&-\frac1{3\sigma^2}
+\left\{\frac{m\sigma\{1+\beta(1-m^2)\}}
{n(1-m^2)}\right\}^{\!2}
-\frac{\beta\sigma^2\{1-\beta(1+m^2)\}}
{n^2\{1-\beta(1-m^2)\}}
\label{cw-gamma4-raw}\\
&+\frac{\beta^3(1-m^2)(3m^2-1)\sigma^2}
{3\lambda n^3}
+\frac{2\beta^2m^2\{1+\beta(1-m^2)\}\sigma^2}
{\lambda n^3}+\frac{2\beta^4m^2(1-m^2)^2\sigma^2}
{\lambda^2n^4}.\notag
\end{align}
Substituting the definitions of $\sigma$ and $\lambda$ into \eqref{cw-gamma3-raw}--\eqref{cw-gamma4-raw}, we obtain $\gamma_{3,n}$ and $\gamma_{4,n}$ as the leading terms displayed in \eqref{cw-explicit-cumulants}. Together with \eqref{cw-fourth-global-remainders}, these coefficients determine the scalar part of \eqref{cw-fourth-stein}. To write the integral remainders in \eqref{cw-fourth-Wf}--\eqref{cw-fourth-T3} in kernel form, define, on the support of $H_\Delta(u)$, $s_\Delta(u)=-u/\Delta$. When $\Delta=0$, set every expression containing $s_\Delta(u)H_\Delta(u)$ equal to zero.  For every integrable weight $\varphi(s)$,
\begin{align}
\Delta^k\int_0^1\varphi(s)f^{(4)}(W-s\Delta)\,ds
=\int_{|u|\leq2/\sigma}f^{(4)}(W+u)
\Delta^{k-1}\varphi\{s_\Delta(u)\}H_\Delta(u)\,du.
\label{cw-fourth-kernel-change}
\end{align}
Applying \eqref{cw-fourth-kernel-change} to each integral remainder and collecting the coefficients of $f^{(4)}(W+u)$, we obtain the following six random kernel components:

\begin{align}
\mathcal K_{4,1,0}(u)
\coloneqq{}&\frac1{24\lambda}\E\left[
\Delta^4\{3(1-s_\Delta(u))-2(1-s_\Delta(u))^3\}
H_\Delta(u)\,\mid\Sigma\right],
\label{cw-fourth-kernel-0}\displaybreak[1]\\
\mathcal K_{4,1,1}(u)
\coloneqq{}&\frac1{4\lambda}\left[
\frac{m\sigma\{1+\beta(1-m^2)\}}{n(1-m^2)}
+\frac{\beta^2m(1-m^2)\sigma}{\lambda n^2}
\right]\label{cw-fourth-kernel-1}\\*
&\quad\times\E\left[
\Delta^3s_\Delta(u)\{1-s_\Delta(u)\}H_\Delta(u)
\,\mid\Sigma\right],
\notag\displaybreak[1]\\
\mathcal K_{4,1,2}(u)
\coloneqq{}&-\frac1{2\lambda}\Biggl[
\left\{\frac{m\sigma\{1+\beta(1-m^2)\}}
{n(1-m^2)}\right\}^{\!2}
-\frac{\beta\sigma^2\{1-\beta(1+m^2)\}}
{n^2\{1-\beta(1-m^2)\}}
\label{cw-fourth-kernel-2}\displaybreak[1]\\
&\quad
+\frac{\beta^3(1-m^2)(3m^2-1)\sigma^2}
{3\lambda n^3}
+\frac{2\beta^2m^2\{1+\beta(1-m^2)\}\sigma^2}
{\lambda n^3}
+\frac{2\beta^4m^2(1-m^2)^2\sigma^2}
{\lambda^2n^4}\Biggr]\notag\\
&\quad\times\E\left[
\Delta^2\{1-s_\Delta(u)\}H_\Delta(u)
\,\mid\Sigma\right],\notag\displaybreak[1]\\
\mathcal K_{4,1,3}(u)
\coloneqq{}&\frac{W}{4\lambda}\Biggl[
\frac{\beta^3(1-m^2)(3m^2-1)\sigma^2}
{3\lambda n^3}
+\frac{\beta^2m^2\{1+\beta(1-m^2)\}\sigma^2}
{\lambda n^3}
\label{cw-fourth-kernel-3}\\
&\quad
+\frac{2\beta^4m^2(1-m^2)^2\sigma^2}
{\lambda^2n^4}
-\frac{\beta\sigma^2\{1-\beta(1+m^2)\}}
{n^2\{1-\beta(1-m^2)\}}\Biggr]\E\left[
\Delta^3\{1-s_\Delta(u)\}^2H_\Delta(u)
\,\mid\Sigma\right],\notag\displaybreak[1]\\
\mathcal K_{4,1,4}(u)
\coloneqq{}&-\frac1{2\lambda}\Biggl[
\frac{m\sigma\{1+\beta(1-m^2)\}}{n(1-m^2)}
\left\{
\left(\frac{m\sigma\{1+\beta(1-m^2)\}}
{n(1-m^2)}\right)^{\!2}\right.
\label{cw-fourth-kernel-4}\\
&\left.\quad
-\frac{\beta\sigma^2\{1-\beta(1+m^2)\}}
{n^2\{1-\beta(1-m^2)\}}
-\frac1{3\sigma^2}\right\}
+\chi_{4,n}\Biggr]\E\{\Delta H_\Delta(u)\mid\Sigma\},
\notag\displaybreak[1]\\
\mathcal K_{4,1,5}(u)
\coloneqq{}&\frac1{12\lambda}
\Biggl[
-\frac{\beta^2m^2\{1+\beta(1-m^2)\}\sigma^2}
{\lambda n^3}
-\frac{2\beta^3(1-m^2)(3m^2-1)\sigma^2}
{3\lambda n^3}
\label{cw-fourth-kernel-5}\\*
&\quad-\frac{2\beta^4m^2(1-m^2)^2\sigma^2}
{\lambda^2n^4}
+\frac{\beta^2m(1-m^2)\sigma}{\lambda n^2}W
\notag\\*
&\quad-\left\{
\frac{\beta^3(1-m^2)(3m^2-1)\sigma^2}
{3\lambda n^3}
+\frac{\beta^4m^2(1-m^2)^2\sigma^2}
{\lambda^2n^4}\right\}W^2
\Biggr]\notag\\*
&\quad\times\E\left[
\Delta^4\{1-s_\Delta(u)\}^3H_\Delta(u)
\,\mid\Sigma\right].\notag
\end{align}

Let
\begin{equation}
\mathcal K_{4,1}(u)
=\sum_{\ell=0}^{5}\mathcal K_{4,1,\ell}(u),
\qquad
\mathcal K_{4,2}(u)=0.
\label{cw-fourth-final-kernel}
\end{equation}
By \eqref{cw-fourth-kernel-change} and the definitions \eqref{cw-fourth-kernel-0}--\eqref{cw-fourth-final-kernel}, all terms involving $f^{(4)}(w)$ in the expansion of \eqref{cw-fourth-working-identity} sum to
\[
\E\int_{|u|\leq2\sigma^{-1}}f^{(4)}(W+u)
\{\mathcal K_{4,1}(u)+\mathcal K_{4,2}(u)\}\,du.
\]

\par\smallskip\noindent\textit{Step 2.}\quad We identify the cumulants and verify the remainder and kernel bounds. The following lemma identifies $\gamma_{3,n}$ and $\gamma_{4,n}$ with the leading terms of the centered cumulants of $W$.

\begin{lemma}\label{lemma-cw-moment-expansion}
With the notation of Section \ref{subsec-cw-model}, the first four moments of $W$ satisfy
\begin{align}
\E W
={}&-\frac{\beta m(1-m^2)}
{\sigma\{1-\beta(1-m^2)\}^2}
+O_{\beta,h}(n^{-3/2}),
\label{cw-raw-moment-expansions}\\*
\E W^2
={}&1+\frac{\beta(1-m^2)}
{\sigma^2\{1-\beta(1-m^2)\}^4}
\bigl[3m^2-1+\beta(1-m^2)(1+2m^2)\bigr]
+O_{\beta,h}(n^{-2}),\notag\\*
\E W^3
={}&-\frac{m\{2+3\beta(1-m^2)\}}
{\sigma\{1-\beta(1-m^2)\}^2}
+O_{\beta,h}(n^{-3/2}),\notag\\*
\E W^4
={}&3+\frac{1}{\sigma^2\{1-\beta(1-m^2)\}^4}
\Bigl[2(3m^2-1)+4\beta(1-m^2)(8m^2-1)\notag\\*
&\hphantom{3+\frac{1}{\sigma^2\{1-\beta(1-m^2)\}^4}\Bigl[}
+6\beta^2(1-m^2)^2(1+2m^2)\Bigr]
+O_{\beta,h}(n^{-2}).\notag
\end{align}
Consequently, $\E W=O_{\beta,h}(n^{-1/2})$ and $\Var(W)=1+O_{\beta,h}(n^{-1})$, while $\Gamma_3=\gamma_{3,n}+O_{\beta,h}(n^{-3/2})$ and $\Gamma_4=\gamma_{4,n}+O_{\beta,h}(n^{-2})$.
\end{lemma}

\begin{proof}[Proof of Lemma \ref{lemma-cw-moment-expansion}]
We first bound the absolute moments of the normalized magnetization. This lemma only needs a bound at $h$, whereas \eqref{cw-tilted-moments} requires bounds under exponential tilting. Exponential tilting in \eqref{intro-tilted-definitions} multiplies the unnormalized Gibbs weight in \eqref{cw-density} by $e^{tW}=e^{tS/\sigma-tnm/\sigma}$. The factor $e^{-tnm/\sigma}$ cancels upon normalization, giving the Gibbs law with external field $h+t/(\beta\sigma)$. We therefore prove the bound uniformly for fields $H$ near $h$.

Choose a compact interval $J$ containing $h$ in its interior, with $0\notin J$ when $\beta\geq1$. For $H\in J$, write $\p_{\beta,H}$ for the Gibbs law with field $H$. Let $m(H)$ be the solution of $m(H)=\tanh\{\beta m(H)+\beta H\}$ with $m(H)H\geq0$. Define $\sigma(H)$ by \eqref{cw-normalization} with $m$ replaced by $m(H)$.

We use the Gaussian representation in \citet[Proposition~4.1]{DM2023}. Let $Z\sim N(0,1)$ be independent of the spins and put $X_H=S/n+Z/\sqrt{n\beta}$. Completing the square in the Gibbs weight shows that $X_H$ has density proportional to $e^{-nG_H(u)}$, where
\begin{equation}\label{cw-gaussian-potential}
 G_H(u)=\frac{\beta u^2}{2}-\log\cosh(\beta u+\beta H),\qquad u\in\mathbb R.
\end{equation}
The function in \eqref{cw-gaussian-potential} has a unique global minimum at $m(H)$, with $G_H''\{m(H)\}=\beta[1-\beta\{1-m(H)^2\}]>0$. The implicit function theorem makes $m(H)$ smooth on $J$. The quadratic comparison in the proof of \citet[Proposition~6.1]{DM2023} holds uniformly for $H\in J$. Indeed, the ratio $[G_H(u)-G_H\{m(H)\}]/\{u-m(H)\}^2$ extends continuously at $u=m(H)$ to $G_H''\{m(H)\}/2>0$. It is positive elsewhere and tends to $\beta/2$ uniformly as $|u|\to\infty$. Compactness of $J$ therefore gives a uniform positive lower bound. The upper bound follows from $G_H''(u)\leq\beta$, giving
\[
 c\{u-m(H)\}^2\leq G_H(u)-G_H\{m(H)\}
 \leq\frac{\beta}{2}\{u-m(H)\}^2,\qquad H\in J,\quad u\in\mathbb R,
\]
where $c>0$ depends only on $\beta$ and $J$. Comparing the two Gaussian integrals, we obtain
\begin{equation}\label{cw-gaussian-moment-bound}
 \E_{\beta,H}\exp\!\left\{\frac{cn}{2}\{X_H-m(H)\}^2\right\}
 \leq\frac{\int_{\mathbb R}e^{-cnu^2/2}\,du}{\int_{\mathbb R}e^{-n\beta u^2/2}\,du}
 =\sqrt{\frac{\beta}{c}}.
\end{equation}
Since $\E_{\beta,H}(X_H\mid S)=S/n$, conditional Jensen's inequality gives the same bound in \eqref{cw-gaussian-moment-bound} with $X_H$ replaced by $S/n$. Together with $\sigma(H)\asymp_{\beta,J}\sqrt n$, this yields
\begin{equation}\label{cw-fixed-moment-bound}
 \sup_{H\in J}\E_{\beta,H}\left|\frac{S-nm(H)}{\sigma(H)}\right|^r
 \leq C_{\beta,J,r},\qquad 0\leq r\leq10.
\end{equation}
In particular, $\E|W|^r\leq C_{\beta,h,r}$ for the original field $h$.

We now use the exchangeable-pair identity to obtain the required moment expansions. For an integer $k\geq1$, we take $f(w)=w^k$ in \eqref{cw-exchangeable-stein} and expand $W^k-(W-\Delta)^k$ to obtain
\begin{align}
\E W^{k+1}
={}&k\E W^{k-1}+k\E(\hat LW^{k-1})+\E(RW^k)
+\frac1{2\lambda}
\sum_{\ell=2}^{k}(-1)^{\ell+1}\binom{k}{\ell}
\E(\Delta^{\ell+1}W^{k-\ell}).\label{cw-raw-moment-recursion}
\end{align}
In particular, exchangeability, the definition of $\hat L$, and the identities for powers of $\Delta$ imply
\begin{align}
\E W^2
={}&1+\E\hat L+\E(RW),
\label{cw-second-moment-identity}\displaybreak[1]\\
\E W^3
={}&2\E W+2\E(W\hat L)+\E(RW^2),
\label{cw-third-moment-identity}\displaybreak[1]\\
\E W^4
={}&3\E W^2+3\E(W^2\hat L)+\E(RW^3)
-\frac2{\sigma^2}(1+\E\hat L).
\label{cw-fourth-moment-identity}
\end{align}
Taking $f(w)=1$ gives $\E W=\E R$. By \eqref{cw-R-expansion}, \eqref{cw-L-expansion}, and \eqref{cw-fixed-moment-bound}, the three quantities $\E W$, $\E(W\hat L)$ and $\E(RW^2)$ are $O_{\beta,h}(\sigma^{-1})$. Thus \eqref{cw-third-moment-identity} gives $\E W^3=O_{\beta,h}(\sigma^{-1})$. Substitution into the same local expansions yields $\E\hat L=O_{\beta,h}(\sigma^{-2})$ and $\E(RW)=O_{\beta,h}(\sigma^{-2})$, so $\E W^2=1+O_{\beta,h}(\sigma^{-2})$.

For $k=4$, exchangeability gives
\[
\E(\Delta W^2)=\E(\Delta^2W)=2\lambda\E\{(1+\hat L)W\}.
\]
Using the identities for powers of $\Delta$ in \eqref{cw-raw-moment-recursion}, we obtain
\[
 \E W^5=4\E W^3+4\E(\hat LW^3)+\E(RW^4)
       -8\sigma^{-2}\{\E W+\E(\hat LW)\}.
\]
The local expansions and \eqref{cw-fixed-moment-bound} bound $\E(\hat LW^3)$ and $\E(RW^4)$ by $O_{\beta,h}(\sigma^{-1})$, and hence $\E W^5=O_{\beta,h}(\sigma^{-1})$. These odd-moment estimates give $\E(W^2\hat L),\E(RW^3)=O_{\beta,h}(\sigma^{-2})$ and therefore
\[
\E W^4=3+O_{\beta,h}(\sigma^{-2}).
\]

For $k=6$, the difference sum in \eqref{cw-raw-moment-recursion} is $O_{\beta,h}(\sigma^{-1})$, using $|\Delta|\leq2/\sigma$, $\lambda^{-1}=O_{\beta,h}(\sigma^2)$ and \eqref{cw-fixed-moment-bound}. The same order holds for $\E(\hat LW^5)$ and $\E(RW^6)$, so $\E W^7=O_{\beta,h}(\sigma^{-1})$. For $k=5$, exchangeability gives
\[
\E(\Delta W^3)=\frac12\E\{\Delta^2(W^2+WW'+W'^2)\}
=O_{\beta,h}(\lambda).
\]
Together with $\Delta^6=16\sigma^{-4}\Delta^2$, this bounds the difference sum by $O_{\beta,h}(\sigma^{-2})$. The estimates for the fifth and seventh moments also give $\E(\hat LW^4),\E(RW^5)=O_{\beta,h}(\sigma^{-2})$. Thus $\E W^6=5\E W^4+O_{\beta,h}(\sigma^{-2})$. We have proved
\begin{equation}\label{cw-rough-moments}
 \begin{aligned}
 &\E W^{2j+1}=O_{\beta,h}(\sigma^{-1})\quad(0\leq j\leq3),\\
 &\E W^2=1+O_{\beta,h}(\sigma^{-2}),\quad
 \E W^4=3+O_{\beta,h}(\sigma^{-2}),\quad
 \E W^6=15+O_{\beta,h}(\sigma^{-2}).
 \end{aligned}
\end{equation}

By Taylor's expansion, we can write the remainder in \eqref{cw-tanh-expansion} as
\begin{equation}\label{cw-tanh-fourth-remainder}
 \mathcal E_{i,4}
 =\frac{\beta^4m(1-m^2)(2-3m^2)}{3n^4}
   (\sigma W-\sigma_i)^4+\mathcal E_{i,5},
\end{equation}
where $|\mathcal E_{i,5}|\leq C_{\beta,h}n^{-5}(\sigma|W|+1)^5$. Substitution of the fourth-order term on the right-hand side of \eqref{cw-tanh-fourth-remainder} into \eqref{cw-fourth-R-definition} and \eqref{cw-L-basic} yields, respectively,
\begin{align}
&\frac{\beta^4m(1-m^2)(2-3m^2)}
{3\lambda n^5\sigma}
\bigl[(n-4)\sigma^4W^4-4nm\sigma^3W^3
+(6n-4)\sigma^2W^2-4nm\sigma W+n\bigr],
\label{cw-R-quartic-term}\displaybreak[1]\\
&-\frac{\beta^4m(1-m^2)(2-3m^2)}
{3n^5(1-m^2)}
\bigl[\sigma^5W^5+nm\sigma^4W^4
-(4n-6)\sigma^3W^3\notag\\*
&\qquad\qquad
+6nm\sigma^2W^2-(4n-1)\sigma W+nm\bigr].\label{cw-L-quartic-term}
\end{align}
The Taylor remainders $\mathcal E_{i,5}$ in \eqref{cw-tanh-fourth-remainder} contribute the following terms to \eqref{cw-fourth-R-definition} and \eqref{cw-L-basic}, respectively:
\begin{equation}\label{cw-fifth-remainder-definitions}
 \mathcal E_{R,5}=\frac1{\lambda n\sigma}\sum_{i=1}^n\mathcal E_{i,5},\qquad
 \mathcal E_{L,5}=-\frac1{n(1-m^2)}\sum_{i=1}^n\sigma_i\mathcal E_{i,5}.
\end{equation}
The bound following \eqref{cw-tanh-fourth-remainder} gives
\begin{equation}\label{cw-RL-fifth-errors}
|\mathcal E_{R,5}|
\leq C_{\beta,h}n^{-2}(1+|W|^5),\qquad
|\mathcal E_{L,5}|
\leq C_{\beta,h}n^{-5/2}(1+|W|^5).
\end{equation}

We combine the polynomial terms in \eqref{cw-R-expansion} and \eqref{cw-L-expansion} with the fourth-order contributions \eqref{cw-R-quartic-term} and \eqref{cw-L-quartic-term}. The remaining random terms are defined in \eqref{cw-fifth-remainder-definitions}. Multiplying by $W^k$, taking expectations, and using \eqref{cw-normalization}, we obtain
\nopagebreak
\begin{align}
\E(RW^k)
={}&-\frac{\beta m(1-m^2)}
{\sigma\{1-\beta(1-m^2)\}}\E W^k-\frac{\beta^2m(1-m^2)^2}
{\sigma\{1-\beta(1-m^2)\}^2}\E W^{k+2}
\label{cw-R-moment-reduction}\\*
&-\frac{\beta(1-m^2)^2(1-2\beta m^2)}
{\sigma^2\{1-\beta(1-m^2)\}^2}\E W^{k+1}
+\frac{\beta^3(1-m^2)^3(3m^2-1)}
{3\sigma^2\{1-\beta(1-m^2)\}^3}\E W^{k+3}+\varepsilon_{R,k},\notag\displaybreak[1]\\
\E(\hat LW^k)
={}&-\frac{m\{1+\beta(1-m^2)\}}
{\sigma\{1-\beta(1-m^2)\}}\E W^{k+1}+\frac{\beta(1-m^2)}
{\sigma^2\{1-\beta(1-m^2)\}}\E W^k
\label{cw-L-moment-reduction}\\*
&-\frac{\beta(1-m^2)^2(1-\beta m^2)}
{\sigma^2\{1-\beta(1-m^2)\}^2}\E W^{k+2}
+\varepsilon_{L,k}.\notag
\end{align}
Here $k$ is an integer with $0\leq k\leq4$ in \eqref{cw-R-moment-reduction}, while $0\leq k\leq3$ in \eqref{cw-L-moment-reduction}. The errors $\varepsilon_{R,k}$ and $\varepsilon_{L,k}$ are deterministic, unlike the random remainders $\mathcal E_{R,5}$ and $\mathcal E_{L,5}$. They contain the expectations of the omitted polynomial terms, together with $\E(\mathcal E_{R,5}W^k)$ and $\E(\mathcal E_{L,5}W^k)$, respectively. The omitted terms of order $\sigma^{-3}$ are even polynomials in $W$ for $R$ and odd polynomials for $\hat L$. Thus the odd-moment estimates in \eqref{cw-rough-moments} give an extra factor $\sigma^{-1}$ when $k$ is odd for $R$ or even for $\hat L$. Together with \eqref{cw-fixed-moment-bound} and \eqref{cw-RL-fifth-errors}, this gives
\begin{align}
\varepsilon_{R,k}
&=\begin{cases}
O_{\beta,h}(\sigma^{-3}),&k=0,2,4,\\
O_{\beta,h}(\sigma^{-4}),&k=1,3,
\end{cases}
&
\varepsilon_{L,k}
&=\begin{cases}
O_{\beta,h}(\sigma^{-4}),&k=0,2,\\
O_{\beta,h}(\sigma^{-3}),&k=1,3.
\end{cases}
\label{cw-moment-reduction-errors}
\end{align}

Since $\E W=\E R$, \eqref{cw-R-moment-reduction} with $k=0$ and \eqref{cw-moment-reduction-errors} yield the expansion of $\E W$ below.  Substituting \eqref{cw-L-moment-reduction} with $k=1$ and \eqref{cw-R-moment-reduction} with $k=2$ into \eqref{cw-third-moment-identity}, and then using \eqref{cw-moment-reduction-errors}, we obtain the expansion of $\E W^3$:
\begin{align}
\E W
={}&-\frac{\beta m(1-m^2)}
{\sigma\{1-\beta(1-m^2)\}}
-\frac{\beta^2m(1-m^2)^2}
{\sigma\{1-\beta(1-m^2)\}^2}
+O_{\beta,h}(\sigma^{-3})
\notag\\
={}&-\frac{\beta m(1-m^2)}
{\sigma\{1-\beta(1-m^2)\}^2}
+O_{\beta,h}(\sigma^{-3}),
\label{cw-first-moment-refined}\displaybreak[1]\\
\E W^3
={}&2\E W-\frac{2m\{1+\beta(1-m^2)\}}
{\sigma\{1-\beta(1-m^2)\}}
-\frac{\beta m(1-m^2)}
{\sigma\{1-\beta(1-m^2)\}}\notag\\
&-\frac{3\beta^2m(1-m^2)^2}
{\sigma\{1-\beta(1-m^2)\}^2}
+O_{\beta,h}(\sigma^{-3})\label{cw-third-moment-refined}\\
={}&-\frac{m\{2+3\beta(1-m^2)\}}
{\sigma\{1-\beta(1-m^2)\}^2}
+O_{\beta,h}(\sigma^{-3}).\notag
\end{align}
Substituting \eqref{cw-L-moment-reduction} with $k=3$ and \eqref{cw-R-moment-reduction} with $k=4$ into \eqref{cw-raw-moment-recursion} with $k=4$, and then using \eqref{cw-moment-reduction-errors}, we obtain
\begin{align}\label{cw-fifth-moment-refined}
\E W^5
={}&4\E W^3-\frac{12m\{1+\beta(1-m^2)\}}
{\sigma\{1-\beta(1-m^2)\}}
-\frac{3\beta m(1-m^2)}
{\sigma\{1-\beta(1-m^2)\}}
-\frac{15\beta^2m(1-m^2)^2}
{\sigma\{1-\beta(1-m^2)\}^2}
+O_{\beta,h}(\sigma^{-3})
\notag\\
={}&-\frac{5m\{4+3\beta(1-m^2)\}}
{\sigma\{1-\beta(1-m^2)\}^2}
+O_{\beta,h}(\sigma^{-3}).
\end{align}
Substituting \eqref{cw-L-moment-reduction} with $k=0$ and \eqref{cw-R-moment-reduction} with $k=1$ into \eqref{cw-second-moment-identity}, and then using \eqref{cw-first-moment-refined}, \eqref{cw-third-moment-refined}, \eqref{cw-rough-moments}, and \eqref{cw-moment-reduction-errors}, we obtain
\begin{align}
\E W^2
={}&1+\frac{\beta(1-m^2)}
{\sigma^2\{1-\beta(1-m^2)\}^4}
\bigl[3m^2-1+\beta(1-m^2)(1+2m^2)\bigr]
+O_{\beta,h}(\sigma^{-4}).
\label{cw-second-moment-refined}
\end{align}

For $\E(W^2\hat L)$ and $\E(RW^3)$ in \eqref{cw-fourth-moment-identity}, we use \eqref{cw-L-moment-reduction} with $k=2$ and \eqref{cw-R-moment-reduction} with $k=3$, respectively. After substituting \eqref{cw-first-moment-refined}--\eqref{cw-second-moment-refined} and using \eqref{cw-moment-reduction-errors}, \eqref{cw-rough-moments}, and \mbox{$\E W^6=15+O_{\beta,h}(\sigma^{-2})$}, we obtain the fourth-moment expansion through order~$\sigma^{-2}$:
\begin{align}
\E W^4
={}&3+\frac1{\sigma^2}\Biggl[
\frac{3\beta(1-m^2)
\{3m^2-1+\beta(1-m^2)(1+2m^2)\}}
{\{1-\beta(1-m^2)\}^4}
\label{cw-fourth-moment-refined}\\*
&+3\Biggl\{
\frac{m^2\{1+\beta(1-m^2)\}
\{2+3\beta(1-m^2)\}}
{\{1-\beta(1-m^2)\}^3}
+\frac{\beta(1-m^2)}
{1-\beta(1-m^2)}\notag\\*
&\quad-\frac{3\beta(1-m^2)^2(1-\beta m^2)}
{\{1-\beta(1-m^2)\}^2}\Biggr\}\notag\\*
&+\frac{\beta m^2(1-m^2)\{2+3\beta(1-m^2)\}}
{\{1-\beta(1-m^2)\}^3}
+\frac{5\beta^2m^2(1-m^2)^2
\{4+3\beta(1-m^2)\}}
{\{1-\beta(1-m^2)\}^4}\notag\\*
&-\frac{3\beta(1-m^2)^2(1-2\beta m^2)}
{\{1-\beta(1-m^2)\}^2}
+\frac{5\beta^3(1-m^2)^3(3m^2-1)}
{\{1-\beta(1-m^2)\}^3}-2\Biggr]
+O_{\beta,h}(\sigma^{-4})\notag\displaybreak[1]\\
={}&3+\frac{
2(3m^2-1)+4\beta(1-m^2)(8m^2-1)
+6\beta^2(1-m^2)^2(1+2m^2)}
{\sigma^2\{1-\beta(1-m^2)\}^4}+O_{\beta,h}(\sigma^{-4}).\notag
\end{align}
Since $\sigma\asymp_{\beta,h}\sqrt n$, \eqref{cw-first-moment-refined}--\eqref{cw-fourth-moment-refined} establish \eqref{cw-raw-moment-expansions} and the stated mean and variance bounds.

It remains to determine the third and fourth cumulants from \eqref{cw-raw-moment-expansions}.  For the third cumulant,
\begin{align}
\Gamma_3
=\E W^3-3(\E W)(\E W^2)+2(\E W)^3
=\E W^3-3\E W+O_{\beta,h}(n^{-3/2})
=\gamma_{3,n}+O_{\beta,h}(n^{-3/2}).\notag
\end{align}
We use the first four moments to compute the fourth cumulant.
\begin{align}\label{cw-fourth-cumulant-from-raw}
\Gamma_4
&=\E(W-\E W)^4
-3\{\E(W-\E W)^2\}^2\\*
&=\E W^4-4(\E W)(\E W^3)-3(\E W^2)^2
+12(\E W)^2\E W^2-6(\E W)^4.\notag
\end{align}
After substituting \eqref{cw-raw-moment-expansions} into \eqref{cw-fourth-cumulant-from-raw}, the coefficient of $\sigma^{-2}\{1-\beta(1-m^2)\}^{-4}$ is
\begin{align}
&2(3m^2-1)+4\beta(1-m^2)(8m^2-1)
+6\beta^2(1-m^2)^2(1+2m^2)
\label{cw-fourth-cumulant-coefficient}\\*
&\quad-4\beta m^2(1-m^2)\{2+3\beta(1-m^2)\}\notag\\*
&\quad-6\beta(1-m^2)
\{3m^2-1+\beta(1-m^2)(1+2m^2)\}
+12\beta^2m^2(1-m^2)^2\notag\\*
&=2\{3m^2-1+\beta(1-m^2)(3m^2+1)\}.\notag
\end{align}
\eqref{cw-explicit-cumulants}, \eqref{cw-fourth-cumulant-from-raw}, and \eqref{cw-fourth-cumulant-coefficient} therefore imply $\Gamma_4=\gamma_{4,n}+O_{\beta,h}(n^{-2})$, which completes the proof.
\end{proof}

To bound the scalar remainders in Lemma \ref{lemma-cw-fourth-stein}, we first control the tilted moments of $W$. Using the field shift and the mean-field solution $m(H)$ from the proof of Lemma \ref{lemma-cw-moment-expansion}, we set $\tilde m=m\{h+t/(\beta\sigma)\}$. Set $\tilde\sigma^2=n(1-\tilde m^2)/\{1-\beta(1-\tilde m^2)\}$. The fixed-point equation implies, uniformly for $|t|\leq c_0n^{3/10}$,
\begin{equation}\label{cw-mtilde-expansion}
\tilde m-m
=\frac{1-m^2}{1-\beta(1-m^2)}\frac{t}{\sigma}
+O_{\beta,h}\left(\frac{t^2}{\sigma^2}\right),
\end{equation}
and hence
\begin{equation}\label{cw-sigmatilde-ratio}
\frac{\tilde\sigma}{\sigma}
=1+O_{\beta,h}\left(\frac{|t|}{\sigma}\right).
\end{equation}
Since $\sigma\asymp_{\beta,h}\sqrt n$, the field changes by $O_{\beta,h}(n^{-1/5})$. For sufficiently large $n$, it remains in the interval $J$ used in the proof of Lemma \ref{lemma-cw-moment-expansion}. The uniform bound \eqref{cw-fixed-moment-bound} therefore gives $\E_t|(S-n\tilde m)/\tilde\sigma|^r\leq C_{\beta,h,r}$ for $0\leq r\leq6$. As
\[
W=\frac{\tilde\sigma}{\sigma}\frac{S-n\tilde m}{\tilde\sigma}
+\frac{n(\tilde m-m)}{\sigma},
\]
\eqref{cw-mtilde-expansion} and \eqref{cw-sigmatilde-ratio} imply
\begin{equation}\label{cw-tilted-moments}
\E_t|W|^r\leq C_{\beta,h,r}(1+|t|^r),
\qquad 0\leq r\leq6,\quad |t|\leq c_0n^{3/10}.
\end{equation}

The finitely many constants $C_{\beta,h,r}$ are absorbed into $C_{\beta,h}$.  For $k=0,1,2$, \eqref{cw-fourth-local-bounds} and \eqref{cw-tilted-moments} imply
\begin{align}
\E_t\{|W|^k|\mathcal R_{4,1,1}|\}
&\leq C_{\beta,h}\bigl[
n^{-1/2}(1+|t|^k)
+n^{-3/2}(1+|t|^{k+4})\bigr],
\label{cw-fourth-rhoR-tilted}\displaybreak[1]\\
\E_t\{|W|^k|\mathcal R_{4,2,1}|\}
&\leq C_{\beta,h}\bigl[
n^{-1}(1+|t|^k)
+n^{-3/2}(1+|t|^{k+3})\bigr].
\label{cw-fourth-rhoL-tilted}
\end{align}
The exact decompositions \eqref{cw-fourth-R-local}--\eqref{cw-fourth-L-local} further imply
\begin{align}
\E_t\{|W|^k|R|\}
\leq{}&C_{\beta,h}\bigl[
n^{-1/2}(1+|t|^{k+2})+n^{-1}(1+|t|^{k+3})\bigr]
+\E_t\{|W|^k|\mathcal R_{4,1,1}|\},
\label{cw-fourth-R-tilted}\displaybreak[1]\\
\E_t\{|W|^k|\hat L|\}
\leq{}&C_{\beta,h}\bigl[
n^{-1/2}(1+|t|^{k+1})+n^{-1}(1+|t|^{k+2})\bigr]
+\E_t\{|W|^k|\mathcal R_{4,2,1}|\}.
\label{cw-fourth-L-tilted}
\end{align}
In \eqref{cw-fourth-R42-L}--\eqref{cw-fourth-R44-R}, each local remainder is multiplied by a polynomial of degree at most two in $W$.  The $k=0,1,2$ cases of \eqref{cw-fourth-rhoR-tilted}--\eqref{cw-fourth-L-tilted}, together with $|t|\leq cn^{3/10}$, control all terms in \eqref{cw-fourth-R42-L}--\eqref{cw-fourth-R44-R} and imply \eqref{cw-fourth-R41-bound}--\eqref{cw-fourth-R44-bound}.  In $\mathcal R_{4,4}$, the $\chi_{4,n}W$-terms in \eqref{cw-fourth-R44-R0} and \eqref{cw-fourth-R44-R} cancel, while $-\chi_{4,n}R$ is bounded by \eqref{cw-fourth-R-tilted}.  Thus, uniformly for $|t|\leq cn^{3/10}$,
\begin{align*}
|t|\E_t|\mathcal R_{4,1}|
&=O_{\beta,h}(c^5)+o_{\beta,h,c}(1),\qquad
|t|^2\E_t|\mathcal R_{4,2}|
=O_{\beta,h}(c^5)+o_{\beta,h,c}(1),\\
|t|^3\E_t|\mathcal R_{4,3}|
&=O_{\beta,h}(c^5)+o_{\beta,h,c}(1),\qquad
n|t|\E_t|\mathcal R_{4,4}|
=O_{\beta,h,c}(n^{-1/10}).
\end{align*} 
The six kernels in \eqref{cw-fourth-kernel-0}--\eqref{cw-fourth-kernel-5} are supported on $|u|\leq|\Delta|\leq2\sigma^{-1}$. The definitions of $\lambda$, $\sigma$, and $W$ imply $\lambda^{-1}=O_{\beta,h}(n)$, $\sigma^{-1}=O_{\beta,h}(n^{-1/2})$ and $|W|\leq C_{\beta,h}n^{1/2}$, and \eqref{cw-fourth-chi} implies $\chi_{4,n}=O_{\beta,h}(n^{-3/2})$.  Since $\int_{|u|\leq2/\sigma}|H_\Delta(u)|\,du=|\Delta|$ and $0\leq s_\Delta(u)\leq1$ on the support of $H_\Delta(u)$, the conditional total variations of $\mathcal K_{4,1,0}(u),\ldots,\mathcal K_{4,1,5}(u)$ are bounded, respectively, by
\begin{align}
&C_{\beta,h}\lambda^{-1}\sigma^{-5},
\qquad
C_{\beta,h}n^{-1/2}\lambda^{-1}\sigma^{-4},
\qquad
C_{\beta,h}n^{-1}\lambda^{-1}\sigma^{-3},
\label{cw-fourth-kernel-masses}\\
&C_{\beta,h}n^{-1}|W|\lambda^{-1}\sigma^{-4},
\qquad
C_{\beta,h}n^{-3/2}\lambda^{-1}\sigma^{-2},\qquad C_{\beta,h}\lambda^{-1}
\{n^{-1}+n^{-1/2}|W|+n^{-1}W^2\}\sigma^{-5}.\notag
\end{align}
Each quantity in \eqref{cw-fourth-kernel-masses} is at most $C_{\beta,h}n^{-3/2}$.  On the support of the kernels, $e^{|t||u|}\leq e^{2|t|/\sigma}\leq C_{\beta,h}$ for $|t|\leq cn^{3/10}$ and $0<c\leq c_0(\beta,h)$. Consequently, the support condition and \eqref{cw-kernel-bound} hold with $\rho_4=2\sigma^{-1}$ and $\vartheta_4=C_{\beta,h}n^{-3/2}$. This completes the proof.
\end{proof}

\section*{Acknowledgments}
Song-Hao Liu was partially supported by the Fundamental Research Funds for the Central Universities (No. DUT25RC(3)133). Dawei Lu was supported by the National Natural Science Foundation of China (No. 12371154) and the National Key R\&D Program of China (No. 2023YFA1009200).

The authors declare that they have no known competing financial interests or personal relationships that could have appeared to influence the work reported in this paper.

\section*{Declaration of generative AI and AI-assisted technologies in the manuscript preparation process}
During the preparation of this work the authors used ChatGPT in order to improve the readability of the article.  After using this tool, the authors reviewed and edited the content as needed and take full responsibility for the content of the published article.

\end{document}